\documentclass[a4paper,reqno,10pt]{amsart}

\usepackage{a4wide}

\usepackage{hyperref}

\usepackage{amssymb}
\usepackage{amstext}
\usepackage{amsmath}
\usepackage{amscd}
\usepackage{amsthm}
\usepackage{amsfonts}

\usepackage{graphicx}
\usepackage{latexsym}
\usepackage{mathrsfs}

\usepackage{enumitem}
\setlist[enumerate,1]{label={\upshape(\arabic*)}}
\setlist[enumerate,2]{label={\upshape(\alph*)}}

\usepackage{tikz}
\usetikzlibrary{cd,arrows}
\tikzset{black/.style={circle,fill=black,inner sep=3pt,outer sep=3pt},
         white/.style={circle,fill=white,draw=black,inner sep=3pt,outer sep=3pt},
}

\usepackage{array}
\newcolumntype{C}{>{$}c<{$}}

\usepackage{makecell}

\newtheorem{theorem}{Theorem}[section]
\newtheorem{theoremi}{Theorem}
\newtheorem{corollaryi}[theoremi]{Corollary}

\newtheorem{corollary}[theorem]{Corollary}
\newtheorem{lemma}[theorem]{Lemma}
\newtheorem*{lemma*}{Lemma}
\newtheorem*{theorem*}{Theorem}
\newtheorem{proposition}[theorem]{Proposition}
\newtheorem{definition-proposition}[theorem]{Definition-Proposition}
\newtheorem{problem}[theorem]{Problem}

\theoremstyle{definition}
\newtheorem{definition}[theorem]{Definition}
\newtheorem{assumption}[theorem]{Assumption}
\newtheorem{remark}[theorem]{Remark}
\newtheorem{example}[theorem]{Example}
\newtheorem{examplei}[theoremi]{Example}

\renewcommand{\AA}{\mathcal{A}}

\newcommand{\CC}{\mathcal{C}}
\newcommand{\MM}{\mathsf{M}}

\newcommand{\DD}{\mathcal{D}}

\newcommand{\KK}{\mathsf{K}}
\newcommand{\FF}{\mathcal{F}}

\newcommand{\PP}{\mathsf{P}}

\renewcommand{\SS}{\mathcal{S}}
\newcommand{\TT}{\mathcal{T}}

\newcommand{\Z}{\mathbb{Z}}
\renewcommand{\P}{\mathcal{P}}
\newcommand{\I}{\mathcal{I}}

\newcommand{\N}{\mathbb{N}}

\renewcommand{\top}{\operatorname{top}\nolimits}

\newcommand{\rad}{\operatorname{rad}\nolimits}

\newcommand{\Ext}{\operatorname{Ext}\nolimits}

\newcommand{\Hom}{\operatorname{Hom}\nolimits}

\newcommand{\gl}{\operatorname{gl.\!dim}\nolimits}

\newcommand{\op}{\operatorname{op}\nolimits}
\newcommand{\RHom}{\mathbf{R}\strut\kern-.2em\operatorname{Hom}\nolimits}

\newcommand{\Image}{\operatorname{Im}\nolimits}

\newcommand{\Cokernel}{\operatorname{Coker}\nolimits}

\newcommand{\rank}{\operatorname{rank}}

\newcommand{\coker}{\Cokernel}
\newcommand{\im}{\Image}

\newcommand{\un}{\underline}
\newcommand{\ov}{\overline}

\newcommand{\can}{\mathsf{can}}

\newcommand{\ot}{\leftarrow}

\DeclareMathOperator{\moduleCategory}{\mathsf{mod}} \renewcommand{\mod}{\moduleCategory}

\DeclareMathOperator{\Ob}{\mathsf{Ob}}
\DeclareMathOperator{\proj}{\mathsf{proj}}

\DeclareMathOperator{\ind}{\mathsf{ind}}
\DeclareMathOperator{\simp}{\mathsf{sim}}

\DeclareMathOperator{\gp}{\mathsf{gp}}

\DeclareMathOperator{\Iso}{\mathsf{Iso}}
\DeclareMathOperator{\Atom}{\mathsf{Atom}}

\DeclareMathOperator{\Sub}{\mathsf{Sub}}
\DeclareMathOperator{\Fac}{\mathsf{Fac}}

\DeclareMathOperator{\CM}{\mathsf{CM}}
\DeclareMathOperator{\GP}{\mathsf{GP}}

\DeclareMathOperator{\add}{\mathsf{add}}

\DeclareMathOperator{\id}{\mathsf{id}}

\DeclareMathOperator{\inv}{\mathsf{inv}}
\DeclareMathOperator{\Binv}{\mathsf{Binv}}

\DeclareMathOperator{\supp}{\mathsf{supp}}

\DeclareMathOperator{\udim}{\un{\dim}}

\DeclareMathOperator{\WR}{\mathsf{W}_R}

\newcommand{\iso}{\cong}
\newcommand{\infl}{\rightarrowtail}
\newcommand{\defl}{\twoheadrightarrow}

\newcommand{\hc}[1]{\left[ #1 \right)}

\newenvironment{sbmatrix}{\left[\begin{smallmatrix}}{\end{smallmatrix}\right]}

\renewcommand{\AA}{\mathcal{A}}

\newcommand{\EE}{\mathcal{E}}

\numberwithin{equation}{section}

\begin{document}
\title[(JHP) and Grothendieck monoids of exact categories]{The Jordan-H\"older property \\and Grothendieck monoids of exact categories}

\author[H. Enomoto]{Haruhisa Enomoto}
\address{Graduate School of Science, Osaka Prefecture University, 1-1 Gakuen-cho, Naka-ku, Sakai, Osaka 599-8531, Japan}
\email{the35883@osakafu-u.ac.jp}
\begin{abstract}
  We investigate the Jordan-H\"older property (JHP) in exact categories. First, we show that (JHP) holds in an exact category if and only if the Grothendieck monoid introduced by Berenstein and Greenstein is free. Moreover, we give a criterion for this which only uses the Grothendieck group and the number of simple objects.
  Next, we apply these results to the representation theory of artin algebras. For a large class of exact categories including functorially finite torsion(-free) classes, (JHP) holds precisely when the number of indecomposable projectives is equal to that of simples. We study torsion-free classes in a quiver of type A in detail using the combinatorics of symmetric groups. We introduce Bruhat inversions of permutations and show that simples in a torsion-free class are in bijection with Bruhat inversions of the corresponding $c$-sortable element. We use this to give a combinatorial criterion for (JHP).
\end{abstract}

\maketitle

\tableofcontents

\section{Introduction}
To begin with, let us recall the \emph{Jordan-H\"older theorem} for modules. This classical theorem says that the ways in which $M$ can be built up from simple modules are essentially the same.
\begin{theorem*}[the Jordan-H\"older theorem]
  Let $\Lambda$ be a ring and $M$ a $\Lambda$-module of finite length. Then all composition series of $M$ are equivalent. In particular, composition factors of $M$ together with their multiplicities are uniquely determined by $M$.
\end{theorem*}

See for example \cite[Theorem I.1.2]{ARS} for the proof.
Let us express this situation as follows: the category of $\Lambda$-modules of finite length satisfies \emph{the Jordan-H\"older property}, abbreviated by \emph{(JHP)}.  The aim of this paper is to investigate to what extent this property is valid in various settings, especially those arising in the representation theory of algebras. It is known that an abelian category satisfies (JHP) if every object has finite length (e.g. \cite[p.92]{st}), so the abelian case is rather trivial. In this paper, we study (JHP) in the context of Quillen's \emph{exact categories}, which generalize abelian categories and serve as a useful categorical framework for studying various subcategories of module categories.

As in the case of module categories, we can define a \emph{poset of admissible subobjects} of an object of an exact category, and classical notions like simple objects, composition series and (JHP) make sense in this setting (see Section 2).
Typical examples of exact categories are extension-closed subcategories of $\mod\Lambda$ for an artin algebra $\Lambda$, and in this case, all objects have \emph{at least one} composition series. However, it turns out that there exist many categories which do not satisfy (JHP), as well as those which do. Here are some examples (we refer to Section \ref{jhpex} for more examples).
\begin{examplei}[{c.f. \cite[Example 6.9]{bhlr}}]\label{exa}
  Let $k$ be a field and define $\EE$ as the category of finite-dimensional $k$-vector spaces whose dimensions are not equal to $1$. Then obviously $\EE$ is closed under extensions in $\mod k$, thus an exact category, and $k^2$ and $k^3$ are \emph{simple} objects in $\EE$. However, we have the following two different decompositions of $k^6$ into simples:
  \[
  k^3 \oplus k^3 = k^6 = k^2 \oplus k^2 \oplus k^2.
  \]
  Thus, composition series (and lengths) are not unique.
\end{examplei}
The next examples are \emph{torsion-free classes} of a module category, which have been studied in the representation theory of algebras. See Section \ref{typeasec} for the detailed treatment of this example.
\begin{examplei}[c.f. Example \ref{main1ex} (1)]\label{introex}
  Let $Q$ be a quiver $1 \to 2 \ot 3$ and consider the path algebra $kQ$. Then the Auslander-Reiten quiver of $\mod kQ$ is given as follows:
  
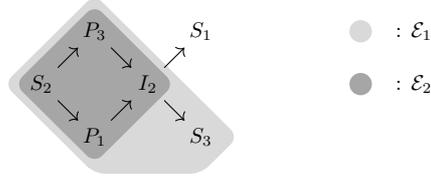
\[
  \begin{tikzpicture}
    [baseline={([yshift=-.5ex]current bounding box.center)}, scale=0.7,  every node/.style={scale=0.8}]

    \fill[gray!30, rounded corners] (-.7,1) -- (1,-.7) -- (3,-.7) -- (3.7,0) -- (1,2.7) -- cycle;
    \fill[gray!70, rounded corners] (-.5,1) -- (1,-.5) -- (2.5,1) -- (1,2.5) -- cycle;

    \node (12) at (3,2) {$S_1$};
    \node (13) at (1,0) {$P_1$};
    \node (14) at (2,1) {$I_2$};
    \node (23) at (0,1) {$S_2$};
    \node (24) at (1,2) {$P_3$};
    \node (34) at (3,0) {$S_3$};

    \draw[->] (23) -- (24);
    \draw[->] (23) -- (13);
    \draw[->] (24) -- (14);
    \draw[->] (13) -- (14);
    \draw[->] (14) -- (12);
    \draw[->] (14) -- (34);

    \fill[gray!30] (6,2) circle (6pt);
    \fill[gray!70] (6,1) circle (6pt);
    \node at (7,2) {: $\EE_1$};
    \node at (7,1) {: $\EE_2$};
  \end{tikzpicture}
  \]
  Consider the subcategory $\EE_1 := \add \{S_2,P_1,P_3,I_2,S_3\}$ of $\mod\Lambda$. The objects $S_2, P_1, S_3$ are simples of $\EE_1$, and $\EE_1$ satisfies (JHP). For example, $P_3$ decomposes into $S_2$ and $S_3$, and $I_2$ decomposes into $P_1$ and $S_3$.

  On the other hand, consider $\EE_2:= \add \{S_2,P_1,P_3,I_2\}$. Then (JHP) fails in $\EE_2$: in fact, all four indecomposables are simple in $\EE_2$, but we have two exact sequences
  \begin{align*}
    &0 \to P_1 \to P_1 \oplus P_3 \to P_3 \to 0 \quad  \text{(a split sequence),} \\
    \text{and }\quad&0 \to S_2 \to P_1 \oplus P_3 \to I_2 \to 0 \quad \text{(an almost split sequence),}
  \end{align*}
  which implies that $P_1 \oplus P_3$ has two different decomposition: one into $P_1$ and $P_3$ and the other into $S_2$ and $I_2$. Thus, \emph{composition factors} of $P_1 \oplus P_3$ are not unique.
\end{examplei}

To explore (JHP), we make use of a lesser-known invariant of exact categories, \emph{Grothendieck monoids}, which were introduced by Berenstein and Greenstein in \cite{bg}. First, let us make some observation on the Grothendieck \emph{group} $\KK_0(\EE)$ of $\EE$, which is more famous. If $\EE$ satisfies (JHP), then we can easily check that $\KK_0(\EE)$ is a free abelian group, since simple objects form a free basis of $\KK_0(\EE)$ by (JHP). Thus, (JHP) implies the freeness of $\KK_0(\EE)$.
However, the converse is not true: $\KK_0(\EE_i) \iso \Z^3$ holds for $i=1,2$ in Example \ref{introex} but (JHP) fails in $\EE_2$. Therefore, we need a \emph{more sophisticated invariant} of exact categories than Grothendieck groups.
To this purpose, we propose to study a \emph{Grothendieck monoid $\MM(\EE)$} of an exact category $\EE$, which is a monoid defined by the same universal property as the Grothendieck group. In the representation theory of algebras, this monoid is closely related to the monoid of dimension vectors of modules (see Corollary \ref{setofdim}).

By using $\MM(\EE)$, we obtain the following simple characterization of (JHP):
\begin{theoremi}[= Theorem \ref{JHchar1.5}]\label{introthm1}
  Let $\EE$ be a skeletally small exact category. Then the following are equivalent:
  \begin{enumerate}
    \item $\EE$ satisfies (JHP).
    \item $\MM(\EE)$ is a free monoid.
    \item $\KK_0(\EE)$ is a free abelian group, and the images of non-isomorphic simple objects in $\KK_0(\EE)$ form a basis of $\KK_0(\EE)$.
  \end{enumerate}
\end{theoremi}

By using this theorem, we obtain the following easier criterion for (JHP) which can be useful in actual situations.
\begin{theoremi}[= Theorem \ref{JHchar2}]\label{introthm2}
  Let $\EE$ be a skeletally small exact category. Suppose that $\KK_0(\EE)$ is finitely generated. Then the following are equivalent:
  \begin{enumerate}
    \item $\EE$ satisfies (JHP).
    \item $\KK_0(\EE)$ is a free abelian group, and $\rank \KK_0(\EE)$ is equal to the number of non-isomorphic simple objects in $\EE$.
  \end{enumerate}
\end{theoremi}

We apply these results to the representation theory of artin algebras. In many cases, the Grothendieck group turns out to be free of finite rank, whose rank is equal to the number of non-isomorphic indecomposable projective objects (Proposition \ref{assumprop}). Thus, Theorem \ref{introthm2} tells us that all we have to do to check (JHP) is to count the number of simple objects in $\EE$, and then to compare it to that of projectives.
In particular, we can give a criterion whether a functorially finite torsion(-free) class satisfies (JHP) using the language of $\tau$-tilting theory (Corollary \ref{tautilting}).
As an easy application of this theorem, we show the following result on Nakayama algebras.
\begin{corollaryi}[= Corollary \ref{nakayamamain}]
  Let $\Lambda$ be a Nakayama algebra. Then every torsion-free class of $\mod\Lambda$ satisfies (JHP).
\end{corollaryi}

Finally, we discuss simple objects in torsion-free classes of the category $\mod kQ$ for a type $A_n$ quiver $Q$. Torsion-free classes of $\mod kQ$ are classified in \cite{it,airt,thomas}: they are in bijection with \emph{$c_Q$-sortable elements} in the symmetric group $S_{n+1}$.
Let $\FF(w)$ be the torsion-free class of $\mod kQ$ corresponding to a $c_Q$-sortable element $w$. Then it is natural to expect that we can describe simples in $\FF(w)$ and the validity of (JHP) by using the combinatorics of $S_{n+1}$. We obtain the following result along these lines.
\begin{theoremi}[= Theorem \ref{binvsimp}, Corollary \ref{jhtypea}]\label{introthm3}
  Let $Q$ be a quiver of type $A_n$ and $w\in S_{n+1}$ a $c_Q$-sortable element. Consider the corresponding torsion-free class $\FF(w)$ of $\mod kQ$. Then we have the following:
  \begin{enumerate}
    \item Simple objects in $\FF(w)$ are in bijection with Bruhat inversions of $w$.
    \item $\FF(w)$ satisfies (JHP) if and only if the number of supports of $w$ is equal to that of Bruhat inversions of $w$.
  \end{enumerate}
\end{theoremi}
Here \emph{Bruhat inversions} are inversions of $w$ which give rise to covering relations in the Bruhat order of $S_{n+1}$ (see Definition \ref{binvdef}). In a forthcoming paper \cite{enforth}, it will be shown that this holds for other Dynkin types (see Remark \ref{forthcoming}).

\subsection*{Relation to other works}
Let us clarify the relation and differences between this paper and other related papers.
\begin{itemize}
  \item In \cite{bg}, the notions of Grothendieck monoids and simple objects were originally introduced and studied in order to investigate Hall algebras of exact categories. Actually several statements about Grothendieck monoids in Section 3.1 are found in \cite{bg}: Propositions \ref{gromon}, \ref{explicit}, \ref{mreduced} and \ref{simpleatom} (see these propositions for the precise references). Also a particular case of Proposition \ref{compex} was considered in \cite[Section 3.4]{bg}.
  However, Section 3.1 is about elementary properties of Grothendieck monoids, hence it is not the main part of this paper. All the other results are not contained in \cite{bg}.
  \item In \cite{bhlr}, the notions of simple objects, composition series and lengths of objects and the Jordan-H\"older property for exact categories were introduced. In particular, they used lengths in detail to consider the Gabriel-Roiter measure on exact categories. Except for these definitions of these concepts and Example \ref{exa}, there is no overlap between this paper and \cite{bhlr}.
  \item In \cite{hr}, the notion of exact categories satisfying \emph{admissible intersection property} was defined, and it was shown that these categories satisfy (JHP). There is no overlap between this paper and \cite{hr} except for terminologies.
  \item After the first version of this paper, in \cite{bht}, some class of exact categories satisfying (JHP) were studied, namely, \emph{diamond exact categories} and \emph{Artin-Wedderburn exact categories}.
  Proposition \ref{interval} appears in \cite[Proposition 3.8]{bht} under the name of \emph{the fourth $\EE$-isomorphism theorem}.
\end{itemize}

\subsection*{Organization}
This paper is divided into two parts and an appendix.
In Part 1, we develop the general theory of posets of admissible subobjects, (JHP) and Grothendieck monoids of exact categories.
In Part 2, we apply results in Part 1 to the representation theory of artin algebras, and give some examples of computations.
In Appendix A, we collect basic properties on monoids which we need in the body part of the paper.

The content of each section is as follows.
In Section 2, we study posets of admissible subobjects of objects in exact categories, and give several basic definitions which we use throughout this paper, such as (JHP).
In the latter part of Section 2, we give some categorical conditions which ensure that subobject posets are (modular) lattices.
In Section 3, we define the Grothendieck monoid of an exact category, and study its basic properties as a monoid.
In Section 4, we give characterizations of (JHP)  in terms of Grothendieck monoids or groups (Theorems \ref{introthm1} and \ref{introthm2}).

In Section 5, we apply the general results to the representation theory of artin algebras. We mainly consider \emph{good} extension-closed subcategories of module categories (see Assumption \ref{assumperp}), and study (JHP) for this class of categories.
In Section 6, we consider torsion-free classes of $\mod kQ$ where $Q$ is a quiver of type A and prove Theorem \ref{introthm3}.
In Section 7, we give examples of computations of the Grothendieck monoids.
In Section 8, we collect counter-examples on various conditions which we have investigated.
In Section 9, some open problems are discussed.

\subsection*{Conventions and notation}
Throughout this paper, \emph{we assume that all categories are skeletally small}, that is, the isomorphism classes of objects form a set. In addition, \emph{all subcategories are assumed to be full and closed under isomorphisms}. For a category $\EE$, we denote by $\Iso \EE$ the set of all isomorphism classes of objects in $\EE$. For an object $X$ in $\EE$, the isomorphism class of $X$ is denoted by $[X] \in \Iso\EE$.

For a set of object $\CC$ of an additive category $\EE$, we denote by $\add\CC$ the subcategory of $\EE$ consisting of direct summands of finite direct sums of objects in $\CC$.

For a Krull-Schmidt category $\EE$, we denote by $\ind\EE$ the set of isomorphism classes of indecomposable objects in $\EE$. We denote by $|X|$ the number of non-isomorphic indecomposable direct summands of $X$.

By a \emph{module} we always mean right modules unless otherwise stated.
For a noetherian ring $\Lambda$, we denote by $\mod\Lambda$ (resp. by $\proj\Lambda$) the category of finitely generated right $\Lambda$-modules (resp. finitely generated projective right $\Lambda$-modules).

As for exact categories, we use the terminologies \emph{inflations}, \emph{deflations} and \emph{conflations}. We refer the reader to \cite{buhler} for the basics of exact categories.
For an inflation $A \infl X$ in an exact category, we denote by $X \defl X/A$ the cokernel of $A \infl X$.
We say that an exact category $\EE$ \emph{has a progenerator $P$} (resp. \emph{an injective cogenerator $I$}) if $P$ is projective (resp. $I$ is injective) and every object in $\EE$ admits a deflation from a finite direct sum of $P$ (resp. an inflation into a finite direct sum of $I$).
For an exact category $\EE$, we denote by $\P(\EE)$ (resp. $\I(\EE)$) the category consisting of all projective (resp. injective) objects in $\EE$.

Let $\AA$ be an abelian category and $\EE$ a subcategory of $\AA$. We say that $\EE$ is \emph{extension-closed} if for every short exact sequence $0 \to X \to Y \to Z \to 0$ in $\AA$, we have that $Y$ belongs to $\EE$ whenever both $X$ and $Z$ belong to $\EE$. In such a case, unless otherwise stated, we regard $\EE$ as an exact category with the induced exact structure; conflations are precisely short exact sequences of $\AA$ whose terms all belong to $\EE$.

For a poset $P$ and two elements $a,b \in P$ with $a \leq b$, we denote by $[a,b]$ the \emph{interval poset} $[a,b]:=\{ x \in P \, | \, a \leq x \leq b \}$ with the obvious partial order.

By a \emph{monoid} $M$ we mean a \emph{commutative} semigroup with a unit, and we always use an additive notation: the operation is denoted by $+$, and the unit of addition is always denoted by $0$. A \emph{homomorphism} between monoids is a map which preserves the addition and $0$.

We denote by $\N$ the monoid of non-negative integers: $\N = \{ 0, 1, 2, \cdots \}$ with the addition $+$.

For a set $A$, we denote by $\# A$ the cardinality of $A$.

\part{General theory}
In this part, we develop the general theory of posets of admissible subobjects, (JHP) and Grothendieck monoids of exact categories.
\section{Posets of admissible subobjects}
To study the Jordan-H\"older property (JHP) on exact categories, one has to define what (JHP) exactly means, which we will do in this section. The contents of this section are natural generalizations of the corresponding ones in module categories or abelian categories.
\subsection{Basic properties}\label{posetdef}
First, we collect some basic definitions on the poset of admissible subobjects, which we use throughout this paper.
\begin{definition}
  Let $\EE$ be a skeletally small exact category and $X$ an object of $\EE$.
  \begin{itemize}
    \item We call an inflation $A \infl X$ an \emph{admissible subobject of $X$} (c.f. \cite[Definition 3.1]{bhlr}). We often call $A$ an admissible subobject of $X$ in this case.
    \item Two subobjects $A, B$ of $X$ are called \emph{equivalent} if there exists an isomorphism $A \xrightarrow{\sim} B$ which makes the following diagram commutes:
    \[
    \begin{tikzcd}
      A \rar[rightarrowtail] \dar["\rotatebox{90}{$\sim$}"] & X \\
      B \ar[ru, rightarrowtail]
    \end{tikzcd}
    \]
    We denote by $\PP(X)$ the equivalence class of admissible subobjects of $X$. When we want to emphasize the ambient category $\EE$, we write $\PP_\EE(X)$.
    \item We define a partial order on $\PP(X)$ as follows: We write $A \leq B$ for $A, B \in \PP(X)$ if there exists an \emph{inflation} $A \infl B$ which makes the following diagram commutes:
    \[
    \begin{tikzcd}
      A \rar[rightarrowtail] \dar[rightarrowtail] & X \\
      B \ar[ru, rightarrowtail]
    \end{tikzcd}
    \]
    It is easily checked that this relation actually gives a partial order on $\PP(X)$.
  \end{itemize}
\end{definition}
We remark that $\PP(X)$ always has the greatest element $X$ and the smallest element $0$.

\begin{example}
  Let $\AA$ be an abelian category and $\EE$ an extension-closed subcategory of $\AA$. Then for an object $X \in \EE$, we have that
  \[
  \PP_\EE(X) = \{ A \,|\, A \text{ is a subobject of $X$ in $\AA$ which satisfies $A, X/A \in \EE$} \}.
  \]
\end{example}

Although we assume that there exists an \emph{inflation} $A \infl B$ in order to have $A \leq B$, if $\EE$ is \emph{weakly idempotent complete}, then this condition can be weakened. We refer to \cite[Section 7]{buhler} for weakly idempotent completeness.
\begin{lemma}\label{wic}
  Let $\EE$ be a weakly idempotent complete exact category. Then for an object $X$ and $A, B \in \PP(X)$ of $X$, the following are equivalent:
  \begin{enumerate}
    \item $A \leq B$ holds in $\PP(X)$.
    \item There exists a morphism $\varphi \colon A \to B$ which makes the following diagram commute:
    \[
    \begin{tikzcd}
      A \rar[rightarrowtail] \dar["\varphi"'] & X \\
      B \ar[ru, rightarrowtail, "\iota"']
    \end{tikzcd}
    \]
  \end{enumerate}
\end{lemma}
\begin{proof}
  (1) $\Rightarrow$ (2): Obvious.

  (2) $\Rightarrow$ (1):
  It is enough to check that $\varphi$ is an inflation. This follows from that $\iota \varphi$ is an inflation and $\EE$ is weakly idempotent complete, see \cite[Proposition 7.6]{buhler} for example.
\end{proof}
\begin{example}
  Let $\AA$ be an abelian category and $\EE$ an extension-closed subcategory of $\AA$. Suppose that $\EE$ is \emph{closed under direct summands}, that is, if $A \oplus B$ belongs to $\EE$ then so do $A$ and $B$. Then $\EE$ is idempotent complete, thus weakly idempotent complete. Many important exact categories investigated in the representation theory of algebras arise in this way.
\end{example}

The next proposition ensures that for any interval $[A,B]$ in $\PP(X)$, we have an isomorphism of posets $[A,B] \iso \PP(B/A)$, as in the case of abelian categories.
\begin{proposition}\label{interval}
  Let $\EE$ be a skeletally small exact category and $X$ an object of $\EE$. Then the following hold, where all the intervals are considered in $\PP(X)$.
  \begin{enumerate}
    \item For $A \in \PP(X)$, we have an isomorphism of posets $[0,A] \iso \PP(A)$.
    \item For $A \in \PP(X)$, we have an isomorphism of posets
    $(-)/A \colon [A,X] \iso \PP(X/A)$.
    Moreover, $X/B \iso (X/A)/(B/A)$ holds for any $B$ in $[A,X]$.
    \item () For $A, B \in \PP(X)$ with $A \leq B$, we have an isomorphism of posets $(-)/A\colon  \allowbreak[A,B] \iso \PP(B/A)$.
    Moreover, for any $X_1,X_2$ in $\PP(X)$ with $A \leq X_1 \leq X_2 \leq B$, we have that $X_2/X_1 \iso (X_2/A)/(X_1/A)$.
  \end{enumerate}
\end{proposition}
\begin{proof}
  (1)
  For an element $B \in [0,A]$, there exists an inflation $B \infl A$ since $B \leq A$ holds. Thus, $B \in \PP(A)$ holds.
  On the other hand, let $B \in \PP(A)$. Then we have an inflation $B \infl A$. Since inflations are closed under compositions, it follows that the composition $B \infl A \infl X$ is an inflation, hence $B$ is an admissible subobject of $X$. Consequently we have $B \in \PP(X)$, and clearly $B \in [0,A]$ holds. These assignments are easily shown to be mutually inverse isomorphisms of posets between $[0,A]$ and $\PP(A)$.

  (2)
  For an element $B \in [A,X]$, we have inflations $A \infl B \infl X$. Then we obtain the following commutative diagram
  \[
  \begin{tikzcd}
    A \dar[equal] \rar[rightarrowtail] & B \rar[twoheadrightarrow] \dar[rightarrowtail] & B/A \dar[rightarrowtail] \\
    A \rar[rightarrowtail] & X \rar[twoheadrightarrow] \dar[twoheadrightarrow] & X/A \dar[twoheadrightarrow] \\
    & X/B \rar[equal] & X/B
  \end{tikzcd}
  \]
  in $\EE$ (see \cite[Lemma 3.5]{buhler}).
  Thus, the assignment $B \mapsto B/A$ gives a morphism of poset $(-)/A \colon [A,X] \to \PP(X/A)$. Moreover, by the above sequence, $X/B \iso (X/A)/(B/A)$ holds.

  Conversely, let $M \infl X/A$ be an admissible subobject of $X/A$. By taking the pullback of $X \defl X/A$ along $M \infl X/A$, we obtain the following diagram
  \[
  \begin{tikzcd}
    A \dar[equal]\rar[rightarrowtail] & B \rar[twoheadrightarrow] \dar[rightarrowtail] & M \dar[rightarrowtail] \\
    A \rar[rightarrowtail] & X \rar[twoheadrightarrow] \dar[twoheadrightarrow] & X/A \dar[twoheadrightarrow] \\
    & X/B \rar[equal] & X/B
  \end{tikzcd}
  \]
  in $\EE$ where all rows and columns are conflations (see \cite[Proposition 2.15]{buhler}).
  Since $A \leq B$ holds in $\PP(X)$, the assignment $M \mapsto B$ gives a morphism of poset $\PP(X/A) \to [A,X]$. These maps are easily seen to be mutually inverse to each other.

  (3)
  This follows from (1) and (2), since we have the following chain of isomorphisms of posets
  \[
  \begin{tikzcd}
    & & \PP(A) \\
    & \PP(B) \rar["\iso"] & \left[0,B\right] \ar[u,hookrightarrow] \\
    \PP(B/A) \rar["\iso"] & \left[A,B\right] \ar[u,hookrightarrow] \rar["\iso"] & \left[A,B\right] \ar[u,hookrightarrow]
  \end{tikzcd}
  \]
  where the middle interval is considered in $\PP(B)$ and the two right intervals are considered in $\PP(A)$.

  The isomorphism $X_2/X_1 \iso (X_1/A)/(X_2/A)$ is obtained by applying (2) to $A \in \PP(X_2)$.
\end{proof}

Now we introduce \emph{simple} objects in exact categories, that is, objects which cannot be decomposed into smaller pieces with respect to conflations. These objects play a central role throughout this paper, and to determine simple objects in a given exact category is an essential (and difficult) task when we check (JHP).
\begin{definition}
  Let $\EE$ be an exact category and $X$ an object of $\EE$. We say that $X$ is \emph{simple} if the poset $\PP(X)$ of admissible subobjects of $X$ has exactly two distinct elements $0$ and $X$. This is equivalent to the condition that $X$ is not zero and there exists no conflation of the form $L \infl X \defl N$ in $\EE$ with $L, N \neq 0$. We denote by $\simp \EE$ the set of isomorphism classes of simple objects in $\EE$.
\end{definition}
We remark that the notion of simple objects in exact categories was also introduced in \cite[Definition 3.2]{bhlr} (under the name of \emph{$\EE$-simple} objects) and \cite[Definition 5.2]{bg} (under the name of \emph{almost simple} objects).
\subsection{Basic definitions on inflation series, composition series and (JHP)}\label{basicdef}
In module categories, \emph{submodule series} (a chain of submodules) serves as a basic tool when we consider composition series and the Jordan-H\"older theorem. It is natural to introduce the corresponding notion, \emph{inflation series}, in the context of exact categories. Here are the definitions of it and related notions, including (JHP).

First, recall some notions from poset theory.
A \emph{chain} of a poset $P$ is a totally ordered subset of $P$.
A chain $T$ of $P$ is called \emph{maximal} if there exists no chain of $P$ which properly contains $T$.
A chain $T$ is called \emph{finite} if $T$ is a finite set. Such a chain is of the form $x_0 < x_1 < \cdots < x_n$, and we say that this chain \emph{has length $n$} in this case.

Let $\EE$ be a skeletally small exact category.
\begin{itemize}
  \item For $X$ in $\EE$, an \emph{inflation series} of $X$ is a finite sequence of inflations $0 = X_0 \infl X_1 \infl \cdots \infl X_n = X$.
  We often identify it with a weakly increasing sequence $0 = X_0 \leq X_1 \leq \cdots \leq X_n = X$ in $\PP(X)$.
  \item We say that an inflation series $0 = X_0 \infl X_1 \infl \cdots \infl X_n = X$ of $X$ is a \emph{proper inflation series} if none of $X_i \infl X_{i+1}$ is an isomorphism, or equivalently, $X_0 < X_1 < \cdots < X_n$ holds in $\PP(X)$. In this case, we say that this inflation series has \emph{length} $n$. We often identify proper inflation series with finite chains of $\PP(X)$ which contain $0$ and $X$.
  \item Let $X$ and $Y$ be objects in $\EE$, and let $0 = X_0 \infl X_1 \infl \cdots \infl X_n = X$ and $0 = Y_0 \infl Y_1 \infl \cdots \infl Y_m = Y$ be two inflation series of $X$ and $Y$ respectively.
  Then these inflation series are called \emph{isomorphic} if $m = n$ and there exists a permutation $\sigma$ of the set $\{1,2,\cdots n\}$ such that $X_i / X_{i-1}$ and $ Y_{\sigma(i)} / Y_{\sigma(i) -1}$ are isomorphic in $\EE$ for all $i$.
  In this case, we say that $X$ and $Y$ \emph{have isomorphic inflation series}.
  \item An inflation series $0 = X_0 \infl X_1 \infl \cdots \infl X_n = X$ of $X$ is called a \emph{composition series} if $X_0 < X_1 < \cdots < X_n$ is a maximal chain in $\PP(X)$, or equivalently, each quotient $X_i/X_{i-1}$ is a simple object for all $i$ (this follows from Proposition \ref{interval}), c.f. \cite[Section 5.3]{bg}, \cite[Definition 6.1]{hr}. We often identify  composition series with finite maximal chains of $\PP(X)$.
  \item $\EE$ is called a \emph{length exact category} if for every object $X$ in $\EE$, the set of lengths of proper inflation series of $X$ has an upper bound, or equivalently, the set of lengths of finite chains of $\PP(X)$ has an upper bound.
  \item A length exact category $\EE$ satisfies the \emph{unique length property} if all composition series of $X$ have the same length for every object $X$ in $\EE$.
  \item A length exact category $\EE$ satisfies the \emph{Jordan-H\"older property}, abbreviated by \emph{(JHP)}, if all composition series of $X$ are isomorphic to each other for every object $X$ in $\EE$.
\end{itemize}

First, we prove some properties of length exact categories which easily follow from definitions. Actually the proof only uses the general theory of posets.
\begin{proposition}
  For a length exact category $\EE$ and an object $X$ in $\EE$, the following holds:
  \begin{enumerate}
    \item Every chain of $\PP(X)$ is finite.
    \item The composition series of $X$ are precisely the maximal chains of $\PP(X)$.
    \item Every proper inflation series of $X$ can be refined to a composition series.
    \item $X$ has at least one composition series.
  \end{enumerate}
\end{proposition}
\begin{proof}
  (1)
  Suppose that $\PP(X)$ has a chain $T$ consisting of infinitely many elements. Then by choosing a finite subset of $T$, we can obtain a finite chain of $\PP(X)$ with an arbitrary large length, which is a contradiction.

  (2)
  This follows from (1) and the definition of compositions series.

  (3)
  Suppose that a proper chain $0 = X_0 < X_1 < \cdots < X_n = X$ of $\PP(X)$ is given. If this chain is maximal, then we have nothing to do. Suppose that this is not the case.
  If $[X_{i-1}, X_i] = \{ X_{i-1}, X_i \}$ for every $i$, then this chain is clearly maximal. Thus, there exist some $i$ and $Y \in \PP(X)$ with $X_{i-1} < Y < X_i$.
  Consequently, we obtain a chain $X_0 < X_1 < \cdots < X_{i-1} < Y < X_i < \cdots < X_n$ of $\PP(X)$ with length $n+1$, which is a refinement of the original chain. By iterating this process, we will eventually obtain a finite maximal chain, since the set of lengths of finite chains of $\PP(X)$ has an upper bound.

  (4)
  We may assume that $X \neq 0$. Then (4) follows by applying (3) to a chain $0 < X$ of $\PP(X)$.
\end{proof}
\begin{remark}\label{lengthremark}
  Let $\EE$ be an exact category and suppose that for every $X\in\EE$ there exists at least one composition series of $X$. Even so, $\EE$ may not be a length exact category, since lengths of finite chains of $\PP(X)$ may be arbitrary large (unfortunately the author does not know such an example).
  However, if $\EE$ is an abelian category (or more generally, an exact category such that every subobject poset is a modular lattice), then $\EE$ is a length exact category. This follows from Theorem \ref{latticejh}.
\end{remark}

\begin{remark}
  In a recent paper \cite{bht}, two classes of exact categories are introduced and are shown to satisfy (JHP); \emph{diamond exact categories} and Krull-Schmidt \emph{Artin-Wedderburn categories}.
\end{remark}
By definition, (JHP) implies the unique length property for length exact categories. For concrete (counter-)examples concerning (JHP) and the unique length property, we refer the reader to Section \ref{jhpex}.

In the remaining part of this section, we give a sufficient condition for the unique length property by using the theory of modular lattices.

\subsection{Quasi-abelian implies lattice}
Recall that a \emph{lattice} $L$ is a poset such that for every two elements $a,b \in L$, there exist the \emph{meet} $a \wedge b$, the greatest lower bound, and the \emph{join} $a \vee b$, the least upper bound.
In abelian categories, it is classical that the subobject poset $\PP(X)$ is always a lattice by considering sum and intersection of subobjects. However, this does not hold in general, \emph{even if we consider pre-abelian categories} (see Examples \ref{nonlattice1} and \ref{nonlattice2}).
In this subsection, we will show that $\PP(X)$ is a lattice when we consider torsion(-free) classes of abelian categories, or more generally, \emph{quasi-abelian} categories.

First, we recall the definition of quasi-abelian categories.
\begin{definition}
  Let $\EE$ be an additive category.
  \begin{enumerate}
    \item $\EE$ is \emph{pre-abelian} if every morphism in $\EE$ has a kernel and a cokernel. It is well-known that pre-abelian category has all pullbacks and pushouts.
    \item A pre-abelian category $\EE$ is \emph{quasi-abelian} if the class of kernel morphisms is closed under pushouts and the class of cokernel morphisms is closed under pullbacks.
  \end{enumerate}
\end{definition}
Quasi-abelian categories in the above sense are sometimes called \emph{almost abelian} \cite{rumpstar,rump}, or \emph{semi-abelian} \cite{raikov}. See Historical remark in \cite[Section 2]{rumpraikov} for the somewhat involved history of quasi-abelian categories.

Every quasi-abelian category has the natural greatest exact structure, and we always consider this greatest exact structure in what follows. More precisely, the following proposition holds. We refer to \cite[Proposition 4.4]{buhler} for the proof.
\begin{proposition}
  Let $\EE$ be a quasi-abelian category. Then $\EE$ has the greatest exact structure, in which conflations, inflations and deflations are precisely kernel-cokernel pairs, kernel morphisms and cokernel morphisms respectively.
\end{proposition}
By this proposition, for an object $X$ in a quasi-abelian category $\EE$, an inflation $A \infl X$ is nothing but a kernel morphism in $\EE$.
Moreover, for two admissible subobjects $\iota_A \colon A \infl X$ and $\iota_B \colon B \infl X$, we have that $A \leq B$ holds in $\PP(X)$ if and only if $\iota_A$ factors through $\iota_B$. This follows from Proposition \ref{wic} since every pre-abelian category is idempotent complete.

A typical example of quasi-abelian categories are torsion(-free) classes of abelian categories.
\begin{definition}\label{tfdef}
  Let $\AA$ be an abelian category.
  \begin{enumerate}
    \item Let $\TT$ and $\FF$ be subcategories of $\AA$. We say that a pair $(\TT,\FF)$ is a \emph{torsion pair} in $\AA$ if it satisfies the following conditions:
    \begin{enumerate}
      \item $\AA(\TT,\FF)=0$ holds, that is, $\AA(T,F)=0$ holds for every $T \in \TT$ and $F \in \FF$.
      \item For every object $X \in \AA$, there exists a short exact sequences
      \[
      0 \to TX \to X \to FX \to 0
      \]
      in $\AA$ with $TX \in \TT$ and $FX \in \FF$.
    \end{enumerate}
    In this case, we say that $\TT$ is a torsion class and $\FF$ is a torsion-free class.
    \item We say that a torsion pair $(\TT,\FF)$ is \emph{hereditary} if $\TT$ is closed under subobjects in $\AA$. In this case, we say that $\FF$ is a \emph{hereditary torsion-free class}.
  \end{enumerate}
\end{definition}
For a torsion pair $(\TT,\FF)$ on an abelian category, it is easily checked that $\TT$ is closed under quotients and extensions, and that $\FF$ is closed under subobjects and extensions. Hence we can regard $\TT$ and $\FF$ as exact categories.

The next proposition is classical, e.g. \cite[Theorem 2]{rump}. For the convenience of the reader, we give a simple proof which makes use of the exact structure.
\begin{proposition}
  Let $(\TT,\FF)$ be a torsion pair in an abelian category $\AA$. Then $\TT$ and $\FF$ are both quasi-abelian.
\end{proposition}
\begin{proof}
  We only show that $\FF$ is quasi-abelian. Since $\FF$ is closed under extensions in $\AA$, we have that $\FF$ has the natural exact structure, where conflations are short exact sequences with all terms in $\FF$.
  By the axiom of exact categories, inflations are closed under pushouts and deflations are closed under pullbacks. Thus, it suffices to show that $\EE$ is pre-abelian, that every kernel morphism in $\FF$ is an inflation and that every cokernel morphism in $\FF$ is a deflation.

  First, we show that every morphism has a kernel, and every kernel morphism is an inflation. Let $f \colon X \to Y$ be an arbitrary morphism in $\FF$. Then we have the following commutative diagram in $\AA$ with exact rows:
  \[
  \begin{tikzcd}
    0 \rar & K \rar["\iota"] \dar[equal] & X \dar[equal] \rar["f"] & Y \\
    0 \rar & K \rar & X \rar & \im f \uar[hookrightarrow] \rar & 0.
  \end{tikzcd}
  \]
  Since $\FF$ is closed under subobjects in $\AA$, we have $K, \im f \in \FF$. Therefore, $\iota \colon K \to X$ is a kernel of $f$ in $\FF$, and this is actually an inflation in $\FF$ since the bottom row is a conflation in $\FF$.

  Next, we show that every morphism in $\FF$ has a cokernel and every cokernel morphism is a deflation in $\FF$. Let $f \colon X \to Y$ be an arbitrary morphism in $\FF$ and denote by $C$ the cokernel of $f$ in $\AA$. Then we have the following commutative diagram in $\AA$ with exact rows and columns
  \[
  \begin{tikzcd}
    & & X \dar["f"] \\
    0 \rar & L \rar \dar & Y \rar["\pi"] \dar & FC \rar \dar[equal] & 0 \\
    0 \rar & TC \rar & C \rar\dar & FC \rar & 0 \\
    & & 0
  \end{tikzcd}
  \]
   where $TC \in \TT$ and $FC \in \FF$. Then $L \in \FF$ holds since $\FF$ is closed under subobjects, and it is straightforward to see that $\pi$ is a cokernel of $f$ in $\FF$. It follows that the middle row is a conflation in $\FF$ and that $\pi$ is actually a deflation in $\FF$.
\end{proof}
We omit the proof of the following standard lemma, see e.g. \cite[Proposition I.13.2]{mitchel}.
\begin{lemma}\label{pblemma}
  Let $\EE$ be an additive category and suppose that we have a commutative diagram
  \[
  \begin{tikzcd}
    & E \rar["i'"] \dar \ar[rd, phantom, "{\rm p.b.}"] & W \rar["fg"] \dar["g"] & Y \dar[equal] \\
    0 \rar & K \rar["i"] & X \rar["f"] & Y
  \end{tikzcd}
  \]
  such that the left square is pullback and $i$ is a kernel of $f$. Then $i'$ is a kernel of $fg$. In particular, kernel morphisms are closed under pullbacks if $\EE$ has pullbacks.
\end{lemma}
Now we can prove that \emph{quasi-abelian implies lattice property}.
\begin{proposition}\label{quasiabelian}
  Let $\EE$ be a quasi-abelian category with the greatest exact structure and $X$ an object of $\EE$. Then $\PP(X)$ is a lattice.
\end{proposition}
\begin{proof}
  Let $A, B \in \PP(X)$ be two admissible subobjects of $X$.

  We first construct the meet $A \cap B$. Note that since $\EE$ is pre-abelian, $\EE$ has pullbacks. Now we take the following pullback:
  \[
  \begin{tikzcd}
    A\cap B \ar[rd, phantom, "{\rm p.b.}"] \rar \dar & A \dar[rightarrowtail] \\
    B \rar[rightarrowtail] & X
  \end{tikzcd}
  \]
  Then by Lemma \ref{pblemma}, the above morphisms $A \cap B \to A$ and $A \cap B \to B$ are kernel morphisms, hence inflations. Thus, the composition $A \cap B \infl A \infl X$ is an inflation, so it is an admissible subobject of $X$. Hence we have $A \cap B \in \PP(X)$ and that $A\cap B \leq A, B$.
  Suppose that $W\in\PP(X)$ satisfies $W \leq A$ and $W \leq B$. Then by the universal property of the pullback, we have a morphism $W \to A \cap B$ such that the composition $W \to A \cap B \infl X$ is an inflation. Note that $\EE$ is idempotent complete because $\EE$ is pre-abelian. Then Lemma \ref{wic} shows that $W \leq A \cap B$ holds in $\PP(X)$. Therefore, $A \cap B$ is the meet of $A$ and $B$ in the poset $\PP(X)$.

  Next, we will show the existence of the join $A + B \in \PP(X)$. Here we shall give two constructions.

  {\bf (First Construction)}:
  This construction is in fact the dual of the previous one, but we include it here for the completeness.
  Since $\EE$ is pre-abelian, it has pushouts. Now we take the following pushout:
  \[
  \begin{tikzcd}
    X \ar[rd, phantom, "{\rm p.o.}"] \rar[twoheadrightarrow] \dar[twoheadrightarrow] & X/A \dar \\
    X/B \rar & C
  \end{tikzcd}
  \]
  Then by the dual of Lemma \ref{pblemma}, we have that the above morphism $X/A \to C$ is a cokernel morphism, hence a deflation. It follows that the composition $X \defl X/A \defl C$ is a deflation, so $C$ can be written as $C = X/(A+B)$ for some $A+B \in \PP(X)$. Now we leave it to the reader to verify that $A+B$ is actually a join of $A$ and $B$ in $\PP(X)$.

  {\bf (Second Construction)}:
  The second construction is similar to the usual construction of the join in abelian categories: the join $A+B$ is obtained by taking the image of the map $A\oplus B \to X$. Moreover, we will use this construction later to show the modular property of the lattice.
  To state the construction, we use the following basic property on quasi-abelian categories.
  \begin{lemma}[{\cite[Proposition 4.8]{buhler}}]
    Every morphism $f$ in a quasi-abelian category admits a factorization $f = \iota p$ such that $\iota$ is a kernel morphism and $p$ is an epimorphism.
  \end{lemma}
  Now for admissible subobjects $\iota_A \colon A \infl X$ and $\iota_B \colon B \infl Y$, consider the induced morphism $[\iota_A,\iota_B] \colon A\oplus B \to X$.
  By the above lemma, there exist an object $A+B$ and morphisms $A\oplus B \xrightarrow{p} A+B \xrightarrow{\iota} X$ such that $[\iota_A,\iota_B] = \iota p$ holds, $p$ is an epimorphism and $\iota$ is a kernel morphism. We claim that $\iota \colon A+B \infl X$ is a join of $A$ and $B$.

  Clearly $\iota_A$ and $\iota_B$ factors through $\iota$, so $A ,B \leq A+B$ holds by Proposition \ref{wic}.
  On the other hand, suppose that $\iota_C \colon C \infl X$ satisfies $A,B \leq C$. Then since $\iota_A, \iota_B$ factors through $\iota_C$, the universal property of $A\oplus B$ yield a morphism $\varphi \colon A\oplus B \to C$ which makes the following diagram commute, where the bottom row is a conflation.
  \[
  \begin{tikzcd}
    & A \oplus B \rar["p"] \dar["\varphi"] & A + B \dar[rightarrowtail, "\iota"] \\
    0 \rar & C \rar[rightarrowtail, "\iota_C"] & X \rar[twoheadrightarrow, "\pi_C"] & X/C \rar & 0
  \end{tikzcd}
  \]
  It follows that $\pi_C \iota p = 0$, and since $p$ is an epimorphism, we have that $\pi_C \iota=0$. Thus, $\iota$ factors through $\iota_C$, so $A+B \leq C$ holds in $\PP(X)$ by Proposition \ref{wic}.
\end{proof}

\begin{remark}
  Proposition \ref{quasiabelian} can be proved directly in the case of a torsion-free class, as follows.
  Let $(\TT,\FF)$ be a torsion pair in an abelian category $\AA$ and $X$ an object of $\FF$. We give a construction of a meet and a join of $A,B \in \PP_\FF(X)$.
  \begin{itemize}
    \item A meet $A \cap B$ is the usual meet of subobjects in an abelian category $\AA$. It suffices to observe that $A \cap B$ is actually \emph{admissible}, that is, $X/(A \cap B)$ belongs to $\FF$. We have the following exact sequence in $\AA$:
    \[
    0 \to X/(A\cap B) \to X/A \oplus X/B
    \]
    Since $\FF$ is closed under direct sums and subobjects, it follows that $X/(A \cap B)$ belongs to $\FF$, that is, $A \cap B$ is an admissible subobject of $X$.
    \item A join of $A$ and $B$ in $\PP_\FF(X)$ is in general \emph{different from} the usual join $A+B$ of subobjects in an abelian category $\AA$. The problem is that $A+B$ is not necessarily admissible, that is, $X/(A+B)$ may not belong to $\FF$. However, we have the following exact sequence
    \[
    0 \to T\left( X/(A+B) \right) \to X/(A+B) \to F\left( X/(A+B) \right) \to 0.
    \]
    with the left term in $\TT$ and the right term in $\FF$, and let us define $\ov{A+B} \in \PP_\FF(X)$ by the isomorphism $X/(\ov{A+B}) \iso F \left( X/(A+B) \right)$. We leave it to the reader to check that $\ov{A+B}$ is indeed a join of $A$ and $B$ in $\PP_\FF(X)$.
  \end{itemize}
\end{remark}

\subsection{Integral quasi-abelian implies modularity and the unique length property}
It is known that a submodule lattice, or more generally, a subobject lattice in an abelian category, is a \emph{modular lattice} (see e.g. \cite[Proposition IV.5.3]{st}).
First, let us see what modularity of subobject lattices implies in our context.

Let us recall the definition of modularity. A lattice $L$ is called \emph{modular} if for every $a,b,x$ in $L$ with $a \leq b$, we have $(x \vee a) \wedge b = (x \wedge b) \vee a$.
Modularity is a useful tool to study composition series, since we have the following \emph{Jordan-H\"older theorem} for modular lattices. For the proof, we refer the reader to \cite[Corollary 3.2, Proposition 3.3]{st}.
\begin{theorem}[The Jordan-H\"older theorem for modular lattices]\label{latticejh}
  Let $L$ be a modular lattice with the least element $0$ and the greatest element $1$. Suppose that $L$ has at least one finite maximal chain. Then the following holds:
  \begin{enumerate}
    \item Every chain of $L$ must be finite.
    \item Every maximal chain of $L$ has the same length.
    \item Every chain can be refined to some maximal chain, thus the set of lengths of chains of $L$ has an upper bound.
  \end{enumerate}
\end{theorem}
Moreover, we have some kind of \emph{uniqueness of composition series}: for two maximal chains $0 = x_0 < x_1 < \cdots < x_l = 1$ and $0 = y_0 < y_1 < \cdots < y_l = 1$, the intervals $[x_{i-1},x_i]$ and $[y_{j-1},y_j]$ are \emph{``isomorphic''} up to permutations
(\emph{isomorphic} here means \emph{projective} in lattice theory, see e.g. \cite[III. Section 2]{st} for the detail).
Actually, the Jordan-H\"older theorem for abelian categories can be proved by using the above lattice-theoretic one (see \cite[p.92]{st} for example).

For exact categories, modularity of subobject lattices does not imply (JHP) in general because \emph{``isomorphic"} above may not induce an isomorphism. However, modularity clearly \emph{does} imply the unique length property, so the following is an immediate corollary of Theorem \ref{latticejh}.
\begin{corollary}\label{modulp}
  Let $\EE$ be an exact category. Suppose that for every object $X \in \EE$, the poset $\PP(X)$ is a modular lattice and that at least one composition series of $X$ exists. Then $\EE$ is a length exact category and satisfies the unique length property.
\end{corollary}
In the rest of this section, we will show that \emph{integral quasi-abelian categories} are exact categories in which this situation occurs.

\begin{definition}
  A quasi-abelian category $\EE$ is called \emph{integral quasi-abelian} if the class of epimorphisms is closed under pullbacks and the class of monomorphisms is closed under pushouts.
\end{definition}
The typical example is a \emph{hereditary torsion-free class} of abelian categories (Definition \ref{tfdef}), see e.g. \cite[Lemma 6]{rumpstar} for the proof.

To show modularity, we need the following lemma of integral quasi-abelian categories, which asserts that a regular morphism is an \emph{essential epimorphism}:
\begin{lemma}\label{integralemma}
  Let $\EE$ be an integral quasi-abelian category. Suppose that we have morphisms $A \xrightarrow{f} B \xrightarrow{g} C$ in $\EE$ which satisfy the following conditions:
  \begin{enumerate}
    \item $g$ is a monomorphism and an epimorphism (that is, $g$ is regular).
    \item $g f$ is an epimorphism.
  \end{enumerate}
  Then $f$ is an epimorphism.
\end{lemma}
\begin{proof}
  Consider the following commutative diagram in $\EE$:
  \[
  \begin{tikzcd}
    A \rar["f"] & B \rar["\pi", twoheadrightarrow] \ar[rd, phantom, "{\rm p.o.}"] \dar["g"'] & \coker f \dar["g'"] \rar & 0. \\
    & C \rar["\pi'"'] & W
  \end{tikzcd}
  \]
  Then we have $\pi' g f = g' \pi f = 0$, and since $gf$ is an epimorphism, $\pi' = 0$. Thus, $g' \pi = 0$ holds. On the other hand, since monomorphisms are stable under pushouts and $g$ is a monomorphism, so is $g'$. Therefore, $\pi = 0$ holds, so $f$ is an epimorphism.
\end{proof}
To prove modularity, the following criterion is quite useful.
\begin{lemma}[{\cite[Proposition III.2.3]{st}}]\label{modularlemma}
  Let $L$ be a lattice. Then the following are equivalent:
  \begin{enumerate}
    \item $L$ is a modular lattice.
    \item Take any $a,b\in L$ with $a \leq b$, and take $x,c,c' \in [a,b]$. Suppose that $x \vee c = b = x \vee c'$ and $x \wedge c = a = x \wedge c'$ hold (that is, $c$ and $c'$ are complements of $x$ in $[a,b]$), and that $c \leq c'$ holds. Then we have $c = c'$.
  \end{enumerate}
\end{lemma}

\begin{proposition}\label{integralmodular}
  Let $\EE$ be an integral quasi-abelian category with the greatest exact structure and $X$ an object of $\EE$. Then $\PP(X)$ is a modular lattice.
\end{proposition}
\begin{proof}
  For admissible subobjects $A$ and $B$ of $X$, we denote by $A \cap B$ (resp. $A + B$) the meet (resp. join) of them constructed in Proposition \ref{quasiabelian}.

  We make use of Lemma \ref{modularlemma}. By Proposition \ref{interval}, every interval of $\PP(X)$ is isomorphic to $\PP(Y)$ for some object $Y$ of $\EE$. Thus, it suffices to show the following claim:

  {\bf (Claim)}: \emph{Let $A, B_1, B_2$ be admissible subobjects of $X$. Suppose that $A \cap B_1 = A \cap B_2 = 0$, $A + B_1 = A + B_2 = X$ and $B_1 \leq B_2$ hold in $\PP(X)$. Then $B_1 = B_2$ holds}.

  \emph{(Proof of Claim).}
  Let $\iota_A, \iota_{B_1}$ and $\iota_{B_2}$ be the inflations corresponding to $A, B_1$ and $B_2$ respectively.
  By the construction of the meet given by Proposition \ref{quasiabelian}, we have the following pullback diagram for each $i=1,2$.
  \[
  \begin{tikzcd}
    0 \rar \dar \ar[rd, phantom, "{\rm p.b.}"] & B_i \dar[rightarrowtail, "\iota_{B_i}"] \\
    A \rar[rightarrowtail, "\iota_A"] & X
  \end{tikzcd}
  \]
  It follows that $0 \to A \oplus B_i$ is a kernel of $[\iota_A, \iota_{B_i}] \colon A \oplus B_i \to X$, so $[\iota_A, \iota_{B_i}]$ is a monomorphism for each $i=1,2$.

  On the other hand, by the second construction of the join given in Proposition \ref{quasiabelian}, we have the following factorization on the join $A +B_i$ for each $i$:
  \[
  [\iota_A,\iota_{B_i}] \colon A \oplus B_i \xrightarrow{r_i} A+B_i \infl X
  \]
  such that $r_i$ is an epimorphism. Now since $A+B_i = X$, we have $r_i = [\iota_A,\iota_{B_i}] \colon A \oplus B_i \to X$.

  We have shown that $r_i$ is both a monomorphism and an epimorphism for each $i$. Let $\iota \colon B_1 \infl B_2$ be the inflation corresponding to $B_1 \leq B_2$. Then we have the following commutative diagram:
  \[
  \begin{tikzcd}[ampersand replacement=\&]
    A \oplus B_1 \rar["r_1"] \dar[rightarrowtail, "{\begin{sbmatrix} 1 & 0 \\ 0 & \iota \end{sbmatrix}}"'] \& X \dar[equal] \\
    A \oplus B_2 \rar["r_2"'] \& X
  \end{tikzcd}
  \]
  Since $r_2$ is a monomorphism and an epimorphism and $r_1$ is an epimorphism, it follows that $\begin{sbmatrix} 1 & 0 \\ 0 & \iota \end{sbmatrix} \colon A\oplus B_1 \to A \oplus B_2$ is an epimorphism by Lemma \ref{integralemma}. However, $\begin{sbmatrix} 1 & 0 \\ 0 & \iota \end{sbmatrix}$ is a kernel morphism, so it must be an isomorphism. Thus, $\iota$ also is an isomorphism by the standard matrix calculation.
\end{proof}

\begin{corollary}\label{htfulp}
  Let $\EE$ be an integral quasi-abelian category which is length as an exact category (for example, a hereditary torsion-free class of abelian length category). Then $\EE$ satisfies the unique length property.
\end{corollary}
\begin{proof}
  This immediately follows from Proposition \ref{integralmodular} and Corollary \ref{modulp}.
\end{proof}
We refer the reader to Example \ref{ulpnonjh} for examples which satisfy the unique length property.

\section{Grothendieck monoid of an exact category}
To an exact category $\EE$, we can associate the \emph{Grothendieck group} $\KK_0(\EE)$, which is the abelian group defined by the universal property with respect to conflations. The investigation and calculation of this group is a very classical topic in various area of mathematics. However, much less attention has been paid to the \emph{monoid version} of the Grothendieck group, the \emph{Grothendieck monoid} $\MM(\EE)$, whose properties we investigate in this section.

\begin{remark}
  Let us mention some related works on Grothendieck monoids. The notion of Grothendieck monoids was originally introduced and investigated in \cite{bg} to study Hall algebras of exact categories. As stated in the introduction, some statements in Section 3.1 are found in \cite{bg}, and the author is grateful to J. Greenstein for pointing it out.

  Also Brookfield studied the structure of the Grothendieck monoid of the category of finitely generated modules over noetherian rings in \cite{bro1,bro2,bro3}.
\end{remark}

\subsection{Construction and basic properties of the Grothendieck monoid}
Recall that \emph{monoids are assumed to be commutative} in our convention.
First, we give the definition of Grothendieck monoids and study their basic properties as monoids. We refer to Appendix A for some unexplained notions on monoids.
\begin{definition}\label{def:grmon}
  Let $\EE$ be a skeletally small exact category.
  \begin{enumerate}
    \item A map $f \colon \Iso\EE \to M$ to a monoid $M$ is said to \emph{respect conflations} if it satisfies the following conditions:
   \begin{enumerate}
     \item $f[0] = 0$ holds.
     \item For every conflation
     \[
     0 \to X \to Y \to Z \to 0
     \]
     in $\EE$, we have that $f[Y] = f[X] + f[Z]$ holds in $M$.
   \end{enumerate}
   \item A \emph{Grothendieck monoid} $\MM(\EE)$ is a monoid $\MM(\EE)$ together with a map $\pi \colon \Iso\EE \to \MM(\EE)$ which satisfies the following universal property:
   \begin{enumerate}
     \item $\pi$ respects conflations in $\EE$.
     \item Every map $f\colon \Iso\EE \to M$ to an monoid $M$ which respects conflations in $\EE$ uniquely factors through $\pi$, that is, there exists a unique monoid homomorphism $\ov{f} \colon \MM(\EE) \to M$ which satisfies $f = \ov{f} \pi$.
     \[
     \begin{tikzcd}
       \Iso\EE \dar["\pi"'] \rar["f"] & M \\
       \MM(\EE) \ar[ru, "\ov{f}"', dashed]
     \end{tikzcd}
     \]
   \end{enumerate}
   By abuse of notation, we often write $\pi[X] = [X]$ to represent an element in $\MM(\EE)$.
  \end{enumerate}
\end{definition}

First of all, we must show that the Grothendieck monoid \emph{actually exists}.
\begin{proposition}[{c.f. \cite[Definition 2.6]{bg}}]\label{gromon}
  For a skeletally small exact category $\EE$, its Grothendieck monoid $\MM(\EE)$ exists.
\end{proposition}
\begin{proof}
  Define the operation $+$ on $\Iso\EE$ by $[A] + [B] := [A\oplus B]$, then clearly $\Iso\EE$ is a monoid with an additive unit $[0]$.

  Now let us define a monoid congruence $\sim$ on $\Iso\EE$ which is generated by the following relations (see Proposition \ref{gencong}):
   \emph{For every conflation $X \infl Y \defl Z$ in $\EE$, we impose $[Y] \sim [X] + [Z]$.}
  Using this, we obtain a monoid $\MM(\EE):= \Iso\EE / \sim$. Now it is clear from the construction that $\MM(\EE)$ enjoys the required universal property of the Grothendieck monoid.
\end{proof}
This construction was used to define Grothendieck monoids in \cite[Definition 2.6]{bg}.
The following is a more direct characterization whether $[X] \sim [Y]$ holds in $\MM(\EE)$.
Note that for an inflation series $0=X_0 \infl X_1 \infl \cdots \infl X_n = X$ of $X$, we can easily show inductively that $[X] = [X_1] + [X_2/X_1] + \cdots + [X_n/X_{n-1}]$ holds in $\MM(\EE)$. Thus, if $X$ and $Y$ have isomorphic inflation series, then $[X] = [Y]$ holds in $\MM(\EE)$.
\begin{proposition}[{c.f. \cite[Lemma 5.1]{bg}}]\label{explicit}
  Let $\EE$ be a skeletally small exact category and $X,Y$ two objects of $\EE$. Then $[X] = [Y]$ holds in $\MM(\EE)$ if and only if there exist a sequence of objects $X=X_0, X_1, \cdots,\allowbreak X_m = Y$ in $\EE$ such that $X_i$ and $X_{i-1}$ have isomorphic inflation series for each $i$.
\end{proposition}
\begin{proof}
  We freely use the construction of $\MM(\EE)$ in Proposition \ref{gromon}. For two elements $[X],[Y]$ in $\Iso\EE$, we write $[X] \approx [Y]$ if there exists objects $X=X_0, X_1, \cdots, X_m = Y$ such that $X_i$ and $X_{i-1}$ have isomorphic inflation series for each $i$. It is clear that $\approx$ is an equivalence relation on $\Iso\EE$, and it suffices to show that $\approx$ and $\sim$ coincides.

  First, we show that $\approx$ is a monoid congruence on $\Iso\EE$. It is enough to show that for every object $A$ in $\EE$, if $X$ and $Y$ have isomorphic inflation series, then so do $A \oplus X$ and $A \oplus Y$.
  Let $0 = X_0 \infl X_1 \infl \cdots \infl X_n = X$ and $0 = Y_0 \infl Y_1 \infl \cdots \infl Y_n = Y$ be two isomorphic inflation series. Then it is obvious that two inflation series $X_0 \infl \cdots \infl X_n \infl X_n \oplus A$ and $Y_0 \infl Y_1 \infl \cdots \infl Y_n \infl Y_n \oplus A$ are isomorphic, where the last inflations are the inclusion into direct summands. Therefore, $\approx$ is a monoid congruence on $\Iso\EE$.

  Next, we will prove that $\approx$ coincides with $\sim$. From the argument above this proposition, we have that $[X] \approx [Y]$ implies $[X] \sim [Y]$ since $[X] = [X_0] \sim [X_1] \sim \cdots \sim [X_m] = [Y]$. To show the converse implication, it suffices to show that for every conflation $X \infl Y \defl Z$, we have that $[Y] \approx [X] + [Z] = [X \oplus Z]$.
  This is clear since we have inflation series $0 \infl X \infl Y$ of $Y$ and $0 \infl X \infl X \oplus Z$ of $X \oplus Z$ which are isomorphic since $Y/X \iso Z$.
 \end{proof}

Using this description of the Grothendieck monoid, we can prove some properties of it.

\begin{proposition}[{c.f. \cite[Lemma 2.9]{bg}}]\label{mreduced}
  Let $\EE$ be a skeletally small exact category. Then the following hold.
  \begin{enumerate}
    \item For an object $X$ in $\EE$, we have that $[X] = 0$ in $\MM(\EE)$ if and only if $X \iso 0$.
    \item $\MM(\EE)$ is reduced.
  \end{enumerate}
\end{proposition}
\begin{proof}
  (1)
  This follows from Proposition \ref{explicit}, since if $0$ and $X$ has isomorphic inflation series, then clearly $X \iso 0$ holds.

  (2)
  Suppose that $[X] + [Y] = 0$ holds in $\MM(\EE)$. Then $[X\oplus Y] = 0$ holds, hence $X \oplus Y \iso 0$ by (1). Therefore, both $X$ and $Y$ must be isomorphic to $0$ in $\EE$. Thus, $\MM(\EE)$ is reduced.
\end{proof}

We can prove that non-isomorphic simples in $\EE$ are distinct in $\MM(\EE)$, and in fact they are in bijection with \emph{atoms} of $\MM(\EE)$. This property is remarkable compared to the Grothendieck group, since non-isomorphic simples may represent the same element in the Grothendieck group $\KK_0(\EE)$ (see Section \ref{kroex}). We refer the reader to Definition \ref{monoiddef} for the notion of atoms in monoids.
\begin{proposition}[{c.f. \cite[Lemma 5.3]{bg}}]\label{simpleatom}
  Let $\EE$ be a skeletally small exact category. Then the following hold.
  \begin{enumerate}
    \item Let $S$ and $X$ be two objects in $\EE$. Suppose that $[S] = [X]$ holds in $\MM(\EE)$ and that $S$ is simple. Then $S \iso X$ holds in $\EE$.
    \item Let $S_1$ and $S_2$ be two simple objects in $\EE$. Then $[S_1] = [S_2]$ holds in $\MM(\EE)$ if and only if $S_1 \iso S_2$ holds in $\EE$.
    \item The assignment $S \mapsto [S]$ from $\Ob\EE$ to $\MM(\EE)$ induces a bijection $\simp \EE \xrightarrow{\simeq} \Atom\MM(\EE)$.
  \end{enumerate}
\end{proposition}
\begin{proof}
  (1)
  Suppose that $S$ and $X$ have isomorphic inflation series for $X\in\EE$ and a simple object $S$ in $\EE$. Then it is clear that $S \iso X$ holds, since possible inflation series of $S$ is of the form $0 \infl  \cdots \infl 0 \infl S = \cdots = S$. Now the assertion of (1) inductively follows from Proposition \ref{explicit}.

  (2)
  This is a special case of (1).

  (3)
  We first show that $[S]$ is an atom in the monoid $\MM(\EE)$. Suppose that $[S] = [X] + [Y]$. Then we have that $[S] = [X \oplus Y]$, so (1) implies that $S \iso X \oplus Y$. Since simple objects are indecomposable, it follows that $X \iso 0$ or $Y \iso 0$ holds, that is, $[X] = 0$ or $[Y] = 0$. Therefore, $[S]$ is an atom.

  Conversely, suppose that $X$ is not simple. We may assume that $X \neq 0$, so we have a non-trivial conflation $X_1 \infl X \defl X_2$ with $X_1,X_2 \neq 0$. Then $[X] = [X_1] + [X_2]$ holds in $\MM(\EE)$ and we have $[X_1], [X_2] \neq 0$ by Proposition \ref{mreduced}. Thus, $[X]$ is not an atom.

  Now we have proved that atoms in $\MM(\EE)$ are precisely elements which come from simple objects. Then the remaining claim follows immediately from (2).
\end{proof}

For later use, we show another necessary condition for $[X] = [Y]$ in $\MM(\EE)$. Recall that a subcategory $\DD$ of an exact category $\EE$ is called \emph{Serre} if for any conflation $0 \to X \to Y \to Z \to 0$ in $\EE$, we have that $Y \in \DD$ holds if and only if both $X \in \DD$ and $Z \in \DD$ hold.
\begin{proposition}\label{serre}
  Let $\EE$ be a skeletally small exact category and $X, Y$ two objects of $\EE$. Suppose that $X$ belongs to a Serre subcategory $\DD$ of $\EE$ and that $[X] = [Y]$ holds in $\MM(\EE)$. Then $Y$ also belongs to $\DD$.
\end{proposition}
\begin{proof}
  By Proposition \ref{explicit}, it suffices to show the assertion when $X$ and $Y$ have isomorphic inflation series.
  Take inflation series $0 = X_0 \infl X_1 \infl \cdots \infl X_n = X$ and $0 = Y_0 \infl Y_1 \infl \cdots \infl Y_n = Y$ and a permutation $\sigma$ of $\{ 1, \cdots, n\}$ such that $X_i/X_{i-1} \iso Y_{\sigma(i)}/Y_{\sigma(i)-1}$ for each $i$.
  Since $X \in \DD$ and $\DD$ is closed under subquotients, $X_i/X_{i-1}$ belongs to $\DD$, thus so does $Y_i/Y_{i-1}$ for each $i$. Therefore, it inductively follows that $Y \in \DD$ holds, since $\DD$ is extension-closed and $Y$ can be obtained from $Y_i/Y_{i-1}$'s by extensions.
\end{proof}

\subsection{Grothendieck monoid and positive part of Grothendieck group}\label{pospartsubsec}

The \emph{Grothendieck group} of an exact category is a classical invariant, which is defined by the similar universal property as the Grothendieck monoid. In this section, we will recall the Grothendieck group and discuss its relation to the Grothendieck monoid.

\begin{definition}
  Let $\EE$ be a skeletally small exact category.
  \begin{enumerate}
    \item A map $f \colon \Iso\EE \to G$ to an abelian group $G$ is said to \emph{respect conflations} if it satisfies $f[Y] = f[X] + f[Z]$ in $G$ for every conflation $X \infl Y \defl Z$ in $\EE$.
    \item A \emph{Grothendieck group} $\KK_0(\EE)$ of an exact category $\EE$ is an abelian group $\KK_0(\EE)$ together with a map $\pi \colon \Iso\EE \to \KK_0(\EE)$ which satisfies the following universal property:
    \begin{enumerate}
     \item $\pi$ respects conflations in $\EE$.
     \item Every map $f\colon \Iso\EE \to G$ to an abelian group $G$ which respects conflations in $\EE$ uniquely factors through $\pi$, that is, there exists a unique group homomorphism $\ov{f} \colon \KK_0(\EE) \to G$ which satisfies $f = \ov{f} \pi$.
     \[
     \begin{tikzcd}
       \Iso\EE \dar["\pi"] \rar["f"] & G \\
       \KK_0(\EE) \ar[ru, "\ov{f}"', dashed]
     \end{tikzcd}
     \]
   \end{enumerate}
   By abuse of notation, we often write $\pi[X] = [X]$ to represent an element of $\KK_0(\EE)$.
  \end{enumerate}
\end{definition}

By the defining properties of the Grothendieck monoid and the Grothendieck group, it is clear that the latter is obtained from the former by taking the \emph{group completion}, as the following proposition claims. We refer the reader to Appendix A for this notion.
\begin{proposition}
  Let $\EE$ be a skeletally small exact category. Then the Grothendieck group $\KK_0(\EE)$ exists, and it is realized as the group completion $\gp \MM(\EE)$ of the Grothendieck monoid $\MM(\EE)$.
\end{proposition}
In what follows, we always identify $\gp\MM(\EE)$ with $\KK_0(\EE)$ and we denote by $\iota \colon \MM(\EE) \to \KK_0(\EE)$ the natural map which satisfies $\iota [X] = [X]$ for every $X$ in $\EE$.

\begin{definition}
  Let $\EE$ be a skeletally small exact category. We denote by $\KK_0^+(\EE)$ the image of the natural map $\iota \colon \MM(\EE) \to \KK_0(\EE)$, that is, $\KK_0^+(\EE) := \{ [X] \, | \, X \in \EE \}$. We call $\KK_0^+(\EE)$ the \emph{positive part} of the Grothendieck group.
\end{definition}
It follows from Proposition \ref{canquot} that $\iota \colon \MM(\EE) \to \KK_0(\EE)$ induces an isomorphism of monoids $\MM(\EE)_\can \iso \KK_0^+(\EE)$, where $\MM(\EE)_\can$ is the largest cancellative quotient of $\MM(\EE)$.
Thus, the positive part has lost information on \emph{non-cancellative part} compared to the Grothendieck monoid. Even so, $\KK_0^+(\EE)$ is a more sophisticated invariant of $\EE$ than $\KK_0(\EE)$.
Note that we have many examples in which the Grothendieck monoids are not cancellative, see Section \ref{noncancel}.

We shall see later in Corollary \ref{setofdim} that for a large class of exact categories $\EE$ arising in the representation theory of artin algebras, including functorially finite torsion(-free) classes, we can identify $\KK_0^+(\EE)$ with the monoid of dimension vectors of modules in $\EE$.

\section{Characterization of (JHP)}
In this section, we will show the basic relationship between structures of exact categories and combinatorial properties of the Grothendieck monoids.

\subsection{Length-like functions and length exact categories}
First, we will introduce an analogue of \emph{dimension} or \emph{length} of modules, which characterizes length exact categories.
\begin{definition}\label{lengthlikedef}
  Let $\EE$ be a skeletally small exact category. We say that a map $\nu \colon \Iso \EE \to \N$ is a \emph{weakly length-like function} if it satisfies the following conditions:
  \begin{enumerate}
    \item For every conflation $X \infl Y \defl Z$ in $\EE$, we have $\nu[Y] \geq \nu[X] + \nu[Z]$.
    \item $\nu[X] = 0$ implies $X \iso 0$ for every $X \in \EE$.
  \end{enumerate}
  We say that $\nu$ is a \emph{length-like} function if it satisfies (2) and the following:
  \begin{enumerate}
    \item[(1$'$)] For every conflation $X \infl Y \defl Z$ in $\EE$, we have $\nu[Y] = \nu[X] + \nu[Z]$.
  \end{enumerate}
\end{definition}
\begin{remark}\label{llremark}
  Obviously, a length-like function is weakly length-like. It is clear that we can identify a length-like function with a length-like function on the monoid $\MM(\EE)$ by Proposition \ref{mreduced} (1), see Definition \ref{llmonoid}.
\end{remark}
Note that (weakly) length-like functions are far from being unique (for example, any positive multiple of a length-like function is a length-like function).

A weakly length-like function gives an upper bound for all possible lengths of $X$, as follows.
\begin{lemma}\label{bddlength}
  Let $\EE$ be a skeletally small exact category with a weakly length-like function $\nu$. For every $X \in \EE$ and proper inflation series $0 = X_0 < X_1 < \cdots < X_n = X$ of $X$, we have $n \leq \nu[X]$. In particular, $\EE$ is a length exact category.
\end{lemma}
\begin{proof}
  Since $\nu$ is weakly length-like, we have inductively
  \begin{align*}
    \nu[X] & \geq \nu[X_1] + \nu[X/X_1]\\
    & \geq \nu[X_1] + \nu[X_2 / X_1] + \nu[X/X_2]\\
    & \cdots \\
    &\geq \sum_{i=1}^n \nu[X_i/X_{i-1}].
  \end{align*}
  Since $X_{i-1} \neq X_i$ for each $i$, we have $X_i / X_{i-1} \neq 0$. Therefore, $\nu[X_i/X_{i-1}] > 0$, hence $\nu[X_i/X_{i-1}] \geq 1$, since $\nu$ is weakly length-like. It follows that $\nu[X] \geq n$.

  This means that the set of lengths of finite chains of $\PP(X)$ has an upper bound $\nu[X]$ for every $X \in \EE$, thus $\EE$ is a length exact category.
\end{proof}
Actually, the existence of a weakly length-like function is equivalent to the length-ness:
\begin{theorem}\label{lengthlikelength}
  Let $\EE$ be a skeletally small exact category. Then the following are equivalent:
  \begin{enumerate}
    \item $\EE$ has a weakly length-like function.
    \item $\EE$ is a length exact category.
  \end{enumerate}
\end{theorem}
\begin{proof}
  (1) $\Rightarrow$ (2):
  This is Lemma \ref{bddlength}.

  (2) $\Rightarrow$ (1):
  For an object $X \in \EE$, the set of lengths of chains of $\PP(X)$ has an upper bound, thus we can define $\nu[X]$ by the following:
  \[
  \nu[X] := \max \{\, l \,\, |\,\,  \text{there exists a chain of $\PP(X)$ of length $l$}\, \}.
  \]
  We will show that $\nu\colon \Iso \EE \to \N$ is a weakly length-like function. If $\nu[X] = 0$, then $X \iso 0$ holds since otherwise we have a chain $0 < X$ of length one.

  Suppose that we have a conflation $X \infl Y \defl Z$ in $\EE$ and put $n:=\nu[X]$ and $m:= \nu[Z]$.
  We have chains $0=X_0 < X_1 < \cdots < X_n = X$ of $\PP(X)$ and $0=Z_0 < Z_1 < \cdots < Z_m = Z$ of $\PP(Z)$.
  By Proposition \ref{interval}, we have an isomorphism of posets $[X,Y] \iso \PP(Z)$, so we have a chain $X = \ov{Z_0} < \ov{Z_1} < \cdots < \ov{Z_m} = Y$ of $\PP(Y)$ which corresponds to $Z_i$'s.
  Thus, we obtain a chain $0 = X_0 < \cdots < X_n (= X = \ov{Z_0}) < \ov{Z_1} < \cdots < \ov{Z_m} = Y$ of length $m+n$. Therefore, $\nu[Y] \geq m + n$ follows by the definition of $\nu[Y]$.
\end{proof}
\begin{remark}
  The weakly length-like function constructed in the proof of (2) $\Rightarrow$ (1) is the same as the \emph{length} given in \cite[Section 6]{bhlr}.
\end{remark}
Using this theorem, we obtain numerous examples of length exact categories.
\begin{example}\label{llexample}
  The typical example in the representation theory of artin algebra is given as follows.
  Let $\Lambda$ be an artin algebra and $\EE$ an extension-closed subcategory of $\mod\Lambda$. The assignment $X \mapsto l(X)$, where $l(X)$ denotes the (usual) length of $X$ as a $\Lambda$-module, gives a length-like function on $\EE$.
  It is clear that the same argument holds for any extension-closed subcategory of an abelian length category.
\end{example}
\begin{remark}
  It can be showed that the category of maximal Cohen-Macaulay modules $\CM R$ over a commutative Cohen-Macaulay ring $R$ (or more generally, $\CM \Lambda$ for an $R$-order $\Lambda$) has a length-like function. It is an interesting problem to investigate when these categories satisfy (JHP).
\end{remark}
The following gives the criterion for the Grothendieck monoid to be finitely generated.
\begin{proposition}\label{atomicfg}
  Let $\EE$ be a skeletally small exact category. Suppose that $\EE$ is a length exact category. Then $\MM(\EE)$ is atomic, and the following are equivalent:
  \begin{enumerate}
    \item $\simp \EE$ is a finite set.
    \item $\Atom\MM(\EE)$ is a finite set.
    \item $\MM(\EE)$ is a finitely generated.
  \end{enumerate}
\end{proposition}
\begin{proof}
  Since $\EE$ is a length exact category, each object $X$ of $\EE$ admits a composition series $0 = X_0 \infl X_1 \infl \cdots \infl X_n$, and $X_i/X_{i-1}$ is simple for each $i$.
  It follows that $[X] = [X_1/X_0] + [X_2/X_1] + \cdots + [X_n/X_{n-1}]$ holds in $\MM(\EE)$, and Proposition \ref{simpleatom} implies that each $[X_i/X_{i-1}]$ is an atom in $\MM(\EE)$. Thus, $\MM(\EE)$ is atomic.

  Proposition \ref{simpleatom} shows that (1) is equivalent to (2). On the other hand, Proposition \ref{factprop} (2) shows that (2) is equivalent to (3).
\end{proof}

\subsection{Freeness of monoids and (JHP)}
Now we can prove the characterization of (JHP) in terms of its Grothendieck monoid.
\begin{theorem}\label{JHchar}
  Let $\EE$ be a skeletally small exact category. Then the following are equivalent:
  \begin{enumerate}
    \item $\EE$ is a length exact category and satisfies (JHP).
    \item $\MM(\EE)$ is a free monoid, or equivalently, a factorial monoid.
  \end{enumerate}
  In this case, $\MM(\EE)$ is a free monoid with the basis $\Atom \MM(\EE) = \{ [S] \,|\, S \in \simp \EE \}$, and $\KK_0(\EE)$ is a free abelian group with the basis $\{ [S] \,|\, S \in \simp \EE \}$.
\end{theorem}
\begin{proof}
  Recall that we have an equality $\Atom \MM(\EE) = \{ [S] \,|\, S \in \simp \EE \}$ by Proposition \ref{simpleatom}. Also note that the conditions in (2) are equivalent by Proposition \ref{factprop} (5) and Proposition \ref{mreduced}.

  (1) $\Rightarrow$ (2):
  Let us denote by $F$ the free monoid with basis $\Atom\MM(\EE)$. Then we have a natural monoid morphism $\varphi \colon F \to \MM(\EE)$ defined by $\varphi [S] = [S]$ for $[S]$ in $\Atom \MM(\EE)$. We shall construct the inverse of this map.

  For $[X]$ in $\Iso\EE$, take a composition series $0 = X_0 < X_1 < \cdots < X_n = X$ of $X$. Define the map $\Iso\EE \to F$ by $[X] \mapsto [X_1/X_0] + [X_2/X_1] + \cdots + [X_i/X_{i-1}] \in F$. This map does not depend on the choice of composition series by (JHP). We show that this map respects conflations in $\EE$.

  Suppose that we have a conflation $X \infl Y \defl Z$ in $\EE$ and that we have composition series $0=X_0 < X_1 < \cdots < X_n = X$ in $\PP(X)$ and $0=Z_0 < Z_1 < \cdots < Z_m = Z$ in $\PP(Z)$.
  By Proposition \ref{interval} (2), we have an isomorphism of posets $[X,Y] \iso \PP(Z)$, so we have a chain $X = \ov{Z_0} < \ov{Z_1} < \cdots < \ov{Z_m} = Y$ in $\PP(Y)$ which corresponds to the chosen composition series of $Z$.
  For each $i$, we have $\ov{Z_i}/\ov{Z_{i-1}} \iso Z_i/Z_{i-1}$ by Proposition \ref{interval} (2), so $\ov{Z_i}/\ov{Z_{i-1}}$ is simple. Thus, $0 = X_0 < \cdots < X_n (= X = \ov{Z_0}) < \ov{Z_1} < \cdots < \ov{Z_m} = Y$ is a composition series of $Y$.
  Now it is obvious that the map $\Iso\EE \to F$ constructed above respects conflations. Thus, we obtain a monoid homomorphism $\psi \colon \MM(\EE) \to F$.

  We claim that $\varphi$ and $\psi$ are mutually inverse to each other.

  ($\psi\varphi = \id_F$):
  Since $F$ is generated by $[S]$ with $S\in \simp \EE$, it suffices to show that $\psi\varphi [S] = [S]$ holds. This is obvious since $\varphi [S] = [S]$, and $S$ has the composition series $0 < S$.

  ($\varphi \psi = \id_{\MM(\EE)}$):
  Let $X$ be an object in $\EE$ with a composition series $0 = X_0 < X_1 < \cdots < X_n = X$. Then $\varphi \psi [X] = [X_1/X_0] + [X_2/X_1] + \cdots + [X_n/X_{n-1}]$. On the other hand, inductively we have $[X] = [X_1/X_0] + [X_2/X_1] + \cdots + [X_n/X_{n-1}]$ in $\MM(\EE)$, thus $\varphi\psi [X] = [X]$ holds.

  (2) $\Rightarrow$ (1):
  First, we show that $\EE$ is a length exact category. By Lemma \ref{freell}, we have that $\MM(\EE)$ has a length-like function. Thus, $\EE$ has a length-like function, therefore it is a length exact category by Proposition \ref{lengthlikelength}.

  Next, we will show that $\EE$ satisfies (JHP). Let $X$ be an object of $\EE$ and let $0 = X_0 < X_1 < \cdots < X_m =X$ and $0 = Y_0 < Y_1 < \cdots <Y_n = X$ be two composition series of $X$. Then we have
  \[
  [X] = [X_1/X_0] + [X_2/X_1] + \cdots + [X_m/X_{m-1}]
  = [Y_1/Y_0] + [Y_2/Y_1] + \cdots + [Y_n/Y_{n-1}]
  \]
  in $\MM(\EE)$, where $[X_i/X_{i-1}]$ and $[Y_j/Y_{j-1}]$ belong to $\Atom\MM(\EE)$ for each $i$.
  Then since $\MM(\EE)$ is free on $\Atom\MM(\EE)$ by Proposition \ref{factprop} (3), it follows that $m = n$ holds, and that there exists a permutation $\sigma$ of the set $\{ 1, 2, \cdots, n \}$ such that $[X_i/X_{i-1}] = [Y_{\sigma(i)}/Y_{\sigma(i)-1}]$ holds in $\MM(\EE)$ for each $i$.
  This implies that $X_i/X_{i-1} \iso Y_{\sigma(i)}/Y_{\sigma(i)-1}$ by Proposition \ref{simpleatom} (2). Thus, these two composition series are isomorphic.

  The remaining assertions follow from Proposition \ref{factprop} (3) and Example \ref{freecompl}.
\end{proof}

The most classical example in which (JHP) holds is, as stated in the introduction, an \emph{abelian length category}. We say that an abelian category is called \emph{length} if every object has a composition series  \cite{gabriel}, which implies that it is length as an exact category (see Remark \ref{lengthremark}).
\begin{corollary}\label{abeliancase}
  Let $\AA$ be an abelian length category. Then $\MM(\AA)$ is a free monoid with the basis $\{ [S] \, | \, S \in \simp \AA \}$. In particular, $\KK_0(\AA)$ is a free abelian group with the same basis.
\end{corollary}
\begin{proof}
  This follows from Theorem \ref{JHchar} and the Jordan-H\"older theorem for abelian length categories (see e.g. \cite[p.92]{st}).
\end{proof}

\subsection{Grothendieck groups and (JHP)}
We have shown that (JHP) can be checked by the freeness of the Grothendieck monoid, but the computation of it is rather hard.
In this subsection, we give a criterion for (JHP) in terms of the Grothendieck group, which is easier to compute than the Grothendieck monoid.
All content of this subsection is just a formal consequence of the general theory of monoids, that is, the general criterion for a given monoid to be free (Theorem \ref{freechar0} and Corollary \ref{freechar}).

Before we state this, we will make an inequality concerning the Grothendieck groups.
\begin{proposition}
  Let $\EE$ be a skeletally small length exact category. Suppose that $\KK_0(\EE)$ is a free abelian group. Then the following inequality holds:
  \[
  \rank \KK_0(\EE) \leq \#\Atom \MM(\EE) = \# \simp\EE.
  \]
\end{proposition}
\begin{proof}
  This follows from the corresponding statement on monoids, Proposition \ref{ineq}.
\end{proof}
Now let us state our main characterizations of (JHP) in terms of the Grothendieck group.
\begin{theorem}\label{JHchar1.5}
  Let $\EE$ be a skeletally small exact category. Then the following are equivalent:
  \begin{enumerate}
    \item $\EE$ is a length exact category and satisfies (JHP).
    \item $\MM(\EE)$ is free.
    \item $\EE$ is a length exact category, and $\KK_0(\EE)$ is a free abelian group with the basis $\{ [S] \,|\, S \in \simp \EE \}$.
    \item $\EE$ is a length exact category, and all elements $[S] \in \KK_0(\EE)$ with $S \in \simp\EE$ are linearly independent over $\Z$ in $\KK_0(\EE)$.
  \end{enumerate}
\end{theorem}
\begin{proof}
  This immediately follows from Theorems \ref{JHchar} and \ref{freechar0} once we observe the following:
  \begin{itemize}
    \item $\Atom \MM(\EE) = \{ [S] \,|\, S \in \simp \EE \}$ holds (Proposition \ref{simpleatom}).
    \item $\EE$ has a length-like function if and only if $\MM(\EE)$ has a length-like function (Definition \ref{lengthlikedef}).
    \item $\MM(\EE)$ is reduced (Proposition \ref{mreduced}).
    \item If $\EE$ is a length exact category, then $\MM(\EE)$ is atomic (Proposition \ref{atomicfg}).
  \end{itemize}
\end{proof}
If the Grothendieck group is known to be free of finite rank, then we have a more convenient characterization: we only have to count the number of simples and compare it to the rank.
\begin{theorem}\label{JHchar2}
  Let $\EE$ be a skeletally small exact category. Then the following are equivalent:
  \begin{enumerate}
    \item $\EE$ is a length exact category satisfying (JHP) and $\# \simp\EE$ is finite.
    \item $\MM(\EE)$ is free and $\#\simp\EE$ is finite.
    \item $\MM(\EE)$ is a finitely generated free monoid.
    \item $\EE$ is a length exact category, $\#\simp\EE$ is finite and $\KK_0(\EE)$ is a free abelian group with the basis $\{ [S] \,|\, S \in \simp \EE \}$.
    \item The following conditions hold:
    \begin{enumerate}
      \item $\EE$ is a length exact category.
      \item $\KK_0(\EE)$ is a free abelian group of finite rank.
      \item $\# \simp\EE = \rank \KK_0(\EE)$ holds.
    \end{enumerate}
  \end{enumerate}
\end{theorem}
\begin{proof}
  This immediately follows from Theorems \ref{JHchar} and \ref{freechar0}, as in the proof of Theorem \ref{JHchar1.5}.
\end{proof}
This characterization has lots of applications in Part 2, since Grothendieck groups of various exact categories arising in the representation theory of algebras are turned out to be free of finite rank (see Proposition \ref{cotiltprop}).

\subsection{Half-factoriality of monoids and the unique length property}
In this subsection, we will give a characterization of the \emph{unique length property} in terms of Grothendieck monoids. As we have seen in Theorem \ref{JHchar}, (JHP) corresponds to freeness, or equivalently, factoriality of monoids. We will see that the unique length property corresponds to \emph{half-factoriality}. This is natural since both two properties are about the uniqueness of length of factorizations.
We refer the reader to Definition \ref{monoiddef} for the notion of half-factorial monoids.

\begin{theorem}
  Let $\EE$ be a skeletally small exact category. Then the following are equivalent:
  \begin{enumerate}
    \item $\EE$ is a length exact category and satisfies the unique length property.
    \item $\EE$ has a length-like function $l$ satisfying $l[S]=1$ for every simple object $S$ in $\EE$.
    \item $\MM(\EE)$ is half-factorial.
  \end{enumerate}
\end{theorem}
\begin{proof}
  The proof is similar to that of Theorem \ref{JHchar}, and actually is easier.

  (1) $\Rightarrow$ (2):
  We will construct a length-like function $l\colon \Iso\EE \to \N$.
  Let $X$ be an object of $\EE$. Since $\EE$ satisfies the unique length property, $X$ has at least one composition series, and every composition series of $X$ has the same length $n$. We define $l[X]:=n$.
  Then $l$ is a length-like function by the similar argument to the proof of Theorem \ref{lengthlikelength}. Moreover, clearly $l[S] = 1$ holds for every simple object $S$ in $\EE$.

  (2) $\Leftrightarrow$ (3):
  By Remark \ref{llremark}, we can identify a length-like function on $\EE$ with that of $\MM(\EE)$.
  Moreover, we have an equality $\Atom \MM(\EE) = \{ [S] \,|\, S \in \simp \EE \}$ by Proposition \ref{simpleatom}. Thus, this equivalence follows from Lemma \ref{freell}.

  (2) $\Rightarrow$ (1):
  By Proposition \ref{lengthlikelength}, our category $\EE$ is a length exact category.
  Let $X$ be an object of $\EE$. Then it is easily checked that the length of any composition series of $X$ is equal to $l[X]$. Thus, $\EE$ satisfies the unique length property.
\end{proof}

\part{Applications to artin algebras}
In this part, we investigate several properties on exact categories arising in the representation theory of artin algebras, by using the results in the previous part.
\section{(JHP) for extension-closed subcategories of module categories}
In this section, we investigate (JHP) for extension-closed subcategories of module categories of artin algebras. \emph{In the rest of this paper, we fix a commutative artinian ring $R$}.
An $R$-algebra $\Lambda$ is called an \emph{artin $R$-algebra} if $\Lambda$ is finitely generated as an $R$-module. We often omit the base ring $R$ and call $\Lambda$ an \emph{artin algebra}.
For an artin $R$-algebra $\Lambda$, we denote by $D \colon \mod\Lambda \to \mod\Lambda^{\op}$ the standard Matlis duality.

\subsection{Basic properties on Grothendieck monoids}
First, we collect some basic properties on the Grothendieck monoid, which immediately follow from the general observations we have made. In particular, every extension-closed subcategory of $\mod\Lambda$ for an artin algebra $\Lambda$ is a length exact category.
\begin{proposition}\label{genprop}
  Let $\Lambda$ be an artin algebra and $\EE$ an extension-closed subcategory of $\mod\Lambda$. Then the following hold.
  \begin{enumerate}
    \item The assignment $X \mapsto l(X)$, where $l(X)$ denotes the usual length of $X$ as a $\Lambda$-module, induces a length-like function on $\EE$.
    \item $\EE$ is a length exact category.
    \item $\MM(\EE)$ is atomic.
    \item $\MM(\EE)$ is finitely generated if and only if $\simp\EE$ is a finite set.
  \end{enumerate}
\end{proposition}
\begin{proof}
  These follow from Example \ref{llexample} (1), Theorem \ref{lengthlikelength} and Proposition \ref{atomicfg}.
\end{proof}
The simplest example is $\EE = \mod\Lambda$, and we can compute $\MM(\mod\Lambda)$ and $\KK_0(\mod\Lambda)$ as follows.
\begin{proposition}\label{modprop}
  Let $\Lambda$ be an artin algebra, and let $S_1, \cdots, S_n$ be a complete set of simple right $\Lambda$-modules up to isomorphism. Then $\MM(\mod\Lambda)$ is a free monoid with basis $[S_1], \cdots, [S_n]$. In particular, $\KK_0(\mod\Lambda)$ is a free abelian group with the same basis. Moreover, $|\Lambda| = \# \ind \P(\mod\Lambda) = \# \simp (\mod\Lambda) =  n$ holds.
\end{proposition}
\begin{proof}
  All the assertions except the last follow form Corollary \ref{abeliancase}. The last one holds because  there is a bijection between non-isomorphic indecomposable projective $\Lambda$-modules and non-isomorphic simple $\Lambda$-modules.
\end{proof}
By this, the Grothendieck group $\KK_0(\mod\Lambda)$ is often identified with $\Z^n$, where $n$ is the number of simple $\Lambda$-modules. This is what is usually called the \emph{dimension vectors} of modules.
\begin{definition}\label{dimvecdef}
  Let $\Lambda$ be an artin algebra, and fix a complete set of non-isomorphic simple $\Lambda$-modules $\{ [S_1], \cdots, [S_n]\}$. Then we denote by $\udim \colon \mod\Lambda \to \Z^n$ the assignment which sends $X$ to $(a_1, \cdots, a_n)$, where $a_i$ is the Jordan-H\"older multiplicity of $S_i$ in $X$. This induces an isomorphism $\KK_0(\mod\Lambda) \xrightarrow{\sim} \Z^n$.
  For a $\Lambda$-module $X$, we call $\udim X$ the \emph{dimension vector} of $X$.
\end{definition}

A typical example of extension-closed subcategories is an \emph{Ext-perpendicular category} with respect to a given module.
\begin{definition}
  Let $\Lambda$ be an artin algebra and $U \in \mod\Lambda$ a $\Lambda$-module. We denote by $^\perp U$ the subcategory of $\mod\Lambda$ which consists of $X \in \mod\Lambda$ such that $\Ext^{>0}_\Lambda (X,U)=0$ holds, that is, $\Ext_\Lambda^{i}(X,U) = 0$ for all $i > 0$.
\end{definition}
Note that $^\perp U$ is an extension-closed subcategory of $\mod\Lambda$, and we always regard it as an exact category from now on.
Exact categories arising in this way have nice homological properties as follows. These can be checked directly, so we omit the proofs.
\begin{proposition}\label{perprop}
  Let $\Lambda$ be an artin algebra and $U \in \mod\Lambda$ a $\Lambda$-module, and put $\EE:= {}^\perp U$. Then the following hold.
  \begin{enumerate}
    \item For any short exact sequence $0 \to X \to Y \to Z \to 0$ in $\mod\Lambda$, if $Y$ and $Z$ belong to $\EE$, then so does $X$.
    \item $\EE$ is closed under direct summands, thus it is idempotent complete.
    \item $\proj\Lambda \subset \EE$ holds.
    \item $\EE$ is an exact category with a progenerator $\Lambda$.
  \end{enumerate}
\end{proposition}

\subsection{Ext-perpendicular categories of modules with finite injective dimension}
To check whether (JHP) holds or not, it is easier to deal with the case where the Grothendieck group is free of finite rank, since we can use Theorem \ref{JHchar2}. Let us introduce an assumption which ensures this.
\begin{assumption}\label{assumperp}
  There exists an artin algebra $\Lambda$ and a $\Lambda$-module $U\in\mod\Lambda$ with finite injective dimension such that $\EE$ is equivalent to ${}^\perp U$ as an exact category.
\end{assumption}
We give examples of such $\EE$ later in Example \ref{perpex}.

\begin{lemma}\label{assumplemm}
  Let $\EE$, $\Lambda$ and $U$ be as in Assumption \ref{assumperp}. Then the following hold.
  \begin{enumerate}
    \item Put $n:= \id(U_\Lambda)$. For every module $X \in \mod\Lambda$, there exists an exact sequence
      \begin{equation}\label{kore}
      0 \to \Omega^n X \to P_{n-1} \to \cdots \to P_1 \to P_0 \to X \to 0
      \end{equation}
    in $\mod\Lambda$ such that each $P_i$ is finitely generated projective and $\Omega^n X$ belongs to $\EE$.
    \item The natural inclusion $\EE \hookrightarrow \mod\Lambda$ induces an isomorphism of the Grothendieck groups $\KK_0(\EE) \xrightarrow{\sim} \KK_0(\mod\Lambda)$.
    \end{enumerate}
\end{lemma}
\begin{proof}
  (1)
  By taking a projective resolution of $X$, obviously there exists a short exact sequence of the form (\ref{kore}) such that each $P_i$ is finitely generated projective. Thus, it suffices to show that $\Omega^n X \in {}^\perp U$. This is because $\Ext_\Lambda^{>0}(\Omega^n X,U) = \Ext_\Lambda^{>n}(X,U) = 0$ since $\id U = n$.

  (2)
  This follows from (1) and Quillen's Resolution Theorem \cite[\S 4]{qu} on algebraic K-theory. We give an elementary proof here, which is similar as given in \cite[Lemma 13.2]{yo}.

  We shall construct the inverse of the natural homomorphism $\KK_0(\EE) \to \KK_0(\mod\Lambda)$. For a module $X \in \mod\Lambda$, take an exact sequence of the form (\ref{kore}). Then consider the assignment $X \mapsto \sum_{0 \leq i < n}(-1)^i [P_i] + (-1)^n [\Omega^n X]$.
  This assignment does not depend on the choice of exact sequences of the form (\ref{kore}) by the Schanuel lemma, and it respects short exact sequences by the Horseshoe lemma. Thus, we obtain the map $\KK_0(\mod\Lambda) \to \KK_0(\EE)$. This is the desired inverse, and we leave it to the reader to check the details.
\end{proof}
To sum up, our exact category has the following nice properties.
\begin{proposition}\label{assumprop}
  Let $\EE$ be as in Assumption \ref{assumperp}. Then the following hold.
  \begin{enumerate}
    \item $\EE$ is a length exact category with a progenerator.
    \item $\KK_0(\EE)$ is free of finite rank.
    \item $\rank \KK_0(\EE) = \# \ind \P(\EE) = |P|$ holds, where $P$ is a progenerator of $\EE$.
  \end{enumerate}
\end{proposition}
\begin{proof}
  (1) follows from Propositions \ref{genprop} (2) and \ref{perprop} (4).
  (2) and (3) follows directly from Lemma \ref{assumplemm} and Proposition \ref{modprop}.
\end{proof}
In this situation, the positive part of the Grothendieck group of $\EE$ can be identified with the set of dimension vectors of modules which belong to $\EE$.
\begin{corollary}\label{setofdim}
  Let $\Lambda$ and $\EE$ be as in Assumption \ref{assumperp}, and fix a complete set of simple $\Lambda$-modules $\{[S_1], \cdots, [S_n]\}$.
  Then the natural map $\MM(\EE) \to \KK_0(\EE) \to \KK_0(\mod\Lambda) \iso \Z^n$ induces an isomorphism of monoids between $\MM(\EE)_\can$ and the monoid of dimension vectors of modules in $\EE$:
  \[
  \udim \EE = \{\, \udim X \,\, | \,\, X \in \EE\, \} \subset \N^n.
  \]
\end{corollary}
\begin{proof}
  This is immediate from Proposition \ref{canquot}, the definition of $\KK_0^+(\EE)$, an isomorphism $\KK_0(\EE) \xrightarrow{\sim} \KK_0(\mod\Lambda)$ shown in Proposition \ref{perprop} and the definition of the dimension vector.
\end{proof}
Now our previous characterization of (JHP) has rather simple consequence in this situation.
\begin{theorem}\label{artinmain}
  Let $\EE$ be an exact category which satisfies Assumption \ref{assumperp}. Then $\EE$ satisfies (JHP) if and only if $\# \simp \EE = \#\ind \P(\EE)$ holds, that is, the number of non-isomorphic simples in $\EE$ is equal to that of non-isomorphic indecomposable projectives in $\EE$.
\end{theorem}
\begin{proof}
  Immediately follow from Proposition \ref{assumprop} and Theorem \ref{JHchar2}.
\end{proof}

Actually most exact categories which have been investigated in the representation theory of artin algebras satisfy Assumption \ref{assumperp}. Among them, those arising from \emph{cotilting modules} have been attracted an attention.
\begin{definition}
  Let $\Lambda$ be an artin algebra. We say that a $\Lambda$-module $U \in \mod\Lambda$ is \emph{cotilting} if it satisfies the following conditions:
  \begin{enumerate}
    \item The injective dimension of $U$ is finite.
    \item $\Ext^{>0}_\Lambda(U,U) = 0$ holds.
    \item There exists an exact sequence
    \[
    0 \to U_n \to \cdots \to U_1 \to U_0 \to D \Lambda \to 0
    \]
    in $\mod\Lambda$ for some $n$ such that $U_i \in \add U$ for each $i$.
  \end{enumerate}
\end{definition}
The simplest example of cotilting modules is $D\Lambda$. In this case, the perpendicular category $^\perp(D\Lambda)$ coincides with the module category $\mod\Lambda$. For a general cotilting module, although $^\perp U$ is not abelian, it has nice properties.
\begin{proposition}\label{cotiltprop}
  Let $\Lambda$ be an artin algebra and $U$ a cotilting $\Lambda$-module. Put $\EE:={}^\perp U$. Then the following hold:
  \begin{enumerate}
    \item $\EE$ has a progenerator $\Lambda$ and an injective cogenerator $U$.
    \item $\KK_0(\EE)$ is a free abelian group of finite rank.
    \item $\rank \KK_0(\EE) = \#\ind\P(\EE) = \#\ind\I(\EE)$ holds.
  \end{enumerate}
\end{proposition}
\begin{proof}
  (1)
  By Proposition \ref{perprop}, the exact category $\EE$ has a progenerator $\Lambda$. We refer to \cite[Theorem 5.4]{applications} for the proof of the fact that $U$ is an injective cogenerator of $\EE$.

  (2)
  This follows from Propositions \ref{perprop} and \ref{modprop}.

  (3)
  By Propositions \ref{perprop} and \ref{modprop}, we only have to show that $\#\ind \I(\EE) = |U|$ is equal to $\# \ind \P(\EE)= |\Lambda|$. This follows from tilting theory, e.g. \cite[Theorem 1.19]{miyashita}.
\end{proof}
In particular, Proposition \ref{cotiltprop} says the number of indecomposable \emph{projectives} and \emph{injectives} in $\EE$ coincide. In the case of $\mod\Lambda$, this number is also equal to the number of \emph{simples}. Thus, Theorem \ref{artinmain} says that the violation of this coincidence can be understood as an obstruction for (JHP).

\begin{example}\label{perpex}
  The following exact categories $\EE$ satisfy Assumption \ref{assumperp}, hence Theorem \ref{artinmain} holds for them.
  \begin{enumerate}
    \item $\EE:= {}^\perp U$ for a cotilting $\Lambda$-module $U$ over an artin algebra $\Lambda$. Note that $\# \ind\P(\EE) = \# \ind\I(\EE)$ holds in this case, by Proposition \ref{cotiltprop}.
    \item $\EE = \GP\Lambda$ for an Iwanaga-Gorenstein artin algebra $\Lambda$. Here $\Lambda$ is called \emph{Iwanaga-Gorenstein} if $\id \Lambda_\Lambda$ and $\id {}_\Lambda \Lambda$ are both finite, and in this case, we denote by $\GP\Lambda$ the category ${}^\perp \Lambda$, called the category of \emph{Gorenstein-projective} modules.
    \item Functorially finite torsion-classes and torsion-free classes of $\mod\Lambda$ for an artin algebra $\Lambda$ (see e.g. \cite{applications} for functorial finiteness). This classes of categories has been attracted an attention, since \emph{$\tau$-tilting theory} gives a powerful combinatorial tool to investigate them, e.g. \cite{air}.
  \end{enumerate}
\end{example}
\begin{proof}
  (1) and (2) follow from definition. (3) is well-known to experts (e.g. see \cite[Proposition 1.2.1]{iy}), but we give a sketch here for the convenience of the reader.

  We show that a functorially finite torsion-free class $\FF$ is a special case of (1). By factoring out the annihilator, we may assume that $\FF$ is faithful torsion-free class of $\Lambda$, that is, the intersection of annihilators of modules in $\FF$ is zero. Then it can be shown that $\Lambda_\Lambda \in \FF$ holds. The well-known characterizations on classical (co)-tilting modules tells us that $\FF = {}^\perp U$ holds for a cotilting module $U$ with $\id U \leq 1$ (see e.g. \cite[Theorem VI.6.5]{ASS}).

  For a functorially finite torsion class $\TT$, the similar argument shows that $\TT$ comes from the classical tilting module. Now the Brenner-Butler theorem (\cite[Theorem VI.3.8]{ASS}) tells us that $\TT$ is equivalent as an exact category to the torsion-free class over another algebra which is induced by some classical cotilting module. Thus, $\TT$ is also a special case of (1).
\end{proof}

For those familiar with $\tau$-tilting theory, we will state a consequence of our result in $\tau$-tilting-theoretical language. We omit the related definitions and notation, see \cite{air} for the detail.
\begin{corollary}
  Let $\Lambda$ be an artin algebra and $\FF = \Sub U$ a functorially finite torsion-free class of $\mod\Lambda$ where $U$ is a support $\tau^-$-tilting module. Then the following are equivalent:
  \begin{enumerate}
    \item $\FF$ satisfies (JHP).
    \item The number of non-isomorphic simple objects in $\FF$ is equal to $|U|$.
  \end{enumerate}
\end{corollary}
Dually we obtain the $\tau$-tilting and torsion-class version:
\begin{corollary}\label{tautilting}
  Let $\Lambda$ be an artin algebra and $\TT = \Fac T$ a functorially finite torsion class of $\mod\Lambda$ where $T$ is a  support $\tau$-tilting module. Then the following are equivalent:
  \begin{enumerate}
    \item $\TT$ satisfies (JHP).
    \item The number of non-isomorphic simple objects in $\TT$ is equal to $|T|$.
  \end{enumerate}
\end{corollary}

In the case of functorially finite torsion-free classes, we obtain the following finiteness result on the positive cone $\MM(\FF)_\can \iso \KK_0^+(\FF)$. Note that $\MM(\FF)$ itself is not in general finitely generated (see Section \ref{kroex} for example).
\begin{proposition}\label{affinemonoid}
  Let $\Lambda$ be an artin algebra and $\FF$ a torsion-free class of $\mod \Lambda$ such that $\FF$ is the smallest torsion-free class which contains some $\Lambda$-module $U$ (for example, $\FF$ is functorially finite). Then $\KK_0^+(\FF)$ is isomorphic to a finitely generated submonoid of $\N^n$ for some $n$.
\end{proposition}
\begin{proof}
  This follows from the following fact: $\FF$ is described as the category of modules which have finite filtrations such that each successive quotient is a submodule of $U$ (to see this, it suffices to prove that the latter category is closed under submodules, and this follows from the same argument as in \cite[Lemma 3.1]{ms}). It follows that $\MM(\FF)$ is generated by $\{ [V] \, | \, \text{$V$ is a submodule of $U$} \}$.
  By Corollary \ref{setofdim}, we may identify $\KK_0^+(\FF)$ with the monoid of dimension vectors of modules in $\FF$, which is generated by the set of dimension vectors of submodules of $U$. This set is obviously finite, so $\KK_0^+(\FF)$ is finitely generated.
\end{proof}

\subsection{Torsion-free classes over Nakayama algebras}
In this subsection, we investigate torsion-free classes over Nakayama algebras, and show that all such categories satisfy (JHP).

First, we recall the notion of Nakayama algebras. For an artin algebra $\Lambda$, we say that a $\Lambda$-module $M$ is \emph{uniserial} if the set of submodules of $M$ is totally ordered by inclusion. An artin algebra $\Lambda$ is called \emph{Nakayama} if every indecomposable right and left projective $\Lambda$-module is uniserial. We will use the following description of indecomposable modules over Nakayama algebras (see e.g. \cite[Chapter V]{ASS} or \cite[Section VI.2]{ARS} for the detail).
\begin{proposition}\label{nakayamaprop}
  Let $\Lambda$ be a Nakayama algebra. Then the following hold.
  \begin{enumerate}
    \item Every indecomposable module in $\mod\Lambda$ is uniserial.
    \item $M \in \ind (\mod\Lambda)$ is uniquely determined by the following data:
    \begin{enumerate}
      \item The simple module $S :=\top M = M/\rad M$.
      \item $m:=l(M)$, the length of $M$ as a $\Lambda$-module.
    \end{enumerate}
    In particular, $\Lambda$ is of finite representation type.
    We denote this module by $S^{(m)}$.
    \item If two indecomposable modules $M$ and $N$ in $\mod\Lambda$ satisfy $\top M \iso \top N$ and $l(M) \geq l(N)$, then there exists a surjection $M \defl N$.
  \end{enumerate}
\end{proposition}

We will investigate simple objects in a torsion-free class of $\mod\Lambda$ for a Nakayama algebra.
Note that since $\mod\Lambda$ has finitely many indecomposables, every torsion-free class of $\mod\Lambda$ is functorially finite, so it satisfies Assumption \ref{assumperp} by Example \ref{perpex} (3).
\begin{theorem}\label{nakasimp}
  Let $\Lambda$ be a Nakayama algebra and $\FF$ a torsion-free class of $\mod\Lambda$. Then there exists bijections between the following three sets:
  \begin{enumerate}
    \item $\top \FF$, the set of isomorphism classes of simple modules $\top M$ for $M \in \ind \FF$.
    \item $\simp \FF$, the set of isomorphism classes of simple objects in $\FF$.
    \item $\ind\P(\FF)$, the set of isomorphism classes of indecomposable projective objects in $\FF$.
  \end{enumerate}
  The maps from {\upshape (2)} and {\upshape (3)} to {\upshape (1)} are given by $M \mapsto \top M$.
  On the other hand, for a simple module $S \in \top \FF$ in {\upshape (1)}, the corresponding objects are given by $S^{(m)}$ in {\upshape (2)} and $S^{(n)}$ in {\upshape (3)}, where
  \begin{align*}
    m &:= \min \{ i \,\,|\,\, S^{(i)} \in \FF \}, \\
    \text{and} \quad n &:= \max \{ i \,\,|\,\, S^{(i)} \in \FF \}.
  \end{align*}
\end{theorem}
\begin{proof}
  Let us denote by $\top \colon \ind\FF \to \top \FF$ the map which sends $M$ to $\top M$. This induces maps $\simp\FF \to \top \FF$ and $\ind\P(\FF) \to \top\FF$. We will show that these two maps are bijections.

  ($\top \colon \simp\FF \to \top \FF$ is a bijection):
  First, we will show that this map is a surjection.
  For $S \in \top\FF$, put $m := \min \{ i \,\,|\,\, S^{(i)} \in \FF \}$. We claim that $S^{(m)}$ is a simple object in $\FF$. Suppose this is not the case. Then there exists a short exact sequence
  \[
  0 \to K \to S^{(m)} \to M \to 0
  \]
  in $\FF$ with $K, M \neq 0$. Since $M$ is a quotient of $S^{(m)}$, it has top $S$. But this contradicts the minimality of $m$ because $0 < l(M) < l(S^{(m)}) = m$ holds. Thus, $S^{(m)}$ belongs to $\simp\FF$. Hence the map $\simp\FF \to \top\FF$ is a surjection.

  Next, we will show that this map is an injection. Suppose that this is not the case. Then there exist a simple module $S$ and $0 < i < j$ such that both $S^{(i)}$ and $S^{(j)}$ are simple in $\FF$. Then we have a short exact sequence
  \[
  0 \to K \to S^{(j)} \to S^{(i)} \to 0
  \]
  in $\mod\Lambda$ by Proposition \ref{nakayamaprop} (3), and $K \neq 0$ since $i \neq j$.
  However, $K$ belongs to $\FF$ since $\FF$ is closed under submodules. Thus, the above is a conflation in $\FF$, which shows that $S^{(j)}$ is not a simple object in $\FF$. This is a contradiction, so $\simp\FF \to \top\FF$ is an injection.

  It is clear from the above argument that the inverse of the map $\top \colon \simp\FF \to \top \FF$ is given by $S \mapsto S^{(m)}$ as claimed.

  ($\top \colon \ind\P(\FF) \to \top\FF$ is a bijection):
  First, we will show that this map is a surjection. Let $M$ be an indecomposable object in $\FF$ and put $S:= \top M$. Since $\FF$ satisfies Assumption \ref{assumperp}, it has enough projectives by Proposition \ref{assumprop}
  \footnote{Another way to show this is to use \cite[Corollary 3.15]{en2}: Every Hom-finite Krull-Schmidt exact $k$-category has enough projectives if it has only finitely many indecomposables.}. Thus, there exists a deflation
  \[
  P_1 \oplus P_2 \oplus \cdots \oplus P_l \defl M
  \]
  in $\FF$, where $P_i \in \ind\P(\FF)$ for each $i$. This map is a surjection in $\mod\Lambda$, so it induces a surjection
  \[
  \top P_1 \oplus \top P_2 \cdots \oplus \top P_l \defl \top M = S.
  \]
  It follows that $\top P_i = S$ holds for some $i$. This means that the map $\ind\P(\FF) \to \top\FF$ is surjective.

  Next, we will show that this map is an injection. Suppose that this is not the case. Then there exists a simple module $S$ and $0 < i < j$ such that both $S^{(i)}$ and $S^{(j)}$ are projective objects in $\FF$. Now Proposition \ref{nakayamaprop} (3) implies that there exists a short exact sequence
  \[
  0 \to K \to S^{(j)} \to S^{(i)} \to 0
  \]
  in $\mod\Lambda$. Since $\FF$ is closed under submodules, this is a conflation in $\FF$, and since $i < j$, we have $K \neq 0$. However, the projectivity of $S^{(i)}$ implies that the above sequence splits, which is a contradiction since $S^{(j)}$ is indecomposable. Thus, $\top \colon \ind \P(\FF) \to \top \FF$ is an injection.

  Finally, we shall describe the inverse of $\top \colon \ind \P(\FF) \to \top\FF$.
  For an object $S \in \top \FF$, put $n := \max \{ i \,\,|\,\, S^{(i)} \in \FF \}$, and we claim that the map $S \mapsto S^{(n)}$ is the inverse.
  Take any $1 \leq i < n$ with $S^{(i)} \in \FF$. Then as in the proof of injectivity, we have the following conflation in $\FF$
  \[
  0 \to K \to S^{(n)} \to S^{(i)} \to 0
  \]
  with $K \neq 0$, and this sequence does not split since $S^{(n)}$ is indecomposable. Thus, $S^{(i)}$ cannot be projective in $\FF$. By this fact and the fact that $\top \colon \ind\P(\FF) \to \top \FF$ is surjective, we must have that the inverse of this map is given by $S \mapsto S^{(n)}$.
\end{proof}
As an immediate corollary, we obtain the following result.
\begin{corollary}\label{nakayamamain}
  Let $\Lambda$ be a Nakayama algebra. Then every torsion-free class and every torsion class in $\mod\Lambda$ satisfies (JHP).
\end{corollary}
\begin{proof}
  Let $\FF$ be a torsion-free class in $\mod\Lambda$. By Theorem \ref{nakasimp}, we have an equality $\# \simp\FF = \# \ind \P(\FF)$. Thus, $\FF$ satisfies (JHP) by Theorem \ref{artinmain}.
  The assertion for a torsion class follows from the standard duality $D \colon \mod\Lambda \leftrightarrow \mod\Lambda^{\op}$, since $\Lambda^{\op}$ is also Nakayama and an exact category satisfies (JHP) if so is its opposite category.
\end{proof}

\section{Torsion-free classes of type A quivers and Bruhat inversions}\label{typeasec}
In this section, we investigate simple objects in a torsion-free class of the category of representations of a quiver of type A by using the combinatorics of the symmetric group. For a quiver $Q$, we denote by $k Q$ the path algebra of $Q$ over a field $k$. As usual, we identify representations of $Q$ with right $kQ$-modules.

For an acyclic quiver $Q$, a classification of torsion-free classes of $\mod kQ$ with finitely many indecomposables is known: they are in bijection with so called \emph{$c$-sortable elements} of the Coxeter group of $Q$ (\cite{it} for the Dynkin case and \cite{airt,thomas} for the general case).
For a quiver of type A, the corresponding Coxeter group is just a symmetric group, and we can describe all indecomposable $kQ$-modules in a quite explicit way. In this section, we freely use these description particular to type A, but in a forthcoming paper \cite{enforth}, we will show that the results are valid in other Dynkin types, or more generally, preprojective algebras of Dynkin types (see Remark \ref{forthcoming}).
\subsection{Bruhat inversions of elements in the symmetric group}
First, we recall combinatorial notions on the symmetric groups we need later. These notions are well-studied in the context of Coxeter groups, but we will give an explicit description for type A case here for the convenience of the reader.
The standard reference is \cite{bb}.

We denote by $S_{n+1}$ the symmetric group which acts on the set $\{1,2, \cdots, n, n+1 \}$ from left.
We often use the \emph{one-line notation} to represent elements of $S_{n+1}$, that is, we write $w = w(1)w(2) \cdots w(n+1)$ for $w\in S_{n+1}$.
We denote by $(i \,\, j)$ for $1 \leq i,j \leq n+1$ the transposition of the letters $i$ and $j$, and write $T$ for the set of all transpositions in $S_{n+1}$. Then $S_{n+1}$ is generated by the \emph{simple reflections} $s_i:= (i \,\, i+1)$ for $1 \leq i \leq n$. We write $S$ for the set of all simple reflections.
For example, we have $s_2 s_1 s_3 s_2 = 3412$ in $S_4$.
Note that for a transposition $t = (i \,\,j)$ and $w \in S_{n+1}$, the element $t w $ is obtained by interchanging two letters $i$ and $j$ in the one-line notation for $w$, e.g. we have $(3 \,\, 4) \cdot 3412 = 4312$.

Each element $w \in S_{n+1}$ can be written as a product of simple reflections:
\[
w = s_{i_1} s_{i_2} \cdots s_{i_l}
\]
This expression of $w$ is called \emph{reduced} if $l$ is the minimal among all such expressions. In this case, we call $l$ the \emph{length} of $w$ and write $\ell(w) := l$.

For an element $w \in S_{n+1}$, a transposition $t\in T$ is called an \emph{inversion of $w$} if $\ell(t w) < \ell(w)$ holds, and we denote by $\inv (w)$ the set of all inversions of $w$. It is known that a transposition $(i \,\, j)$ with $i < j $ is an inversion of $w$ if and only if $j$ precedes $i$ in the one-line notation for $w$, that is, $w^{-1}(i) > w^{-1}(j)$. It is also known that $\ell(w) = \# \inv (w)$ holds.

\begin{example}\label{exinv}
  The following are examples of inversions in $S_5$.
  \begin{enumerate}
    \item For $w_1 = s_1 s_3 s_2 s_4 s_1 s_3 s_2 s_4 = 45231$ we have:
    \[
    \inv(w_1) = \{ (1\,\,2), \, (1\,\,3), \, (1\,\,4), \, (1\,\,5), \, (2\,\,4), \, (2\,\,5), \, (3\,\,4), \, (3\,\,5)\}.
    \]
    \item For $w_2 = s_1 s_3 s_2 s_4 s_1 s_3 s_2 s_1 = 54213$ we have:
    \[
    \inv(w_2) = \{ (1\,\,2), \, (1\,\,4),\, (1\,\,5),\, (2\,\,4),\,(2\,\,5),\, (3\,\,4),\,(3\,\,5),\, (4\,\,5) \}
    \]
  \end{enumerate}
\end{example}

The following class of inversions plays an important role in this paper, since we shall see that this corresponds to simple objects in a torsion-free class.
\begin{definition}\label{binvdef}
  We say that an inversion $t$ of an element $w \in S_{n+1}$ is a \emph{Bruhat inversion of $w$} if it satisfies $\ell(t w) = \ell(w) - 1$. We denote by $\Binv(w)$ the set of Bruhat inversions of $w$.
\end{definition}
We can interpret Bruhat inversions in terms of the \emph{cover relation of the Bruhat order}.
Recall that the Bruhat order on $S_{n+1}$ is a partial order $\leq$ generated by the following relation: \emph{for every $t\in T$ and $w \in S_{n+1}$ with $\ell(tw) < \ell(w)$, we have that $t w < w$ holds}. In what follows, we always denote by $\leq$ (and $<$) the Bruhat order on $S_{n+1}$.

\begin{lemma}[{\cite[Lemma 2.1.4]{bb}}]\label{binvlem}
  For a transposition $t = (i \,\, j) \in T$ with $i < j$ and an element $w \in S_{n+1}$, the following are equivalent:
  \begin{enumerate}
    \item $t$ is a Bruhat inversion of $w$.
    \item $tw$ is covered by $w$ in the Bruhat order, that is, $tw < w$ holds and there exists no element $u\in S_{n+1}$ satisfying $tw < u < w$.
    \item In the one-line notation for $w$, the letter $j$ precedes $i$, and there exists no $l$ with $i < l < j$ such that the letter $l$ appears between $j$ and $i$.
    \item $(i\,\,j) \in \inv(w)$ and there exists no $l$ with $i < l < j$ satisfying $(i\,\,l), (l\,\,j) \in \inv(w)$.
  \end{enumerate}
\end{lemma}

\begin{example}
  The following are Bruhat inversions of elements in Example \ref{exinv}.
  \begin{enumerate}
    \item For $w_1 = 45231$, we have:
    \[
    \Binv(w_1) = \{ (1\,\,2), \, (1\,\,3), \, (2\,\,4), \, (2\,\,5), \, (3\,\,4), \, (3\,\,5)\}.
    \]
    \item For $w_2 = 54213$, we have:
    \[
    \Binv(w_2) = \{ (1\,\,2), \, (2\,\,4),\, (3\,\,4),\, (4\,\,5) \}
    \]
  \end{enumerate}
  Note that $\# \inv(w_1) = \# \inv(w_2) = 8$, but $\Binv(w_1) = 6 > 4 = \#\Binv(w_2)$.
\end{example}

\begin{remark}\label{rem:coverref}
  There is a notion which is related to but different from Bruhat inversions, \emph{cover reflections} (see e.g. \cite[Section 1]{reading}). A transposition $t = (i \,\, j) \in T$ with $i < j$ is called a \emph{cover reflection} of $w \in S_{n+1}$ if $tw$ is covered by $w$ in the \emph{right weak order} on $S_{n+1}$, or equivalently $\cdots j i \cdots$ appears in the one-line notation for $w$.
  For example, cover reflections for $45231$ are $(2 \,\, 5)$ and $(1 \, \, 3)$.
  In view of Lemma \ref{binvlem}, the difference of these two concepts is that which poset structure we consider on $S_{n+1}$, the Bruhat order or the right weak order.
  It is immediate from definitions that the set of cover reflections for $w$ is a subset of $\Binv(w)$ (a proper subset in general). See Remark \ref{rem:wide} for a representation-theoretic interpretation of the relation between these two sets.
\end{remark}

We will use the notion of support of elements in $S_{n+1}$.
\begin{definition}
  Let $w$ be an element of $S_{n+1}$. Then $i \in \{ 1, \cdots, n \}$ is called a \emph{support} of $w$ if there exists some reduced expression of $w$ which contains $s_i$. We denote by $\supp(w)$ the set of all supports of $w$. We say that $w$ \emph{has full support} if $\supp(w) = \{ 1, 2, \cdots, n \}$.
\end{definition}
In fact, if $i$ is in $\supp(w)$, then any reduced expression of $w$ contains $s_i$ (\cite[Corollary 1.4.8]{bb}). We will use the following characterization of the support later.
\begin{lemma}\label{supplem}
  For $w \in S_{n+1}$ and $i \in \{1, \cdots, n \}$, the following are equivalent:
  \begin{enumerate}
    \item $i \in \supp(w)$.
    \item There exists some $j$ with $1 \leq j \leq i$ such that $w^{-1}(j) > i$, that is, $j$ does not appear in the initial segment of length $i$ in the one-line notation for $w$.
    \item There exists some $j$ with $1 \leq j \leq i$ such that $w(j) > i$.
    \item There exists some $l$ with $i < l$ such that $l$ appears in the initial segment of length $i$ in the one-line notation for $w$.
  \end{enumerate}
\end{lemma}
For example, by Using this criterion, we can easily check that $\supp(45231) = \{1,2,3,4,5\}$, $\supp(21543) = \{1,3,4\}$, and $\supp(12543)= \{3,4\}$.

\subsection{Coxeter element, c-sortable elements and torsion-free classes}
We will describe the \emph{Ingalls-Thomas bijection} \cite{it} between sortable elements and torsion-free classes for a quiver of type A, following \cite{thomas}.
In what follows, $Q$ is a quiver of type A with $n$ vertices, whose underlying graph is $\begin{tikzcd}
    1 \rar[dash] & 2 \rar[dash] & \cdots \rar[dash] & n
  \end{tikzcd}$.
As usual, we identify right $kQ$-modules over the path algebra $Q$ and representations of $Q$. For a $kQ$-module $M$ and a vertex $i$ of $Q$, we denote by $M_i$ the vector space attached to $i$.

A \emph{Coxeter element} of $S_{n+1}$ is an element $c \in S_{n+1}$ which is obtained as the product of all simple reflections $s_1, \cdots, s_n \in S$ in some order, or equivalently, an element with length $n$ which has full support.
We say that a Coxeter element $c$ is \emph{associated to $Q$} if $s_i$ appears before $s_j$ in $c$ whenever there exists an arrow $i \ot j$ in $Q$. It is known that Coxeter elements are in bijection with orientations of edges in the underlying graph of $Q$.

\begin{example}
  We give an example of the correspondence between Coxeter elements in $S_4$ and orientations of $A_3$ quiver.
  \[
  \begin{array}{c|c}
    \text{Orientations of $A_3$} & \text{Coxeter elements in $S_4$} \\ \hline
    1 \ot 2 \ot 3 & s_1 s_2 s_3 = 2341 \\
    1 \ot 2 \to 3 & s_1 s_3 s_2 = s_3 s_1 s_2 = 2413 \\
    1 \to 2 \ot 3 & s_2 s_1 s_3 = s_2 s_3 s_1 = 3142 \\
    1 \to 2 \to 3 & s_3 s_2 s_1 = 4123
    \end{array}
  \]
\end{example}
In what follows, we adopt the following convention:
\begin{assumption}\label{Aquiver}
  $Q$ is a quiver of type A with $n$ vertices, whose underlying graph is given by
  \[
  \begin{tikzcd}
    1 \rar[dash] & 2 \rar[dash] & \cdots \rar[dash] & n
  \end{tikzcd}
  \]
  and $c$ is the Coxeter element of $S_{n+1}$ associated with $Q$.
\end{assumption}

Torsion-free classes of $\mod kQ$ are classified by the combinatorial notion called \emph{$c$-sortable elements}, which was introduced by Reading \cite{reading}.

\begin{definition}
  Let $c$ be a Coxeter element of $S_{n+1}$. We say that an element $w$ of $S_{n+1}$ is \emph{$c$-sortable} if there exists a reduced expression of the form $w = c^{(0)} c^{(1)} \cdots c^{(m)}$ such that each $c^{(i)}$ is a subword of $c$ satisfying $\supp(c^{(0)}) \supset \supp(c^{(1)}) \supset \cdots \supset \supp(c^{(m)}).$
\end{definition}

Now we can state the correspondence between $c$-sortable elements and torsion-free classes. A \emph{support} of a module $M \in \mod kQ$ is a vertex $i \in Q$ with $M_i \neq 0$, and we denote by $\hc{i,j}$ the set of elements $l \in \{1,2,\cdots,n\}$ with $i \leq l < j$. The following is just a restatement of the well-known classification of indecomposable representations of $Q$.
\begin{proposition}\label{indectrans}
  Let $Q$ be a quiver as in Assumption \ref{Aquiver}. For a transposition $(i \,\, j) \in T$ with $i < j$, there exists a unique indecomposable module $M:= M_{\hc{i,j}}$ in $\mod kQ$ such that the set of all its supports is $\hc{i,j}$. This module is a representation of $Q$ defined by the following:
  \begin{itemize}
    \item For vertices, $M_l = k$ if $l \in \hc{i,j}$ and $M_l =0$ otherwise.
    \item For arrows $i \to j$ in $Q$, we put $\id_k \colon k \to k$ if $M_i = M_j = k$, and $0$ otherwise.
  \end{itemize}
  Moreover, this give a bijection $M_{\hc{\;}} \colon T \xrightarrow{\sim} \ind (\mod kQ)$ which restricts to $S \xrightarrow{\sim} \simp (\mod\Lambda)$.
\end{proposition}

\begin{example}
  Let $Q$ be the quiver $Q = 1 \to 2 \ot 3$. Then the Auslander-Reiten quiver of $\mod kQ$ is as follows, where indecomposables are labelled according to Proposition \ref{indectrans}.
  \[
    \begin{tikzpicture}
      [scale=0.8,  every node/.style={scale=.9}]
      \node (12) at (3,2) {$M_{\hc{1,2}}$};
      \node (13) at (1,0) {$M_{\hc{1,3}}$};
      \node (14) at (2,1) {$M_{\hc{1,4}}$};
      \node (23) at (0,1) {$M_{\hc{2,3}}$};
      \node (24) at (1,2) {$M_{\hc{2,4}}$};
      \node (34) at (3,0) {$M_{\hc{3,4}}$};

      \draw[->] (23) -- (24);
      \draw[->] (23) -- (13);
      \draw[->] (24) -- (14);
      \draw[->] (13) -- (14);
      \draw[->] (14) -- (12);
      \draw[->] (14) -- (34);
    \end{tikzpicture}
  \]
\end{example}
Now torsion-free classes of $\mod kQ$ are classified as follows. Here a \emph{support} of a subcategory $\FF$ of $\mod kQ$ is a vertex of $Q$ which is a support of some module in $\FF$.
\begin{theorem}[{\cite[Theorem 4.2]{thomas}}]\label{sortabletorf}
  Let $Q$ and $c$ be as in Assumption \ref{Aquiver}. For $w \in S_{n+1}$, define the subcategory $\FF(w)$ of $\mod kQ$ by
  \[
  \FF(w):= \add \{ M_{\hc{i,j}} \; | \; (i \,\, j) \in \inv (w) \text{ with $i < j$} \}.
  \]
  then the following hold.
  \begin{enumerate}
    \item If $w$ is $c$-sortable, then $\FF(w)$ is a torsion-free class of $\mod kQ$, and the bijection $T \xrightarrow{\sim} \ind(\mod kQ)$ given by $(i\,\,j) \mapsto M_{\hc{i,j}}$ restricts a bijection $\inv(w) \xrightarrow{\sim} \ind \FF(w)$.
    \item The map $w \mapsto \FF(w)$ gives a bijection between $c$-sortable elements of $S_{n+1}$ and torsion-free classes of $\mod kQ$.
    \item Let $w$ be a $c$-sortable element. Then $\supp(w)$ coincides with the set of all supports of $\FF(w)$, and the equality $\# \supp(w) = \#\ind \P(\FF(w))= \# \ind \I(\FF(w))$ holds.
  \end{enumerate}
\end{theorem}
\begin{proof}
  The proof is essentially contained in other references such as \cite{thomas}, but we will clarify the relation between our convention and others, since others often use the language of root systems.

  We freely use basics of the root system. Let $\Phi$ be a root system of type $A_n$ whose Dynkin graph is as in Assumption \ref{Aquiver}, and let $e_1, \cdots, e_n$ be the set of simple roots corresponding to vertices $1, \cdots, n$. We identify $S_{n+1}$ with the Weyl group $W$ of $\Phi$ as usual.
  Under this identification, a transposition $(i\,\,j)$ of $S_{n+1}$ corresponds to a reflection with respect to a positive root $\alpha_{\hc{i,j}}:= \sum_{l \in \hc{i,j}}e_l$. By this, the set of transpositions are in bijection with the set of positive roots.

  For a $kQ$-module $M$, we define the dimension vector $\udim M$ by $\udim M := \sum_l \dim_k (M_l) e_l$. Then $M_{\hc{i,j}}$ in Proposition \ref{indectrans} satisfies $\udim M_{\hc{i,j}} = \alpha_{\hc{i,j}}$.

  The following description of $\inv(w)$ is well-known, see e.g. \cite[Proposition 4.4.6]{bb}:
  \begin{lemma}\label{invlemma}
    Let $t\in S_{n+1} = W$ be a transposition, $\alpha \in \Phi$ the positive root corresponding to $t$ and $w= s_{i_1}s_{i_2}\cdots s_{i_r}$ a reduced expression of arbitrary $w\in W$. Then $t \in \inv(w)$ if and only if $\alpha = s_{i_1} s_{i_2} \cdots s_{i_{l-1}}(e_{i_l})$ for some $1 \leq l \leq r$.
  \end{lemma}
  Now return to our proof. (1) and (2) follow from \cite[Theorem 4.2]{thomas} and our above identification. Let us prove (3).

  Let $w$ be a $c$-sortable element.
  By Example \ref{perpex}, the equality $\#\ind\P(\FF(w)) = \#\ind \I(\FF(w))$ holds. On the other hand, the number of supports of $\FF(w)$ is equal to $\#\ind\I(\FF(w))$ by the general theory of support ($\tau$-)tilting theory, see e.g. \cite[Proposition 2.5(4), Lemma 2.9]{it} or \cite[Theorem 2.7]{air}. Thus, it suffices to show that the number of supports of $\FF(w)$ is equal to $\#\supp(w)$.

  Take a reduced expression $w = c^{(0)} c^{(1)} \cdots c^{(m)}$ such that $c^{(i)}$ is a subword of $c$ and satisfies $\supp(c^{(0)}) \supset \supp(c^{(1)}) \supset \cdots \supset \supp(c^{(m)})$.
  First, suppose that $l \in \supp(w)$, which means that $s_l$ appears in the word $c^{(0)}$. Then we can write the reduced expression of $c^{(0)}$ as  $c^{(0)} = c' s_l c''$ such that $c'$ (and $c''$) does not contain the letter $s_l$.
  Let $(i\,\,j)$ be a transposition which corresponds to a positive root $\alpha := c'(e_l)$, then it belongs to $\inv(w)$ by Lemma \ref{invlemma}. Clearly the $e_l$-component of $\alpha$ is equal to $1$, thus $l$ is a support of $M_{\hc{i,j}}$ since $\udim M_{\hc{i,j}} = \alpha$.

  Conversely, suppose that $l$ is a support of $M:= M_{\hc{i,j}}$ for some $(i\,\,j) \in \inv(w)$. Take a positive root which corresponds to $(i\,\,j)$. Then we have $\udim M = \alpha$, thus the $e_l$-component of $\alpha$ is strictly positive. On the other hand, Lemma \ref{invlemma} shows that $l$ must be appear in a reduced expression of $w$, thus $l \in \supp(w)$.
\end{proof}

\subsection{Bruhat inversions and simples}
Our main result in this section is the following, which establishes a bijection between simples in $\FF(w)$ and Bruhat inversions of $w$.
\begin{theorem}\label{binvsimp}
  Let $Q$ and $c$ be as in Assumption \ref{Aquiver} and $w$ a $c$-sortable element of $S_{n+1}$. Then the natural bijection $\inv(w) \xrightarrow{\sim} \ind \FF(w)$ restricts to a bijection $\Binv(w) \to \simp \FF(w)$.
  In other words, for $(i \,\, j)\in \inv(w)$ with $i < j$, the object $M_{\hc{i,j}}$ is simple in $\FF(w)$ if and only if $(i \,\, j)$ is a Bruhat inversion of $w$.
\end{theorem}
Before we give a proof, let us state the immediate consequence of this, which characterizes (JHP) in a purely combinatorial way. See Example \ref{main1ex} for the actual example of Theorem \ref{binvsimp} and Corollary \ref{jhtypea}.
\begin{corollary}\label{jhtypea}
  Let $Q$ and $c$ be as in Assumption \ref{Aquiver} and $w$ a $c$-sortable element of $S_{n+1}$. Then $\FF(w)$ satisfies (JHP) if and only if $\# \supp(w) = \# \Binv(w)$ holds.
\end{corollary}
\begin{proof}
  Theorem \ref{binvsimp} implies $\#\simp \FF(w) = \#\Binv(w)$. On the other hand, Theorem \ref{sortabletorf} implies $\# \ind\P(\FF(w)) = \#\supp(w)$. Thus, our assertion follows from Theorem \ref{artinmain}.
\end{proof}
To prove Theorem \ref{binvsimp}, we use the following explicit exact sequences in $\mod kQ$.
\begin{lemma}\label{exlemma}
  Let $(i\,\,j)$ be a transposition in $S_{n+1}$ with $i < j$ and consider the $kQ$-module $M_{\hc{i,j}}$ given in Proposition \ref{indectrans}. Then the following hold.
  \begin{enumerate}
    \item For each $l$ with $i < l < j$, one of the following two exact sequences exists:
    \begin{align}
      \text{either }\quad &0 \to M_{\hc{i,l}} \to M_{\hc{i,j}} \to M_{\hc{l,j}} \to 0, \label{ex1}\\
      \text{or }\quad &0 \to M_{\hc{l,j}} \to M_{\hc{i,j}} \to M_{\hc{i,l}} \to 0. \label{ex2}
    \end{align}
    \item Suppose that we have a monomorphism $M_{\hc{l,l'}} \hookrightarrow M_{\hc{i,j}}$ for some $l < l'$. Then we have $\hc{l,l'} \subset \hc{i,j}$, and there exists an exact sequence
    \begin{equation}\label{ex3}
    0 \to M_{\hc{l,l'}} \to M_{\hc{i,j}} \to M_{\hc{i,l}} \oplus M_{\hc{l',j}} \to 0.
    \end{equation}
    Moreover, there exist the following two exact sequences:
    \begin{align}
      &0 \to M_{\hc{i,l'}} \to M_{\hc{i,j}} \to M_{\hc{l',j}} \to 0, \label{ex4}\\
      \text{and }\quad &0 \to M_{\hc{l,j}} \to M_{\hc{i,j}} \to M_{\hc{i,l}} \to 0.\label{ex5}
    \end{align}
    Here we put $M_{\hc{a,a}} = 0$ for $a \in \{ 1, \cdots, n+1\}$.
  \end{enumerate}
\end{lemma}
\begin{proof}
  (1)
  Consider the orientation of the edge between $l-1$ and $l$ in $Q$. Suppose that we have $l-1 \ot l$ in $Q$.
  A subspace $M_{\hc{i,l}}$ of $M_{\hc{i,j}}$ is closed under the action of $Q$ in $M_{\hc{i,j}}$ because there exists no path which starts from one of $\hc{i,l}$ and ends at one of $\hc{l,j}$.
  Thus, $M_{\hc{i,l}}$ is a submodule of $M_{\hc{i,j}}$, and we obtain (\ref{ex1}).
  By the same reason, if we have $l-1 \to l$ in $Q$, then we have (\ref{ex2}).

  (2)
  Since $M_{\hc{l,l'}}$ must be a subspace of $M_{\hc{i,j}}$, it is clear that $\hc{l,l'} \subset \hc{i,j}$ holds.
  The set of all supports of the quotient $M_{\hc{i,j}}/ M_{\hc{l,l'}}$ is a disjoint union $\hc{i,l} \sqcup \hc{l',j}$, and it is easily checked that the action of $kQ$ on this quotient coincides with that on $M_{\hc{i,l}} \oplus M_{\hc{l',j}}$. Thus, we obtain (\ref{ex3}).
  Moreover, since $M_{\hc{l,l'}}$ is closed under actions of $kQ$ in $M_{\hc{i,j}}$, we must have that either $i= l$ or $l-1 \to l$ in $Q$ holds, and that either $l'=j$ or $l'-1 \ot l'$ in $Q$ holds. Thus, the existence of two exact sequences (\ref{ex4}) and (\ref{ex5}) follows from the proof of (1).
\end{proof}

\begin{proof}[Proof of Theorem \ref{binvsimp}]
  Let $(i\,\,j)$ be an inversion of $w$ with $i< j$. Since any simple object is indecomposable, it suffices to show that $(i\,\,j)$ is a Bruhat inversion of $w$ if and only if $M_{\hc{i,j}}$ is simple in $\FF(w)$.

  First, suppose that $(i\,\,j)$ is \emph{not} a Bruhat inversion of $w$. Then by Proposition \ref{binvlem}, there exists some $l$ with $i < l < j$ such that both $(i\,\,l)$ and $(l\,\,j)$ belong to $\inv(w)$. Thus, both $M_{\hc{i,l}}$ and $M_{\hc{l,j}}$ belong to $\FF(w)$.
  Now we have an exact sequence (\ref{ex1}) or (\ref{ex2}) by Lemma \ref{exlemma}. In either case, this gives a conflation in $\FF(w)$, and since $M_{\hc{i,l}}$ and $M_{\hc{l,j}}$ are non-zero, $M_{\hc{i,j}}$ is \emph{not} a simple object in $\FF(w)$.

  Conversely, suppose that $M_{\hc{i,j}}$ is \emph{not} a simple object in $\FF(w)$. Then we have a non-isomorphic inflation $N \infl M_{\hc{i,j}}$ in $\FF(w)$ for some object $0 \neq N$ of $\FF(w)$.
  Take an indecomposable direct summand $M_{\hc{l,l'}}$ of $N$. Then the composition $M_{\hc{l,l'}} \hookrightarrow N \infl M_{\hc{i,j}}$ is a non-isomorphic inflation in $\FF(w)$, because the section $M_{\hc{l,l'}} \hookrightarrow N$ is an inflation and inflations are closed under compositions.
  Now Lemma \ref{exlemma} (2) tells us that $\hc{l,l'} \subset \hc{i,j}$ and that $M_{\hc{i,l}} \oplus M_{\hc{l',j}}$ belongs to $\FF(w)$, hence both $M_{\hc{i,l}}$ and $M_{\hc{l',j}}$ belong to $\FF(w)$ since $\FF(w)$ is closed under direct summands.

  Since $M_{\hc{l,l'}} \hookrightarrow M_{\hc{i,j}}$ is not an isomorphism, we have $i < l$ or $l' < j$.
  Suppose that the former holds.
  Then since $M_{\hc{i,l}} \in \FF(w)$, the transposition $(i\,\,l)$ is an inversion of $w$.
  On the other hand, the exact sequence (\ref{ex5}) of Lemma \ref{exlemma} implies that $M_{\hc{l,j}}$ belongs to $\FF(w)$, since $\FF(w)$ is closed under submodules. Thus, $(l\,\,j)$ is also an inversion of $w$. This implies that $(i\,\,j)$ is \emph{not} a Bruhat inversion by Lemma \ref{binvlem}. The case $i < l' < j$ is completely similar, except we use (\ref{ex4}) instead of (\ref{ex5}).
\end{proof}

\begin{example}\label{main1ex}
  Let us look at several examples.
  \begin{enumerate}
    \item Let $Q$ be the quiver $Q = 1 \to 2 \ot 3$. Then the Coxeter element associated with $Q$ is $c = s_2 s_1 s_3 = 3142 \in S_4$.
    In Table \ref{213ex}, we list all $c$-sortable elements and the corresponding inversions, Bruhat inversions and torsion-free classes. The black vertices indicate simple objects in $\FF(w)$ and the white ones indicate the rest of indecomposables in $\FF(w)$.
    \begin{table}
      \caption{Example of Theorem \ref{binvsimp} for $Q = 1 \to 2 \ot 3$}
      \label{213ex}
      \begin{tabular}{C|C|C|C|c}
        \text{$c$-sortable elements $w$} & \supp(w) & \inv(w) & \Binv(w) & $\FF(w)$ \\ \hline \hline
        e = 1234 & \varnothing & \varnothing & \varnothing &
        \begin{tikzpicture}
          [baseline={([yshift=-.5ex]current bounding box.center)}, scale=0.4,  every node/.style={scale=0.5}]
          \node (12) at (3,2) {};
          \node (13) at (1,0) {};
          \node (14) at (2,1) {};
          \node (23) at (0,1) {};
          \node (24) at (1,2) {};
          \node (34) at (3,0) {};

          \draw[->] (23) -- (24);
          \draw[->] (23) -- (13);
          \draw[->] (24) -- (14);
          \draw[->] (13) -- (14);
          \draw[->] (14) -- (12);
          \draw[->] (14) -- (34);
        \end{tikzpicture}
        \\ \hline
        s_2 = 1324 & 2 & (2\,\,3) & (2\,\,3) &
        \begin{tikzpicture}
          [baseline={([yshift=-.5ex]current bounding box.center)}, scale=0.4,  every node/.style={scale=0.5}]
          \node (12) at (3,2) {};
          \node (13) at (1,0) {};
          \node (14) at (2,1) {};
          \node (23) at (0,1)[black] {};
          \node (24) at (1,2) {};
          \node (34) at (3,0) {};

          \draw[->] (23) -- (24);
          \draw[->] (23) -- (13);
          \draw[->] (24) -- (14);
          \draw[->] (13) -- (14);
          \draw[->] (14) -- (12);
          \draw[->] (14) -- (34);
          \useasboundingbox ([shift={(.3,.3)}]current bounding box.north east) rectangle ([shift={(-.3,-.3)}]current bounding box.south west);
        \end{tikzpicture}
        \\ \hline
        s_1 = 2134 & 1 & (1\,\,2)  & (1\,\,2)  &
        \begin{tikzpicture}
          [baseline={([yshift=-.5ex]current bounding box.center)}, scale=0.4,  every node/.style={scale=0.5}]
          \node (12) at (3,2) [black] {};
          \node (13) at (1,0) {};
          \node (14) at (2,1) {};
          \node (23) at (0,1) {};
          \node (24) at (1,2) {};
          \node (34) at (3,0) {};

          \draw[->] (23) -- (24);
          \draw[->] (23) -- (13);
          \draw[->] (24) -- (14);
          \draw[->] (13) -- (14);
          \draw[->] (14) -- (12);
          \draw[->] (14) -- (34);
          \useasboundingbox ([shift={(.3,.3)}]current bounding box.north east) rectangle ([shift={(-.3,-.3)}]current bounding box.south west);
        \end{tikzpicture}
        \\ \hline
        s_3 = 1243 & 3 & (3\,\,4) & (3\,\,4) &
        \begin{tikzpicture}
          [baseline={([yshift=-.5ex]current bounding box.center)}, scale=0.4,  every node/.style={scale=0.5}]
          \node (12) at (3,2) {};
          \node (13) at (1,0) {};
          \node (14) at (2,1) {};
          \node (23) at (0,1) {};
          \node (24) at (1,2) {};
          \node (34) at (3,0)[black] {};

          \draw[->] (23) -- (24);
          \draw[->] (23) -- (13);
          \draw[->] (24) -- (14);
          \draw[->] (13) -- (14);
          \draw[->] (14) -- (12);
          \draw[->] (14) -- (34);
          \useasboundingbox ([shift={(.3,.3)}]current bounding box.north east) rectangle ([shift={(-.3,-.3)}]current bounding box.south west);
        \end{tikzpicture}
        \\ \hline
        s_2 s_1 = 3124 & 1, 2 & (1\,\,3), (2\,\,3) & (1\,\,3),(2\,\,3) &
        \begin{tikzpicture}
          [baseline={([yshift=-.5ex]current bounding box.center)}, scale=0.4,  every node/.style={scale=0.5}]
          \node (12) at (3,2) {};
          \node (13) at (1,0)[black] {};
          \node (14) at (2,1) {};
          \node (23) at (0,1)[black] {};
          \node (24) at (1,2) {};
          \node (34) at (3,0) {};

          \draw[->] (23) -- (24);
          \draw[->] (23) -- (13);
          \draw[->] (24) -- (14);
          \draw[->] (13) -- (14);
          \draw[->] (14) -- (12);
          \draw[->] (14) -- (34);
          \useasboundingbox ([shift={(.3,.3)}]current bounding box.north east) rectangle ([shift={(-.3,-.3)}]current bounding box.south west);
        \end{tikzpicture}
        \\ \hline
        s_2 s_3 = 1342 & 2,3 & (2\,\,3), (2\,\,4) & (2\,\,3), (2\,\,4) &
        \begin{tikzpicture}
          [baseline={([yshift=-.5ex]current bounding box.center)}, scale=0.4,  every node/.style={scale=0.5}]
          \node (12) at (3,2) {};
          \node (13) at (1,0) {};
          \node (14) at (2,1) {};
          \node (23) at (0,1)[black] {};
          \node (24) at (1,2)[black] {};
          \node (34) at (3,0) {};

          \draw[->] (23) -- (24);
          \draw[->] (23) -- (13);
          \draw[->] (24) -- (14);
          \draw[->] (13) -- (14);
          \draw[->] (14) -- (12);
          \draw[->] (14) -- (34);
          \useasboundingbox ([shift={(.3,.3)}]current bounding box.north east) rectangle ([shift={(-.3,-.3)}]current bounding box.south west);
        \end{tikzpicture}
        \\ \hline
        s_1 s_3 = 2143 & 1,3 & (1\,\,2), (3\,\,4)  & (1\,\,2),(3\,\,4)  &
        \begin{tikzpicture}
          [baseline={([yshift=-.5ex]current bounding box.center)}, scale=0.4,  every node/.style={scale=0.5}]
          \node (12) at (3,2) [black] {};
          \node (13) at (1,0) {};
          \node (14) at (2,1) {};
          \node (23) at (0,1) {};
          \node (24) at (1,2) {};
          \node (34) at (3,0)[black] {};

          \draw[->] (23) -- (24);
          \draw[->] (23) -- (13);
          \draw[->] (24) -- (14);
          \draw[->] (13) -- (14);
          \draw[->] (14) -- (12);
          \draw[->] (14) -- (34);
          \useasboundingbox ([shift={(.3,.3)}]current bounding box.north east) rectangle ([shift={(-.3,-.3)}]current bounding box.south west);
        \end{tikzpicture}
        \\ \hline
        s_2 s_1 s_2 = 3214 & 1, 2 & (1\,\,2), (1\,\,3), (2\,\,3) & (1\,\,2),(2\,\,3) &
        \begin{tikzpicture}
          [baseline={([yshift=-.5ex]current bounding box.center)}, scale=0.4,  every node/.style={scale=0.5}]
          \node (12) at (3,2)[black] {};
          \node (13) at (1,0)[white] {};
          \node (14) at (2,1) {};
          \node (23) at (0,1)[black] {};
          \node (24) at (1,2) {};
          \node (34) at (3,0) {};

          \draw[->] (23) -- (24);
          \draw[->] (23) -- (13);
          \draw[->] (24) -- (14);
          \draw[->] (13) -- (14);
          \draw[->] (14) -- (12);
          \draw[->] (14) -- (34);
          \useasboundingbox ([shift={(.3,.3)}]current bounding box.north east) rectangle ([shift={(-.3,-.3)}]current bounding box.south west);
        \end{tikzpicture}
        \\ \hline
        s_2 s_3 s_2= 1432 & 2,3 & (2\,\,3), (2\,\,4), (3\,\,4) & (2\,\,3), (3\,\,4) &
        \begin{tikzpicture}
          [baseline={([yshift=-.5ex]current bounding box.center)}, scale=0.4,  every node/.style={scale=0.5}]
          \node (12) at (3,2) {};
          \node (13) at (1,0) {};
          \node (14) at (2,1) {};
          \node (23) at (0,1)[black] {};
          \node (24) at (1,2)[white] {};
          \node (34) at (3,0)[black] {};

          \draw[->] (23) -- (24);
          \draw[->] (23) -- (13);
          \draw[->] (24) -- (14);
          \draw[->] (13) -- (14);
          \draw[->] (14) -- (12);
          \draw[->] (14) -- (34);
          \useasboundingbox ([shift={(.3,.3)}]current bounding box.north east) rectangle ([shift={(-.3,-.3)}]current bounding box.south west);
        \end{tikzpicture}
        \\ \hline
        s_2 s_3 s_1 = 3142 & 1,2,3 & (1\,\,3), (2\,\,3), (2\,\,4) & (1\,\,3), (2\,\,3), (2\,\,4) &
        \begin{tikzpicture}
          [baseline={([yshift=-.5ex]current bounding box.center)}, scale=0.4,  every node/.style={scale=0.5}]
          \node (12) at (3,2) {};
          \node (13) at (1,0)[black] {};
          \node (14) at (2,1) {};
          \node (23) at (0,1)[black] {};
          \node (24) at (1,2)[black] {};
          \node (34) at (3,0) {};

          \draw[->] (23) -- (24);
          \draw[->] (23) -- (13);
          \draw[->] (24) -- (14);
          \draw[->] (13) -- (14);
          \draw[->] (14) -- (12);
          \draw[->] (14) -- (34);
          \useasboundingbox ([shift={(.3,.3)}]current bounding box.north east) rectangle ([shift={(-.3,-.3)}]current bounding box.south west);
        \end{tikzpicture}
        \\ \hline
        s_2 s_3 s_1 s_2 = 3412 & 1,2,3 & \makecell{(1\,\,3), (1\,\,4),\\ (2\,\,3), (2\,\,4)} & \makecell{(1\,\,3), (1\,\,4),\\ (2\,\,3), (2\,\,4)} &
        \begin{tikzpicture}
          [baseline={([yshift=-.5ex]current bounding box.center)}, scale=0.4,  every node/.style={scale=0.5}]
          \node (12) at (3,2) {};
          \node (13) at (1,0)[black] {};
          \node (14) at (2,1)[black] {};
          \node (23) at (0,1)[black] {};
          \node (24) at (1,2)[black] {};
          \node (34) at (3,0) {};

          \draw[->] (23) -- (24);
          \draw[->] (23) -- (13);
          \draw[->] (24) -- (14);
          \draw[->] (13) -- (14);
          \draw[->] (14) -- (12);
          \draw[->] (14) -- (34);
          \useasboundingbox ([shift={(.3,.3)}]current bounding box.north east) rectangle ([shift={(-.3,-.3)}]current bounding box.south west);
        \end{tikzpicture}
        \\ \hline
        s_2 s_3 s_1 s_2 s_1 = 4312 & 1,2,3 & \makecell{(1\,\,3), (1\,\,4),\\ (2\,\,3), (2\,\,4), (3\,\,4)} & (1\,\,3), (2\,\,3), (3\,\,4) &
        \begin{tikzpicture}
          [baseline={([yshift=-.5ex]current bounding box.center)}, scale=0.4,  every node/.style={scale=0.5}]
          \node (12) at (3,2) {};
          \node (13) at (1,0)[black] {};
          \node (14) at (2,1)[white] {};
          \node (23) at (0,1)[black] {};
          \node (24) at (1,2)[white] {};
          \node (34) at (3,0)[black] {};

          \draw[->] (23) -- (24);
          \draw[->] (23) -- (13);
          \draw[->] (24) -- (14);
          \draw[->] (13) -- (14);
          \draw[->] (14) -- (12);
          \draw[->] (14) -- (34);
          \useasboundingbox ([shift={(.3,.3)}]current bounding box.north east) rectangle ([shift={(-.3,-.3)}]current bounding box.south west);
        \end{tikzpicture}
        \\ \hline
        s_2 s_3 s_1 s_2 s_3 = 3421 & 1,2,3 & \makecell{(1\,\,2), (1\,\,3), (1\,\,4),\\ (2\,\,3), (2\,\,4)} & (1\,\,2), (2\,\,3), (2\,\,4) &
        \begin{tikzpicture}
          [baseline={([yshift=-.5ex]current bounding box.center)}, scale=0.4,  every node/.style={scale=0.5}]
          \node (12) at (3,2)[black] {};
          \node (13) at (1,0)[white] {};
          \node (14) at (2,1)[white] {};
          \node (23) at (0,1)[black] {};
          \node (24) at (1,2)[black] {};
          \node (34) at (3,0) {};

          \draw[->] (23) -- (24);
          \draw[->] (23) -- (13);
          \draw[->] (24) -- (14);
          \draw[->] (13) -- (14);
          \draw[->] (14) -- (12);
          \draw[->] (14) -- (34);
          \useasboundingbox ([shift={(.3,.3)}]current bounding box.north east) rectangle ([shift={(-.3,-.3)}]current bounding box.south west);
        \end{tikzpicture}
        \\ \hline
        s_2 s_3 s_1 s_2 s_3 s_1= 4321 & 1,2,3 &
        \makecell{(1\,\,2), (1\,\,3), (1\,\,4),\\ (2\,\,3), (2\,\,4), (3\,\,4)} & (1\,\,2), (2\,\,3), (3\,\,4) &
        \begin{tikzpicture}
          [baseline={([yshift=-.5ex]current bounding box.center)}, scale=0.4,  every node/.style={scale=0.5}]
          \node (12) at (3,2)[black] {};
          \node (13) at (1,0)[white] {};
          \node (14) at (2,1)[white] {};
          \node (23) at (0,1)[black] {};
          \node (24) at (1,2)[white] {};
          \node (34) at (3,0)[black] {};

          \draw[->] (23) -- (24);
          \draw[->] (23) -- (13);
          \draw[->] (24) -- (14);
          \draw[->] (13) -- (14);
          \draw[->] (14) -- (12);
          \draw[->] (14) -- (34);
          \useasboundingbox ([shift={(.3,.3)}]current bounding box.north east) rectangle ([shift={(-.3,-.3)}]current bounding box.south west);
        \end{tikzpicture}
        \\ \hline
      \end{tabular}
    \end{table}
    According to this table, we conclude that $\FF(w)$ satisfies (JHP) except $w=3412$.
    \item Let $Q$ be the quiver $Q = 1 \ot 2 \to 3 \ot 4$. The associated Coxeter element is $c=s_1 s_3 s_2 s_4 = 24153 \in S_5$. The Auslander-Reiten quiver of $\mod kQ$ is as follows:
    \[
    \begin{tikzpicture}
      [scale=0.8,  every node/.style={scale=.9}]
      \node (12) at (0,0) {$M_{\hc{1,2}}$};
      \node (34) at (0,2) {$M_{\hc{3,4}}$};
      \node (14) at (1,1) {$M_{\hc{1,4}}$};
      \node (35) at (1,3) {$M_{\hc{3,5}}$};
      \node (24) at (2,0) {$M_{\hc{2,4}}$};
      \node (15) at (2,2) {$M_{\hc{1,5}}$};
      \node (25) at (3,1) {$M_{\hc{2,5}}$};
      \node (13) at (3,3) {$M_{\hc{1,3}}$};
      \node (45) at (4,0) {$M_{\hc{4,5}}$};
      \node (23) at (4,2) {$M_{\hc{2,3}}$};

      \draw[->] (12) -- (14);
      \draw[->] (34) -- (14);
      \draw[->] (34) -- (35);
      \draw[->] (14) -- (24);
      \draw[->] (14) -- (15);
      \draw[->] (35) -- (15);
      \draw[->] (24) -- (25);
      \draw[->] (15) -- (25);
      \draw[->] (15) -- (13);
      \draw[->] (25) -- (45);
      \draw[->] (25) -- (23);
      \draw[->] (13) -- (23);
    \end{tikzpicture}
    \]
    In Table \ref{1324ex}, we list all faithful torsion-free classes of $\mod kQ$, which corresponds to $c$-sortable elements with full support. Here as in (1), the black vertices indicate simple objects in $\FF(w)$ and the white the rest.
    From this table, for example, we can check that the number of faithful torsion-free classes satisfying (JHP) is 8. By using the computer program, we can calculate that the number of all torsion-free classes is 42, which is the Catalan number, and 34 ones among them satisfy (JHP).
    \begin{table}
      \caption{Example of Theorem \ref{binvsimp} for faithful torsion-free classes over $Q = 1 \ot 2 \to 3 \ot 4$.}
      \label{1324ex}
      \begin{tabular}{C|C|C|c|c}
        \text{$c$-sortable elements $w$} & \inv(w) & \Binv(w) & $\FF(w)$ & $\# \simp\FF(w)$ \\ \hline \hline
        \makecell{c=s_1 s_3 s_2 s_4 \\= 24153} &
        \inv(c) =\makecell{(1\,\,2), (1\,\,4) \\ (3\,\,4), (3\,\,5)} &
        \inv(c) &
        \begin{tikzpicture}
          [baseline={([yshift=-.5ex]current bounding box.center)}, scale=0.3,  every node/.style={scale=0.4}]
          \node (12) at (0,0)[black] {};
          \node (34) at (0,2)[black] {};
          \node (14) at (1,1)[black] {};
          \node (35) at (1,3)[black] {};
          \node (24) at (2,0) {};
          \node (15) at (2,2) {};
          \node (25) at (3,1) {};
          \node (13) at (3,3) {};
          \node (45) at (4,0) {};
          \node (23) at (4,2) {};

          \draw (12) -- (14);
          \draw (34) -- (14);
          \draw (34) -- (35);
          \draw (14) -- (24);
          \draw (14) -- (15);
          \draw (35) -- (15);
          \draw (24) -- (25);
          \draw (15) -- (25);
          \draw (15) -- (13);
          \draw (25) -- (45);
          \draw (25) -- (23);
          \draw (13) -- (23);
          \useasboundingbox ([shift={(.3,.3)}]current bounding box.north east) rectangle ([shift={(-.3,-.3)}]current bounding box.south west);
        \end{tikzpicture}
        & 4
        \\ \hline
        c s_1 = 42153 &
        \inv(c), (2\,\,4) &
        \makecell{(1\,\,2), (2\,\,4), \\ (3\,\,4), (3\,\,5)} &
        \begin{tikzpicture}
          [baseline={([yshift=-.5ex]current bounding box.center)}, scale=0.3,  every node/.style={scale=0.4}]
          \node (12) at (0,0)[black] {};
          \node (34) at (0,2)[black] {};
          \node (14) at (1,1)[white] {};
          \node (35) at (1,3)[black] {};
          \node (24) at (2,0)[black] {};
          \node (15) at (2,2) {};
          \node (25) at (3,1) {};
          \node (13) at (3,3) {};
          \node (45) at (4,0) {};
          \node (23) at (4,2) {};

          \draw (12) -- (14);
          \draw (34) -- (14);
          \draw (34) -- (35);
          \draw (14) -- (24);
          \draw (14) -- (15);
          \draw (35) -- (15);
          \draw (24) -- (25);
          \draw (15) -- (25);
          \draw (15) -- (13);
          \draw (25) -- (45);
          \draw (25) -- (23);
          \draw (13) -- (23);
          \useasboundingbox ([shift={(.3,.3)}]current bounding box.north east) rectangle ([shift={(-.3,-.3)}]current bounding box.south west);
        \end{tikzpicture}
        & 4
        \\ \hline
        c s_3 = 24513 &
        \inv(c), (1\,\,5) &
        \inv(c), (1\,\,5) &
        \begin{tikzpicture}
          [baseline={([yshift=-.5ex]current bounding box.center)}, scale=0.3,  every node/.style={scale=0.4}]
          \node (12) at (0,0)[black] {};
          \node (34) at (0,2)[black] {};
          \node (14) at (1,1)[black] {};
          \node (35) at (1,3)[black] {};
          \node (24) at (2,0) {};
          \node (15) at (2,2)[black] {};
          \node (25) at (3,1) {};
          \node (13) at (3,3) {};
          \node (45) at (4,0) {};
          \node (23) at (4,2) {};

          \draw (12) -- (14);
          \draw (34) -- (14);
          \draw (34) -- (35);
          \draw (14) -- (24);
          \draw (14) -- (15);
          \draw (35) -- (15);
          \draw (24) -- (25);
          \draw (15) -- (25);
          \draw (15) -- (13);
          \draw (25) -- (45);
          \draw (25) -- (23);
          \draw (13) -- (23);
          \useasboundingbox ([shift={(.3,.3)}]current bounding box.north east) rectangle ([shift={(-.3,-.3)}]current bounding box.south west);
        \end{tikzpicture}
        & 5
        \\ \hline
        c s_1 s_3 = 42513 &
        \inv(c), (1\,\,5), (2\,\,4) &
        \makecell{(1\,\,2), (1\,\,5), \\ (2\,\,4), (3\,\,4), (3\,\,5)} &
        \begin{tikzpicture}
          [baseline={([yshift=-.5ex]current bounding box.center)}, scale=0.3,  every node/.style={scale=0.4}]
          \node (12) at (0,0)[black] {};
          \node (34) at (0,2)[black] {};
          \node (14) at (1,1)[white] {};
          \node (35) at (1,3)[black] {};
          \node (24) at (2,0)[black] {};
          \node (15) at (2,2)[black] {};
          \node (25) at (3,1) {};
          \node (13) at (3,3) {};
          \node (45) at (4,0) {};
          \node (23) at (4,2) {};

          \draw (12) -- (14);
          \draw (34) -- (14);
          \draw (34) -- (35);
          \draw (14) -- (24);
          \draw (14) -- (15);
          \draw (35) -- (15);
          \draw (24) -- (25);
          \draw (15) -- (25);
          \draw (15) -- (13);
          \draw (25) -- (45);
          \draw (25) -- (23);
          \draw (13) -- (23);
          \useasboundingbox ([shift={(.3,.3)}]current bounding box.north east) rectangle ([shift={(-.3,-.3)}]current bounding box.south west);
        \end{tikzpicture}
        & 5
        \\ \hline
        c s_3 s_2= 25413 &
        \inv(c), (1\,\,5), (4\,\,5) &
        \makecell{(1\,\,2), (1\,\,4), \\ (3\,\,4), (4\,\,5)} &
        \begin{tikzpicture}
          [baseline={([yshift=-.5ex]current bounding box.center)}, scale=0.3,  every node/.style={scale=0.4}]
          \node (12) at (0,0)[black] {};
          \node (34) at (0,2)[black] {};
          \node (14) at (1,1)[black] {};
          \node (35) at (1,3)[white] {};
          \node (24) at (2,0) {};
          \node (15) at (2,2)[white] {};
          \node (25) at (3,1) {};
          \node (13) at (3,3) {};
          \node (45) at (4,0)[black] {};
          \node (23) at (4,2) {};

          \draw (12) -- (14);
          \draw (34) -- (14);
          \draw (34) -- (35);
          \draw (14) -- (24);
          \draw (14) -- (15);
          \draw (35) -- (15);
          \draw (24) -- (25);
          \draw (15) -- (25);
          \draw (15) -- (13);
          \draw (25) -- (45);
          \draw (25) -- (23);
          \draw (13) -- (23);
          \useasboundingbox ([shift={(.3,.3)}]current bounding box.north east) rectangle ([shift={(-.3,-.3)}]current bounding box.south west);
        \end{tikzpicture}
        & 4
        \\ \hline
        c s_3 s_4 = 24531 &
        \inv(c), (1\,\,3), (1\,\,5) &
        \makecell{(1\,\,2), (1\,\,3), \\ (3\,\,4), (3\,\,5)} &
        \begin{tikzpicture}
          [baseline={([yshift=-.5ex]current bounding box.center)}, scale=0.3,  every node/.style={scale=0.4}]
          \node (12) at (0,0)[black] {};
          \node (34) at (0,2)[black] {};
          \node (14) at (1,1)[white] {};
          \node (35) at (1,3)[black] {};
          \node (24) at (2,0) {};
          \node (15) at (2,2)[white] {};
          \node (25) at (3,1) {};
          \node (13) at (3,3)[black] {};
          \node (45) at (4,0) {};
          \node (23) at (4,2) {};

          \draw (12) -- (14);
          \draw (34) -- (14);
          \draw (34) -- (35);
          \draw (14) -- (24);
          \draw (14) -- (15);
          \draw (35) -- (15);
          \draw (24) -- (25);
          \draw (15) -- (25);
          \draw (15) -- (13);
          \draw (25) -- (45);
          \draw (25) -- (23);
          \draw (13) -- (23);
          \useasboundingbox ([shift={(.3,.3)}]current bounding box.north east) rectangle ([shift={(-.3,-.3)}]current bounding box.south west);
        \end{tikzpicture}
        & 4
        \\ \hline
        c s_1 s_3 s_2 = 45213 &
        \makecell{\inv(c), (1\,\,5),\\ (2\,\,4), (2\,\,5)}  &
        \makecell{(1\,\,2), (2\,\,4), (2\,\,5), \\ (3\,\,4), (3\,\,5)} &
        \begin{tikzpicture}
          [baseline={([yshift=-.5ex]current bounding box.center)}, scale=0.3,  every node/.style={scale=0.4}]
          \node (12) at (0,0)[black] {};
          \node (34) at (0,2)[black] {};
          \node (14) at (1,1)[white] {};
          \node (35) at (1,3)[black] {};
          \node (24) at (2,0)[black] {};
          \node (15) at (2,2)[white] {};
          \node (25) at (3,1)[black] {};
          \node (13) at (3,3) {};
          \node (45) at (4,0) {};
          \node (23) at (4,2) {};

          \draw (12) -- (14);
          \draw (34) -- (14);
          \draw (34) -- (35);
          \draw (14) -- (24);
          \draw (14) -- (15);
          \draw (35) -- (15);
          \draw (24) -- (25);
          \draw (15) -- (25);
          \draw (15) -- (13);
          \draw (25) -- (45);
          \draw (25) -- (23);
          \draw (13) -- (23);
          \useasboundingbox ([shift={(.3,.3)}]current bounding box.north east) rectangle ([shift={(-.3,-.3)}]current bounding box.south west);
        \end{tikzpicture}
        & 5
        \\ \hline
        c s_1 s_3 s_4 = 42531 &
        \makecell{\inv(c), (1\,\,3),\\ (1\,\,5), (2\,\,4)} &
        \makecell{(1\,\,2), (1\,\,3), \\(2\,\,4), (3\,\,4), (3\,\,5)} &
        \begin{tikzpicture}
          [baseline={([yshift=-.5ex]current bounding box.center)}, scale=0.3,  every node/.style={scale=0.4}]
          \node (12) at (0,0)[black] {};
          \node (34) at (0,2)[black] {};
          \node (14) at (1,1)[white] {};
          \node (35) at (1,3)[black] {};
          \node (24) at (2,0)[black] {};
          \node (15) at (2,2)[white] {};
          \node (25) at (3,1) {};
          \node (13) at (3,3)[black] {};
          \node (45) at (4,0) {};
          \node (23) at (4,2) {};

          \draw (12) -- (14);
          \draw (34) -- (14);
          \draw (34) -- (35);
          \draw (14) -- (24);
          \draw (14) -- (15);
          \draw (35) -- (15);
          \draw (24) -- (25);
          \draw (15) -- (25);
          \draw (15) -- (13);
          \draw (25) -- (45);
          \draw (25) -- (23);
          \draw (13) -- (23);
          \useasboundingbox ([shift={(.3,.3)}]current bounding box.north east) rectangle ([shift={(-.3,-.3)}]current bounding box.south west);
        \end{tikzpicture}
        & 5
        \\ \hline
        c s_3 s_2 s_4= 25431 &
        \makecell{\inv(c), (1\,\,3), \\ (1\,\,5), (4\,\,5)} &
        \makecell{(1\,\,2), (1\,\,3), \\ (3\,\,4), (4\,\,5)} &
        \begin{tikzpicture}
          [baseline={([yshift=-.5ex]current bounding box.center)}, scale=0.3,  every node/.style={scale=0.4}]
          \node (12) at (0,0)[black] {};
          \node (34) at (0,2)[black] {};
          \node (14) at (1,1)[white] {};
          \node (35) at (1,3)[white] {};
          \node (24) at (2,0) {};
          \node (15) at (2,2)[white] {};
          \node (25) at (3,1) {};
          \node (13) at (3,3)[black] {};
          \node (45) at (4,0)[black] {};
          \node (23) at (4,2) {};

          \draw (12) -- (14);
          \draw (34) -- (14);
          \draw (34) -- (35);
          \draw (14) -- (24);
          \draw (14) -- (15);
          \draw (35) -- (15);
          \draw (24) -- (25);
          \draw (15) -- (25);
          \draw (15) -- (13);
          \draw (25) -- (45);
          \draw (25) -- (23);
          \draw (13) -- (23);
          \useasboundingbox ([shift={(.3,.3)}]current bounding box.north east) rectangle ([shift={(-.3,-.3)}]current bounding box.south west);
        \end{tikzpicture}
        & 4
        \\ \hline
        \makecell{c^2 = c s_1 s_3 s_2 s_4 \\= 45231} &
        \makecell{\inv(c), (1\,\,3), (1\,\,5),\\ (2\,\,4), (2\,\,5)}  &
        \makecell{(1\,\,2), (1\,\,3), (2\,\,4),\\ (2\,\,5), (3\,\,4), (3\,\,5)} &
        \begin{tikzpicture}
          [baseline={([yshift=-.5ex]current bounding box.center)}, scale=0.3,  every node/.style={scale=0.4}]
          \node (12) at (0,0)[black] {};
          \node (34) at (0,2)[black] {};
          \node (14) at (1,1)[white] {};
          \node (35) at (1,3)[black] {};
          \node (24) at (2,0)[black] {};
          \node (15) at (2,2)[white] {};
          \node (25) at (3,1)[black] {};
          \node (13) at (3,3)[black] {};
          \node (45) at (4,0) {};
          \node (23) at (4,2) {};

          \draw (12) -- (14);
          \draw (34) -- (14);
          \draw (34) -- (35);
          \draw (14) -- (24);
          \draw (14) -- (15);
          \draw (35) -- (15);
          \draw (24) -- (25);
          \draw (15) -- (25);
          \draw (15) -- (13);
          \draw (25) -- (45);
          \draw (25) -- (23);
          \draw (13) -- (23);
          \useasboundingbox ([shift={(.3,.3)}]current bounding box.north east) rectangle ([shift={(-.3,-.3)}]current bounding box.south west);
        \end{tikzpicture}
        & 6
        \\ \hline
        \makecell{c s_1 s_3 s_2 s_1 \\ = 54213} &
        \makecell{\inv(c), (1\,\,5),\\ (2\,\,4), (2\,\,5), (4\,\,5)}  &
        \makecell{(1\,\,2), (2\,\,4), \\ (3\,\,4), (4\,\,5)} &
        \begin{tikzpicture}
          [baseline={([yshift=-.5ex]current bounding box.center)}, scale=0.3,  every node/.style={scale=0.4}]
          \node (12) at (0,0)[black] {};
          \node (34) at (0,2)[black] {};
          \node (14) at (1,1)[white] {};
          \node (35) at (1,3)[white] {};
          \node (24) at (2,0)[black] {};
          \node (15) at (2,2)[white] {};
          \node (25) at (3,1)[white] {};
          \node (13) at (3,3) {};
          \node (45) at (4,0)[black] {};
          \node (23) at (4,2) {};

          \draw (12) -- (14);
          \draw (34) -- (14);
          \draw (34) -- (35);
          \draw (14) -- (24);
          \draw (14) -- (15);
          \draw (35) -- (15);
          \draw (24) -- (25);
          \draw (15) -- (25);
          \draw (15) -- (13);
          \draw (25) -- (45);
          \draw (25) -- (23);
          \draw (13) -- (23);
          \useasboundingbox ([shift={(.3,.3)}]current bounding box.north east) rectangle ([shift={(-.3,-.3)}]current bounding box.south west);
        \end{tikzpicture}
        & 4
        \\ \hline
        \makecell{c^2 s_1 = c s_1 s_3 s_2 s_4 s_1 \\= 54231} &
        \makecell{\inv(c), (1\,\,3), (1\,\,5),\\ (2\,\,4), (2\,\,5), (4\,\,5)}  &
        \makecell{(1\,\,2), (1\,\,3), (2\,\,4),\\ (3\,\,4), (4\,\,5)} &
        \begin{tikzpicture}
          [baseline={([yshift=-.5ex]current bounding box.center)}, scale=0.3,  every node/.style={scale=0.4}]
          \node (12) at (0,0)[black] {};
          \node (34) at (0,2)[black] {};
          \node (14) at (1,1)[white] {};
          \node (35) at (1,3)[white] {};
          \node (24) at (2,0)[black] {};
          \node (15) at (2,2)[white] {};
          \node (25) at (3,1)[white] {};
          \node (13) at (3,3)[black] {};
          \node (45) at (4,0)[black] {};
          \node (23) at (4,2) {};

          \draw (12) -- (14);
          \draw (34) -- (14);
          \draw (34) -- (35);
          \draw (14) -- (24);
          \draw (14) -- (15);
          \draw (35) -- (15);
          \draw (24) -- (25);
          \draw (15) -- (25);
          \draw (15) -- (13);
          \draw (25) -- (45);
          \draw (25) -- (23);
          \draw (13) -- (23);
          \useasboundingbox ([shift={(.3,.3)}]current bounding box.north east) rectangle ([shift={(-.3,-.3)}]current bounding box.south west);
        \end{tikzpicture}
        & 5
        \\ \hline
        \makecell{c^2 s_3 = c s_1 s_3 s_2 s_4 s_3\\= 45321} &
        \makecell{\inv(c), (1\,\,3), (1\,\,5),\\ (2\,\,3), (2\,\,4), (2\,\,5)}  &
        \makecell{(1\,\,2),  (2\,\,3),\\  (3\,\,4), (3\,\,5)} &
        \begin{tikzpicture}
          [baseline={([yshift=-.5ex]current bounding box.center)}, scale=0.3,  every node/.style={scale=0.4}]
          \node (12) at (0,0)[black] {};
          \node (34) at (0,2)[black] {};
          \node (14) at (1,1)[white] {};
          \node (35) at (1,3)[black] {};
          \node (24) at (2,0)[white] {};
          \node (15) at (2,2)[white] {};
          \node (25) at (3,1)[white] {};
          \node (13) at (3,3)[white] {};
          \node (45) at (4,0) {};
          \node (23) at (4,2)[black] {};

          \draw (12) -- (14);
          \draw (34) -- (14);
          \draw (34) -- (35);
          \draw (14) -- (24);
          \draw (14) -- (15);
          \draw (35) -- (15);
          \draw (24) -- (25);
          \draw (15) -- (25);
          \draw (15) -- (13);
          \draw (25) -- (45);
          \draw (25) -- (23);
          \draw (13) -- (23);
          \useasboundingbox ([shift={(.3,.3)}]current bounding box.north east) rectangle ([shift={(-.3,-.3)}]current bounding box.south west);
        \end{tikzpicture}
        & 4
        \\ \hline
        \makecell{c^2 s_1 s_3 = c s_1 s_3 s_2 s_4 s_1 s_3 \\= 54321} &
        \makecell{\inv(c), (1\,\,3), (1\,\,5),\\ (2\,\,3), (2\,\,4), (2\,\,5), (4\,\,5)}  &
        \makecell{(1\,\,2), (2\,\,3), \\ (3\,\,4), (4\,\,5)} &
        \begin{tikzpicture}
          [baseline={([yshift=-.5ex]current bounding box.center)}, scale=0.3,  every node/.style={scale=0.4}]
          \node (12) at (0,0)[black] {};
          \node (34) at (0,2)[black] {};
          \node (14) at (1,1)[white] {};
          \node (35) at (1,3)[white] {};
          \node (24) at (2,0)[white] {};
          \node (15) at (2,2)[white] {};
          \node (25) at (3,1)[white] {};
          \node (13) at (3,3)[white] {};
          \node (45) at (4,0)[black] {};
          \node (23) at (4,2)[black] {};

          \draw (12) -- (14);
          \draw (34) -- (14);
          \draw (34) -- (35);
          \draw (14) -- (24);
          \draw (14) -- (15);
          \draw (35) -- (15);
          \draw (24) -- (25);
          \draw (15) -- (25);
          \draw (15) -- (13);
          \draw (25) -- (45);
          \draw (25) -- (23);
          \draw (13) -- (23);
          \useasboundingbox ([shift={(.3,.3)}]current bounding box.north east) rectangle ([shift={(-.3,-.3)}]current bounding box.south west);
        \end{tikzpicture}
        & 4
        \\ \hline
      \end{tabular}
    \end{table}
  \end{enumerate}
\end{example}

As an application of Corollary \ref{jhtypea}, we obtain the following result on the linearly oriented case. The proof is purely combinatorial. Note that this also follows from Corollary \ref{nakayamamain}, since $kQ$ is a Nakayama algebra in this case.
\begin{corollary}
  Let $Q$ be a linearly oriented quiver of type A. Then every torsion-free class $\FF$ of $\mod kQ$ satisfies (JHP).
\end{corollary}
\begin{proof}
  We may assume that $Q$ is $1 \to 2 \to \cdots \to n$, so $c = s_n \cdots s_2 s_1$.
  We use the combinatorial characterization of the $c$-sortable-ness given in \cite[Theorem 4.12]{reading}: an element $w \in S_{n+1}$ is $c$-sortable if it is \emph{231-avoiding}, that is, there exists no $i < l < j$ such that $l$ appears before $j$ and $j$ appears before $i$ in the one-line notation for $w$.

  By Theorem \ref{sortabletorf}, there exists a $c$-sortable element $w \in S_{n+1}$ satisfying $\FF = \FF(w)$.
  Then according to Corollary \ref{jhtypea}, the assertion amounts to the following purely combinatorial lemma (or equivalently, this lemma follows from two categorical statements: Corollaries \ref{jhtypea} and \ref{nakayamamain}).
  \begin{lemma*}
    For every {\upshape 231}-avoiding element $w$ of $S_{n+1}$, we have $\#\Binv (w) = \#\supp(w)$.
  \end{lemma*}

  To prove this lemma, we will construct an explicit bijection $\Binv(w) \xrightarrow{\sim} \supp(w)$. In what follows, we will work on the one-line notation for $w$.

  Let $(i \,\,j)$ be a Bruhat inversion of $w$ with $i < j$. We claim that $i \in \supp(w)$ holds. If this is not the case, then the initial segment of length $i$ contains exactly $1,2, \cdots, i$ by Lemma \ref{supplem}. This is a contradiction since $j \,\,(> i)$ appears before $i$ by $(i\,\,j) \in \inv(w)$. Thus, $i \in  \supp(w)$ holds, and we obtain a map $\Binv(w) \to \supp(w)$ by $(i\,\,j) \mapsto i$.

  We will show that this map is an injection. Suppose $(i \,\, j_1)$ and $(i \,\,j_2)$ are two different Bruhat inversion of $w$ with $i < j_1 < j_2$, so $j_1$ and $j_2$ appear before $i$. Then $j_1$ must appear before $j_2$, since otherwise $(i\,\,j_2)$ would not be a Bruhat inversion. Thus, $w$ looks like $\cdots j_1 \cdots  j_2 \cdots i \cdots$, which contradicts to that $w$ is 231-avoiding. Thus, this map is an injection.

  Next, we will show that this map is a surjection.
  For $i \in \supp(w)$, we will show the following:

  {\bf (Claim)}:
   \emph{There exists some letter $j$ such that $i < j$ and $j$ appears before $i$}.

   By Lemma \ref{supplem}, there exists some $l$ with $l > i$ such that $l$ appears in the initial segment of length $i$. Suppose that any letters left to $i$ are less than $i$, that is, there is no $j$ as in the above claim. Then in particular $i$ belongs to the initial segment of length $i$, and is left to $l$.
  Also by Lemma \ref{supplem} (2), there exists a letter $i'$ with $i' \leq i$ such that $i'$ does not appear in the initial segment of length $i$.
  However, we must have $i' \neq i$, thus $i' < i$, since $i$ appears in the initial segment of length $i$.
  Now we have $i' < i < l$ and $w$ looks like$\cdots i \cdots l \cdots i' \cdots$, which contradicts to that $w$ is 231-avoiding. Thus, the claim follows.

  Take the rightmost letter $j$ with the claimed property. Then $(i \,\, j)$ is clearly a Bruhat inversion. Thus, the map $\Binv(w) \to \supp(w)$ is surjective.
\end{proof}

\begin{remark}\label{forthcoming}
  In a forthcoming paper \cite{enforth}, we will show that the natural analogue for Theorem \ref{binvsimp} holds for $\mod kQ$ with an arbitrary Dynkin quiver $Q$, or more generally, for $\mod \Pi$ for a preprojective algebra $\Pi$ of an arbitrary Dynkin type.
\end{remark}

\begin{remark}\label{rem:wide}
  Recall from Remark \ref{rem:coverref} that the set of cover relations is a subset of the set of Bruhat inversions. This can be understood representation-theoretically as follows. Let $Q$, $c$ and $w$ be as in Theorem \ref{binvsimp} and consider the torsion-free class $\FF:= \FF(w)$. Associated to $\FF$, Ingalls-Thomas defined a wide subcategory $\WR(\FF)$, which is given by
  \[
  \WR(\FF) := \{ X \in \FF \, | \, \text{for every map $\varphi \colon X \to F$ with $F \in \FF$, we have $\coker\varphi \in \FF$} \}.
  \]
  This category is a proper subcategory of $\FF$ which is abelian, hence one can speak of simple objects in $\WR(\FF)$, and the natural bijection $\inv(w) \xrightarrow{\sim} \ind \FF$ turns out to induce a bijection between the set of cover reflections for $w$ and $\simp\WR(\FF)$.
  Consequently, $\simp\WR(\FF)$ is a subset of $\simp\FF$ by Theorem \ref{binvsimp}, and \emph{this is always the case} under a bijection between wide subcategories and torsion-free classes established in \cite{ms}.
  In fact, $\FF$ is the smallest torsion-free class containing the semibrick $\simp\WR(\FF)$, and a characterization of this semibricks inside $\FF$ in \cite[Lemma 1.7]{asai} immediately implies that $\simp\WR(\FF) \subseteq \simp\FF$ holds. Note that \cite{it,ms,asai} deal with torsion classes, but we use torsion-free classes.
\end{remark}

\section{Computations of the Grothendieck monoids}\label{compgromon}
So far, we do not know any method to compute the Grothendieck monoid $\MM(\EE)$ for a given length exact category $\EE$ in general, except the following information:
\begin{itemize}
  \item $\MM(\EE)$ is atomic, that is, generated by $\Atom \MM(\EE) = \{ [S] \, | \, S \in \simp \EE \}$ (Proposition \ref{atomicfg}).
  \item $\MM(\EE)$ is just a free monoid if $\EE$ turns out to satisfy (JHP) (Theorem \ref{JHchar}).
  \item The cancellative quotient $\MM(\EE)_\can$, which is isomorphic to the positive part $\KK_0^+(\EE)$ of the Grothendieck group (Proposition \ref{canquot}), is nothing but the monoid of dimension vectors of modules belonging to $\EE$ under the Assumption \ref{assumperp} (Corollary \ref{setofdim}).
\end{itemize}
Thus, we have the following problem: we do not know how to check whether $\MM(\EE)$ is cancellative or not (if $\MM(\EE)$ is cancellative, then $\MM(\EE)$ can be calculated in principle by the above), and the calculation seems to be rather difficult if $\MM(\EE)$ is \emph{not} cancellative.

In this section, we will show (a bit artificial) examples of  calculations of $\MM(\EE)$ such that $\MM(\EE)$ is \emph{not} cancellative.

\subsection{A category of modules with a designated set of dimension vectors}
We will convince the reader that lots of monoids can appear as the Grothendieck monoids of exact categories.
First, we consider the case of split exact categories.
\begin{proposition}\label{splitmon}
  Let $\EE$ be a split exact category, that is, every conflation splits. Then $\MM(\EE)$ is isomorphic to $\Iso \EE$ (with addition given by $\oplus$) as monoids.
\end{proposition}
\begin{proof}
  This follows from the construction given in Proposition \ref{gromon}, or one can show this by checking the universal property directly. The details are left to the reader.
\end{proof}

Using this, we can realize any submonoids of $\N^n$ as Grothendieck monoids.
\begin{proposition}
  Let $M$ be a submonoid of $\N^n$ for some $n \geq 0$. Then there exists a split exact category $\EE$ whose Grothendieck monoid is isomorphic to $M$.
\end{proposition}
\begin{proof}
  Let $k$ be a field and consider a semisimple $k$-algebra $\Lambda := k^n$. Then $\Lambda$ has $n$ non-isomorphic simple modules, and we denote by $\udim X \in \N^n$ for $X \in \mod\Lambda$ the dimension vector of $X$, see Definition \ref{dimvecdef}. Define the full subcategory $\EE$ of $\mod \Lambda$ by the following:
  \[
  \EE := \{ X \in \mod\Lambda \, | \, \udim X \in M \}.
  \]
  This is an extension-closed subcategory of $\mod\Lambda$, and we claim that $\MM(\EE) \iso M$ holds.
  First, observe that every short exact sequence in $\mod\Lambda$ splits, so $\EE$ is a split exact category.
  Thus, it suffices to show $\Iso \EE \iso M$ by Proposition \ref{splitmon}. The map $X \mapsto \udim X$ induces an isomorphism $\Iso (\mod\Lambda) \iso \N^n$, and by construction, this map restricts to an isomorphism $\Iso \EE \iso M$.
\end{proof}
If $\EE$ is a full subcategory of a module category, then by regarding $\EE$ as an exact category with a split exact structure, Proposition \ref{splitmon} shows that our monoid $\MM(\EE)$ only concerns with direct sum decompositions of modules.
In this case, $\MM(\EE)$ is free if and only if the uniqueness of direct sum decompositions holds (e.g. Krull-Schmidt categories), thus the combinatorial property of the monoid $\MM(\EE)$ encodes information on the non-unique direct sum factorizations of modules. This has been studied by several authors (e.g. \cite{fac1,fac2,baeg}), and we refer the reader to the recent article \cite{baeg} and references therein for more information on this direction.

In a similar way to the construction of $\EE$ above, we can attach extension-closed subcategories of module categories to all submonoids of $\MM(\mod\Lambda)$ for an artin algebra $\Lambda$.
\begin{definition}\label{emdef}
  Let $\Lambda$ be an artin algebra with $n$ non-isomorphic simple modules. For a submonoid $M$ of $\N^n$, we define the subcategory $\EE_M$ of $\mod\Lambda$ by the following:
  \[
  \EE_M := \{ X \in \mod \Lambda \, | \, \udim X \in M \}.
  \]
  Then it is an extension-closed subcategory of $\mod\Lambda$, thus an exact category.
\end{definition}
By the universal property of $\MM(\EE_M)$, we have a monoid homomorphism $\MM(\EE_M) \to M$ induced by $\udim$. Indeed $\MM(\EE_M)$ is a \emph{thickening} of $M$ by some non-cancellative part:
\begin{proposition}\label{mem}
  Let $\Lambda$ be an artin algebra with $n$ non-isomorphic simple modules and $M$ a submonoid of $\N^n$. Then $\udim \colon \MM(\EE_M) \to M$ induces an isomorphism $\MM(\EE_M)_\can \iso M$ of monoids.
\end{proposition}
\begin{proof}
  Let us introduce some notation. Write $\EE: = \EE_M$ for simplicity. Let $S_i$ denote the simple $\Lambda$-module with $\udim S_i = e_i$ for $1 \leq i \leq n$, where $e_i$ is the standard basis of $\N^n$.
  For an element $d \in \N^n$, we denote by $S_d$ the unique semisimple $\Lambda$-module satisfying $\udim S_d = d$, in other words, $S_d := \bigoplus_i S_i^{\oplus a_i}$ for $d = \sum_{i} a_i e_i$ with $a_i \in \N$.

  For $d \in M$, clearly we have $S_d \in \EE$. Since $\udim S_d = d$ holds, the map $\udim \colon \MM(\EE) \to M$ is a surjection.

  Thus, to prove $\MM(\EE)_\can \iso M$, it suffices to show that for two objects $X, Y \in \EE$, we have $[X] \sim_\can [Y]$ in $\MM(\EE)$ if and only if $\udim X = \udim Y$ holds (see Proposition \ref{canquot}).
  The ``only if'' part is clear since $\N^n$ (or $M$) is cancellative, so we will prove the ``if'' part.
  To do this, it suffices to show $[X] \sim_\can [S_d]$ for every $X \in \EE$ with $\udim X = d$.

  Let $l_{\rm max} (X)$ be the maximum of lengths of indecomposable direct summands of $X$, and we show $[X] \sim_\can [S_d]$ by induction on $l := l_{\rm max}(X)$.
  If $l =1$, then $X$ is semisimple and $X \iso S_d$ holds.
  Suppose $l >1$, and take a direct summand $X_1$ of $X$ with $l(X_1) = l$.
  We have $X = X_1 \oplus X/X_1$, and take a simple submodule $S$ of $X_1$.
  Note that $S$ is a direct summand of $S_d$ since $S$ is a composition factor of $X$.
  Thus, there are two exact sequences in $\mod\Lambda$:
  \[
  \begin{tikzcd}[row sep=0em]
    0 \rar & S \rar & X_1 \rar & X_1/S \rar & 0, \quad \text{and} \\
    0 \rar & S_d/S \rar & S_d/S \oplus X/X_1 \oplus S \rar & X/X_1 \oplus S \rar & 0,
  \end{tikzcd}
  \]
  where the second one is just a split exact sequence.
  By taking direct sum of these two sequences, we obtain the following exact sequence.
  \[
  \begin{tikzcd}
    0 \rar & S_d \rar & X \oplus S_d \rar & S \oplus X_1/S \oplus X/X_1 \rar & 0
  \end{tikzcd}
  \]
  The dimension vectors of modules in this exact sequence is $d$, $2d$ and $d$ respectively, so this is a conflation of $\EE$. Therefore, we have $[X] + [S_d] = [X'] + [S_d]$ in $\MM(\EE)$, hence $[X] \sim_\can [X']$ holds, where $X' = S \oplus X_1/S \oplus X/X_1$.
  Thus, we can replace the direct summand $X_1$ of $X$ with $S \oplus X_1/S$. By iterating this process to all direct summands of $X$ with length $l$, we obtain an object $Z$ in $\EE$ such that $l_{\rm max}(Z) < l$ and $[X] \sim_\can [Z]$. Thus, the claim follows from induction.
\end{proof}
Thus, our category $\EE_M$ has a monoid $M$ as a cancellative quotient of its Grothendieck monoid $\MM(\EE_M)$, but $\MM(\EE_M)$ is not cancellative in general. In the next subsection, we give an example of computation of $\MM(\EE_M)$.

\subsection{The case of $A_2$ quiver and monoids with one generator}
Throughout this subsection, we denote by $k$ a field and by $Q$ a quiver $1 \ot 2$, and put $\Lambda := kQ$. Although the structure of the category $\mod \Lambda$ seems to be completely understood, this category has lots of extension-closed subcategories, and the computation of their Grothendieck monoids is rather hard as we shall see in this subsection. A particular case of this computation was considered in \cite[Section 3.4]{bg}.

First, we recall basic properties of the category $\mod \Lambda$.
We have the complete list of indecomposable objects in $\mod \Lambda$: two simple modules $S_1$ and $S_2$, which are supported at vertices $1$ and $2$ respectively, and one non-simple projective injective module $P$.
The Grothendieck group $\KK_0(\mod\Lambda)$ is a free abelian group with basis $\{[S_1], [S_2]\}$, and we identify it with $\Z^2$.

Now we apply the construction $\EE_M$ defined in the previous subsection to this case, where $M$ is a submonoid of $\N^2$. We will compute the monoid $\MM(\EE_M)$ in the case that $M$ is generated by one element. In such a case, by Proposition \ref{mem}, $\MM(\EE_M)_\can \iso \N$ holds, thus $\MM(\EE_M)$ is a \emph{thickening} of $\N$ in some sense.

To present the monoid structure, we will use the \emph{Cayley quiver}, which is a monoid version of the Cayley graph of a group.
\begin{definition}
  Let $M$ be a monoid with a set of generators $A \subset M$. Then the \emph{Cayley quiver} of $M$ with respect to $A$ is a quiver defined as follows:
  \begin{itemize}
    \item The vertex set is $M$.
    \item For each $a \in A$ and $m \in M$, we draw a (labelled) arrow $m \xrightarrow{a} m+a$.
  \end{itemize}
\end{definition}
For an atomic monoid $M$, the natural choice of $A$ above is the set $\Atom M$ of all atoms of $M$. Now we can state our computation.
\begin{proposition}\label{compex}
  Let $M:=\N (m,n)$ be a submonoid of $\N^2$ generated by $(m,n)\in \N^2$ with $(m,n) \neq (0,0)$.  Consider the exact category $\EE:= \EE_M$ as in Definition \ref{emdef}. Then the following hold, where we set $l := \min \{m,n\}$.
  \begin{enumerate}
    \item $\EE$ has exactly $l+1$ distinct simple objects $A_0, \cdots, A_l$, where
    \[
    A_i := P^{\oplus i} \oplus S_1^{\oplus (m-i)} \oplus S_2^{\oplus(n-i)}
    \]
    for $0 \leq i \leq l$. We have $\dim A_i = (m,n)$ for every $i$.
    \item Put $a_i := [A_i] \in \MM(\EE)$ for $0 \leq i \leq l$, which are precisely the atoms of $\MM(\EE)$. Then the Cayley quiver of $\MM(\EE)$ with respect to $\Atom \MM(\EE)$ is given as follows, where we draw arrows $\twoheadrightarrow$ to represent $l+1$ arrows $a_0, \cdots, a_l$.
    \begin{itemize}
      \item {\bf (Case 1)} The case $m \neq n$:
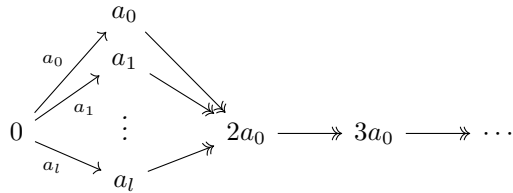
      \[
      \begin{tikzcd}[row sep = tiny]
        & a_0 \ar[rdd, twoheadrightarrow] \\
        & a_1 \ar[rd, twoheadrightarrow]\\
        0 \ar[ruu, "a_0"] \ar[ru, "a_1"'] \ar[rd, "a_l"'] & \vdots & 2a_0 \rar[twoheadrightarrow] & 3a_0 \rar[twoheadrightarrow] & \cdots \\
        & a_l \ar[ru, twoheadrightarrow]
      \end{tikzcd}
      \]
      In particular, if $m=0$ or $n=0$, then $\MM(\EE) \iso \N$ holds.
      \item {\bf (Case 2)} The case $m = n$:
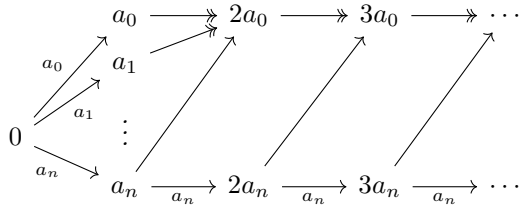
      \[
      \begin{tikzcd}[row sep = tiny]
        & a_0 \rar[twoheadrightarrow] & 2a_0 \rar[twoheadrightarrow] & 3a_0 \rar[twoheadrightarrow] & \cdots \\
        & a_1 \ar[ru, twoheadrightarrow]\\
        0 \ar[ruu, "a_0"] \ar[ru, "a_1"'] \ar[rd, "a_n"'] & \vdots\\
        & a_n \ar[ruuu] \rar["a_n"'] & 2 a_n \ar[ruuu]\rar["a_n"'] & 3a_n \ar[ruuu]\rar["a_n"'] & \cdots
      \end{tikzcd}
      \]
      Here non-labelled arrows $\to$ represent $n$ arrows $a_0, \cdots, a_{n-1}$.
    \end{itemize}
  \end{enumerate}
\end{proposition}
\begin{proof}
  For simplicity, put $\EE := \EE_M$, and $X^i := X^{\oplus i}$ for $X \in \mod\Lambda$.

  (1)
  It is clear that each $A_i$ is a simple object in $\EE$ since $\udim A_i = (m,n)$, and that $A_i \iso A_j$ holds if and only if $i = j$.
  On the other hand, we claim that every object in $\EE$ is a \emph{direct sum} of $A_i$'s, which clearly implies the assertion. In fact, take an object $X \in \EE$ with $\udim X = (Nm, Nn)$ for $N > 0$. Then $X$ is isomorphic to $P^{ i} \oplus S_1^{Nm-i} \oplus S_2^{Nn-i}$ for some $0 \leq i \leq Nl$.
  Take any integers $i_1, i_2,\cdots, i_N$ such that $0 \leq i_j \leq l$ for each $j$ and $i = i_1 + \cdots + i_N$. Then it is straightforward to see that $X \iso A_{i_1} \oplus \cdots \oplus A_{i_N}$ holds.

  (2)
  By Proposition \ref{simpleatom}, we have that $\Atom \MM(\EE) = \{ a_0,a_1, \cdots, a_l\}$ holds for $a_i:=[A_i]$, and all of them are distinct.
  On the other hand, we have a map $\varphi \colon \MM(\EE) \to \N$ which sends an object $X$ with $\udim X = (Nm, Nn)$ to $N$ for $N \in \N$. This map is clearly a surjection (and actually this is the universal cancellative quotient of $\MM(\EE)$ by Proposition \ref{mem}, but this fact is not needed here).
  Denote by $\MM(\EE)_N$ the inverse image $\varphi^{-1}(N)$, and our strategy to compute $\MM(\EE)$ is to compute $\MM(\EE)_N$.
  Clearly we have $\MM(\EE)_0 = \{ 0=[0] \}$ and $\MM(\EE)_1 = \Atom \MM(\EE) = \{a_0,a_1, \cdots, a_l\}$. Moreover, since every object of $\EE$ is a direct sum of $A_i$'s, we have $x \in \MM(\EE)_N$ if and only if there exist integers $0 \leq i_1, \cdots, i_N \leq l$ satisfying $x = a_{i_1} + \cdots + a_{i_N}$.
  The key part of our computation is to show the following equality in $\MM(\EE)$:

  {\bf (Claim 1)}:
  \emph{
    For $1 \leq i \leq l$ and $0 \leq j \leq l$, we have the equality
    $a_i + a_j = a_{i-1} + a_j$ in $\MM(\EE)$, except the case $j =m = n$.
    }

  \emph{Proof of Claim 1}.
  First, note that we have the following exact sequence in $\mod \Lambda$.
  \[
  \begin{tikzcd}
    0 \rar & S_1 \rar & P \rar & S_2 \rar & 0
  \end{tikzcd}
  \]
  Suppose that the equality $j = m = n$ does not hold, then one can easily check that either $j \leq m-1$ or $j \leq n-1$ (or both) holds. Suppose the former, and consider a split exact sequence below
  \[
  \begin{tikzcd}[column sep=1em]
    0 \rar & P^{j} \oplus S_1^{m-1-j} \oplus S_2^{n-j} \rar &
    P^{i+j-1} \oplus S_1^{2m-i-j} \oplus S_2^{2n-i-j} \rar &
    P^{i-1} \oplus S_1^{m-i+1} \oplus S_2^{n-i} \rar & 0.
  \end{tikzcd}
  \]
  By taking direct sum of the above two exact sequences, one obtains a conflation in $\EE$:
  \[
  \begin{tikzcd}
    0 \rar & A_j \rar & A_i \oplus A_j \rar & A_{i-1} \rar & 0.
  \end{tikzcd}
  \]
  Thus, $a_i + a_j = a_{i-1} + a_j$ holds in $\MM(\EE)$.

  Similarly, suppose that $j \leq n-1$ holds. Then by considering the following split sequence
  \[
  \begin{tikzcd}[column sep=1em]
    0 \rar & P^{i-1} \oplus S_1^{m-i} \oplus S_2^{n-i+1} \rar &
    P^{i+j-1} \oplus S_1^{2m-i-j} \oplus S_2^{2n-i-j} \rar &
    P^j \oplus S_1^{m-j} \oplus S_2^{n-1-j} \rar & 0,
  \end{tikzcd}
  \]
  as in the previous case, we obtain a conflation
  \[
  \begin{tikzcd}
    0 \rar & A_{i-1} \rar & A_i \oplus A_j \rar & A_j \rar & 0,
  \end{tikzcd}
  \]
  which shows $a_i + a_j = a_{i-1} + a_j$ as well. This finishes the proof of Claim 1. $\qed$

  In what follows, we consider two cases.

  {\bf (Case 1)}:
  \emph{The case $m \neq n$}.
  For every $i,j$ with $0 \leq i \leq l$ and $0 \leq j \leq l$, we have $a_i + a_j = a_{i-1} + a_j = \cdots = a_0 + a_j = a_0 + a_{j-1} = \cdots = a_0 + a_0 = 2a_0$ by Claim 1. Thus, $\MM(\EE)_2 = \{ 2 a_0 \}$ holds.
  We claim inductively that $\MM(\EE)_N = \{ N a_0 \}$ for $N \geq 2$. This is because every element $x$ in $\MM(\EE)_{N+1}$ can be written as $x = a_i + y$ for some $0 \leq i \leq l$ and $y \in \MM(\EE)_N$, thus $x = a_i + N a_0 = (a_i + a_0) + (N-1) a_0 = 2 a_0 + (N-1) a_0 = (N+1) a_0$.
  Now that $\MM(\EE)_N$ has been computed for all $N \geq 0$, the description of the Cayley quiver of $\MM(\EE)$ easily follows.

  {\bf (Case 2)}:
  \emph{The case $m=n$}.
  Note that $l = m = n$ in this case, so we only use the letter $n$. For every $i,j$ with $0 \leq i \leq n$ and $0 \leq j \leq n$, if $j \neq n$, then $a_i + a_j = a_{i-1} + a_j = \cdots = a_0 + a_j = a_0 + a_{j-1} = \cdots = 2 a_0$ by Claim 1. Thus, $\MM(\EE)_2 = \{ 2 a_0, 2 a_n \}$ holds (and later we will show $2 a_0 \neq 2 a_n$).
  The same inductive argument as in (Case 1) shows that $\MM(\EE)_N = \{ N a_0, N a_n \}$ holds for $N \geq 2$, and it suffices to show that $N a_0 \neq N a_n$ in $\MM(\EE)$ for $N \geq 2$.

  Suppose that $N a_0 = N a_n$ holds in $\MM(\EE)$, that is, $[A_0^N] = [A_n^N]$. Denote by $\EE_P$ the full subcategory of $\EE$ consisting of direct sums of $A_n = P^n$, or equivalently, the subcategory of $\EE$ consisting of projective $kQ$-modules. Note that $A_n^N \in \EE_P$ holds.
  Then the following claim holds:

  {\bf (Claim 2)}:
  \emph{In (Case 2), for every conflation}
  \[
  \begin{tikzcd}
    0 \rar & X \rar & Y \rar["\pi"] & Z \rar & 0
  \end{tikzcd}
  \]
  \emph{in $\EE$, we have that $Y \in \EE_P$ holds if and only if both $X\in \EE_P$ and $Z \in \EE_P$ hold (in other words, $\EE_P$ is a Serre subcategory of $\EE$).
  }

  \emph{Proof of Claim 2}.
  If $X \in \EE_P$ and $Z \in \EE_P$ hold, then the above sequence splits since $Z$ is projective, thus $Y \iso X \oplus Z \in \EE_P$ holds.
  Conversely, suppose that $Y \in \EE_P$ holds. Since $\pi$ induces a surjection  $\top Y \defl \top Z$, and $\top Y$ is a direct sum of $\top P = S_2$, we have that so is $\top Z$. However, $\top A_i$ contains $S_1$ as a direct summand if $i < n$. Therefore, $Z$, which is a direct sum of appropriate $A_i$'s, must be isomorphic to a direct sum of $A_n$. Thus, $Z \in \EE_P$ holds. Moreover, the above sequence splits since $Z$ is projective, thus $Y \iso X \oplus Z$.
  Clearly $\EE_P$ is closed under direct summands in $\EE$, so $X$ belongs to $\EE_P$. $\qed$

  Let us return to our situation, that is, $[A_0^N] = [A_n^N]$ holds. Since $\EE_P$ is a Serre subcategory of $\EE$ by Claim 2 and $A_n^N$ belongs to $\EE_P$, Proposition \ref{serre} implies that $A_0^N$ belongs to $\EE_P$. This is a contradiction since $A_0$ is not projective. Therefore, we have $N a_0 \neq N a_n$, which completes the proof.
\end{proof}

\section{Counter-examples}
In this section, we give examples of \emph{bad behavior} of exact categories on several topics which we have studied in the previous sections.

\subsection{On the poset of admissible subobjects}
In this subsection, we collect some counter-examples on the poset $\PP(X)$ (see Section \ref{posetdef} for the detail)

\subsubsection{Subobject posets which are not lattices}
The following example shows that \emph{$\PP(X)$ is not necessarily a lattice}.
\begin{example}\label{nonlattice1}
  Let $k$ be a field and $\EE$ a category of $k$-vector spaces whose dimensions are not equal to $1$ and $3$. Then $\EE$ is an extension-closed subcategory of $\mod k$, thus an exact category. Denote by $X$ the $6$-dimensional vector space with basis $x_1, \cdots, x_6$, and put $A:= \langle x_1,x_2,x_3,x_4 \rangle$ and $B:= \langle x_1, x_2, x_3, x_5 \rangle$.
  Then since $X/A$ and $X/B$ has dimension $2$, both $A$ and $B$ are admissible subobjects of $X$.
  Now the set-theoretic intersection of $A$ and $B$ is a $3$-dimensional subspace $C:=\langle x_1, x_2, x_3 \rangle$, which does not belong to $\EE$, thus lower bounds of $A$ and $B$ in $\PP_\EE(X)$ are precisely subspaces of $C$ whose dimensions are exactly two or zero. Then $A$ and $B$ do not have the greatest lower bound, since there are many two-dimensional subspaces of $C$.
\end{example}
The following is also an example in which $\PP(X)$ is not a lattice, although the ambient category is \emph{pre-abelian}.
\begin{example}\label{nonlattice2}
  Let $k$ be a field and $\Lambda$ an algebra
  $\Lambda:=k (1 \xleftarrow{\alpha} 2 \xleftarrow{\beta}3 ) / (\beta \alpha)$. Put $\EE:=\proj\Lambda$, then $\EE$ has the natural exact structure induced from $\mod\Lambda$, and every conflation splits in $\EE$. It follows that a morphism $\iota \colon A \to B$ in $\EE$ is an inflation if and only if $\iota$ is a section.
  Moreover, since $\gl \Lambda = 2$ holds, kernels of morphisms between projective modules are projective, thus $\EE$ is pre-abelian.

  We claim that $\PP_\EE(\Lambda_\Lambda)$ is not a lattice. By the previous argument, we can identify $\PP_\EE(\Lambda)$ with the poset of submodules of $\Lambda$ which are  direct summands of $\Lambda$. We can write $\Lambda$ as follows:
  \begin{equation}\label{compfactor}
  \Lambda_\Lambda = P_1 \oplus P_2 \oplus P_3
  = 1 \oplus \begin{smallmatrix} 2\\1 \end{smallmatrix} \oplus \begin{smallmatrix} 3 \\ 2 \end{smallmatrix}
  \end{equation}
  Here $P_i$ is the indecomposable projective modules corresponding to the vertex $i$, and the rightmost notation indicates composition series of modules.

  Consider the following two maps:
  \begin{align*}
    \iota_1 \colon P_1 \oplus P_2 &
    \xrightarrow{\begin{sbmatrix}1 & 0 \\ 0 & 1 \\ 0 & 0 \end{sbmatrix}}
     P_1 \oplus P_2 \oplus P_3 \\
    \iota_2 \colon P_1 \oplus P_2 &
    \xrightarrow{\begin{sbmatrix}1 & 0 \\ 0 & 1 \\ 0 & \beta \end{sbmatrix}}
     P_1 \oplus P_2 \oplus P_3
  \end{align*}
  Here $\beta \colon P_2 \to P_3$ denotes the multiplication map $\beta \cdot (-)$. We can check by matrix elimination that these maps are sections, thus $N_i := \im \iota_i$ belongs to $\PP_\EE(\Lambda)$ for each $i$. In the notation of (\ref{compfactor}), each $N_i$ looks like the following:
  \begin{align*}
    N_1 &= \{ (a, \begin{smallmatrix}b \\ c \end{smallmatrix}, \begin{smallmatrix}0 \\ 0 \end{smallmatrix}) \in P_1 \oplus P_2 \oplus P_3 \, | \, a,b,c \in k \} \\
    N_2 &= \{ (a, \begin{smallmatrix}b \\ c \end{smallmatrix}, \begin{smallmatrix}0 \\ b \end{smallmatrix}) \in P_1 \oplus P_2 \oplus P_3 \, | \, a,b,c \in k \}
  \end{align*}
  We claim that $N_1$ and $N_2$ do not have the greatest lower bound in $\PP_\EE(\Lambda)$. Let $N$ be a lower bound for $N_1$ and $N_2$ in $\PP_\EE(\Lambda)$. Then $N$ must be contained in the \emph{set-theoretic} intersection $N_1 \cap N_2$:
  \[
  N_1 \cap N_2 =
  1 \oplus \begin{smallmatrix} \\1 \end{smallmatrix} \oplus 0
  =
  \{ (a, \begin{smallmatrix}0 \\ c \end{smallmatrix}, \begin{smallmatrix}0 \\ 0 \end{smallmatrix}) \in P_1 \oplus P_2 \oplus P_3 \, | \, a,c \in k \},
  \]
  which is a two-dimensional vector space spanned by $e_1$ and $\alpha$. Here $(a, \begin{smallmatrix}0 \\ c \end{smallmatrix}, \begin{smallmatrix}0 \\ 0 \end{smallmatrix})$ in the above notation corresponds to $a e_1 + c \alpha$. Consider the quotient module $\Lambda/N_1 \cap N_2$:
  \[
  \Lambda/N_1 \cap N_2 = S_2 \oplus P_3 = 2 \oplus \begin{smallmatrix} 3 \\ 2 \end{smallmatrix}
  \]
  This is not projective, thus we have $N_1 \cap N_2 \notin \PP_\EE(\Lambda)$. Hence we have $N \subsetneq N_1 \cap N_2$, so $\dim N = 0$ or $\dim N = 1$ holds. Therefore, to show that the meet of $N_1$ and $N_2$ does not exist in $\PP_\EE(\Lambda)$, it suffices to show that there exist two distinct one-dimensional submodules of $N_1 \cap N_2$ which belong to $\PP_\EE(\Lambda)$.
  To show this, consider the following two maps:
  \begin{align*}
    i_1 \colon P_1 &
    \xrightarrow{{}^t[1,0,0]}
     P_1 \oplus P_2 \oplus P_3 \\
    i_2 \colon P_1 &
    \xrightarrow{^t [1,\alpha,0]}
     P_1 \oplus P_2 \oplus P_3
  \end{align*}
  The images of these maps are one-dimensional submodules generated by $e_1$ and $e_1+\alpha$ respectively. By matrix elimination, we can show that these are direct summands of $\Lambda$, thus both belong to $\PP_\EE(\Lambda)$. Therefore, $\PP_\EE(\Lambda)$ is not a lattice.
\end{example}

\subsubsection{Subobject posets are lattices but not modular lattices}
Recall from Proposition \ref{quasiabelian} that if $\EE$ is a quasi-abelian category with the maximal exact structure, subobject posets are in fact lattices.
However, there exist examples which are quasi-abelian but these lattices are not necessarily modular.
For such examples, we refer the reader to Section \ref{nonulpsec}, since the modularity implies the unique length property (Corollary \ref{modulp}).

\subsection{(JHP) and the unique length property}\label{jhpex}
Next, let us see some counter-examples on the unique factorization properties on exact categories.
\subsubsection{Lengths are not unique}\label{nonulpsec}
We collect some examples in which the unique length property fails (see Section \ref{basicdef}).
Example \ref{exa} in the introduction is one of such example, but is rather artificial and not idempotent complete.
We give several idempotent complete examples (actually torsion-free classes, so quasi-abelian) which do not satisfy the unique length property. Actually we can obtain more examples in a quiver of type A by using the results in Section \ref{typeasec}.
\begin{example}\label{nonulp1}
  Let $Q$ be a quiver $1 \ot 2 \ot 3 \to 4$. Then the Auslander-Reiten quiver of $\mod kQ$ is as follows. Here we use the notation introduced in Proposition \ref{indectrans}.
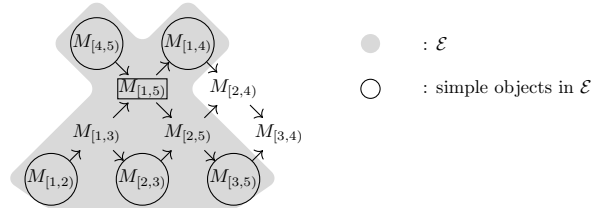
  \[
  \begin{tikzpicture}
    [scale=0.6,  every node/.style={scale=.7}]
    \fill[gray!30, rounded corners] (-1,0) -- (1,2) -- (0,3) -- (1,4) -- (2,3) -- (3,4) -- (4,3) -- (3,2) -- (5,0) -- (4,-.7) -- (0,-.7) -- cycle;

    \node[circle, draw=black, inner sep = 0] (12) at (0,0) {$M_{\hc{1,2}}$};
    \node (34) at (5,1) {$M_{\hc{3,4}}$};
    \node[circle, draw=black, inner sep = 0] (14) at (3,3) {$M_{\hc{1,4}}$};
    \node[circle, draw=black, inner sep = 0] (35) at (4,0) {$M_{\hc{3,5}}$};
    \node (24) at (4,2) {$M_{\hc{2,4}}$};
    \node[rectangle, draw=black, inner sep = .2, outer sep=.2em] (15) at (2,2) {$M_{\hc{1,5}}$};
    \node (25) at (3,1) {$M_{\hc{2,5}}$};
    \node (13) at (1,1) {$M_{\hc{1,3}}$};
    \node[circle, draw=black, inner sep = 0] (45) at (1,3) {$M_{\hc{4,5}}$};
    \node[circle, draw=black, inner sep = 0] (23) at (2,0) {$M_{\hc{2,3}}$};

    \draw[->] (12) -- (13);
    \draw[->] (13) -- (23);
    \draw[->] (13) -- (15);
    \draw[->] (45) -- (15);
    \draw[->] (23) -- (25);
    \draw[->] (15) -- (14);
    \draw[->] (15) -- (25);
    \draw[->] (14) -- (24);
    \draw[->] (25) -- (24);
    \draw[->] (25) -- (35);
    \draw[->] (24) -- (34);
    \draw[->] (35) -- (34);

    \fill[gray!30] (7,3) circle (6pt);
    \node[right] at (8,3) {: $\EE$};

    \draw (7,2) circle (6pt);
    \node[right] at (8,2) {: simple objects in $\EE$};
  \end{tikzpicture}
  \]
  Define $\EE$ as the additive subcategory of $\mod kQ$ corresponding to the gray region, then it is closed under extensions (and actually a torsion-free class) in $\mod kQ$.
  We can check that $\EE$ has $5$ simples indicated by circles: $\simp \EE = \{ M_{\hc{1,2}}, M_{\hc{2,3}}, M_{\hc{3,5}}, M_{\hc{4,5}}, M_{\hc{1,4}}\}$.
  Now consider an object $X:=M_{\hc{1,5}}$.
  Roughly speaking, the gray region below $X$ looks like a module category of an $A_3$ quiver, thus $X$ seems to have length $3$, and the gray region above $X$ looks like a module category of an $A_2$ quiver, thus $X$ seems to have length $2$. In fact, we have the following two composition series of $X$ in $\EE$:
  \begin{align*}
    &0 < M_{\hc{4,5}} < X , \text{ and} \\
    &0 < M_{\hc{1,2}} < M_{\hc{1,3}} < X.
  \end{align*}
  Thus, lengths of $X$ are not unique.
\end{example}
Here is another example, which we have already encountered.
\begin{example}\label{nonulp2}
  Consider Example \ref{main1ex} and let $w$ be the fourth one in Table \ref{1324ex}, that is, $w = 42513$. Then $\FF(w)$ is the additive subcategory corresponding to $\{A,B,C,D,E,F\}$ depicted as follows:
  
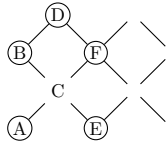
\[
  \begin{tikzpicture}
    [scale=0.5,  every node/.style={scale=.7}]
    \node[circle, draw, inner sep=.1em] (12) at (0,0) {A};
    \node[circle, draw, inner sep=.1em] (34) at (0,2) {B};
    \node (14) at (1,1) {C};
    \node[circle, draw, inner sep=.1em] (35) at (1,3) {D};
    \node[circle, draw, inner sep=.1em] (24) at (2,0) {E};
    \node[circle, draw, inner sep=.1em] (15) at (2,2) {F};
    \node (25) at (3,1) {};
    \node (13) at (3,3) {};
    \node (45) at (4,0) {};
    \node (23) at (4,2) {};

    \draw (12) -- (14);
    \draw (34) -- (14);
    \draw (34) -- (35);
    \draw (14) -- (24);
    \draw (14) -- (15);
    \draw (35) -- (15);
    \draw (24) -- (25);
    \draw (15) -- (25);
    \draw (15) -- (13);
    \draw (25) -- (45);
    \draw (25) -- (23);
    \draw (13) -- (23);
  \end{tikzpicture}
  \]
  We have that $\simp \FF(w) = \{A,B,D,E,F\}$. Consider $X:= C \oplus D$. Then a conflation $A \infl C \defl E$ shows that $C$ has length $2$, hence $X = C \oplus D$ has length $3$. On the other hand, a conflation $B \infl X \defl F$ implies that $X$ has length $2$. Thus, lengths of $X$ are not unique.
\end{example}

\subsubsection{Length are unique but (JHP) fails}
Here is an example  of an exact category which satisfies the unique length property, but does not satisfy (JHP).
\begin{example}\label{ulpnonjh}
  Let $\FF$ be a hereditary torsion-free class of abelian length categories. Then $\FF$ satisfies the unique length property by Corollary \ref{htfulp}. Typically, $\FF$ arises in the following way:
  Take any artin algebra $\Lambda$ and chose any set of simple modules $\SS = \{S_1, \cdots, S_l \}$. Define $\FF$ as follows:
    \[
    \FF := \{ X \in \mod\Lambda \, | \, \Hom_\Lambda(S_i, X)=0 \text{ for $1 \leq i \leq l$} \}.
    \]
  Then $\FF$ is a torsion-free class, and the corresponding torsion class is the smallest Serre subcategory containing $\SS$. Thus, $\FF$ is a hereditary torsion-free class.

  We have already encountered such examples which do not satisfy (JHP):
  $\EE_1$ in Example \ref{introex}, and $\FF(c^2)$ and $\FF(c^2 s_1)$ in Example \ref{main1ex}.
\end{example}

\subsubsection{Examples which satisfy (JHP)}
Finally, we collect examples which satisfy (JHP) for the convenience of the reader.
\begin{example}\label{jhex}
  The following examples satisfy (JHP).
  \begin{enumerate}
    \item Krull-Schmidt categories \emph{together with the split exact structure}. This follows by Proposition \ref{splitmon} and the fact that the uniqueness of direct sum decomposition holds.
    \item Abelian length categories. This is because the Jordan-H\"older theorem holds in any abelian category, see e.g. \cite[p.92]{st}.
    \item Torsion(-free) classes of $\mod\Lambda$ for a Nakayama algebra $\Lambda$ (Corollary \ref{nakayamamain}).
    \item Torsion-free classes of $\mod kQ$ for a quiver $Q$ of type A which satisfies the condition given in Corollary \ref{jhtypea}. Explicit examples are given in Example \ref{main1ex} and $\EE_1$ in Example \ref{introex}.
    \item The category $\FF(\Delta)$ of modules with $\Delta$-filtrations over a quasi-hereditary algebra, or more generally, over a standardly stratified algebra (see e.g. \cite[Proposition 1.2]{pr}). Simple objects in $\FF(\Delta)$ are precisely standard modules.
  \end{enumerate}
\end{example}

\subsection{Non-cancellative Grothendieck monoids}\label{noncancel}
In this subsection, we will give some examples in which the Grothendieck monoids are \emph{not finitely generated} or \emph{not cancellative}.

\subsubsection{Functorially finite torsion-free class, but neither finitely generated nor cancellative}\label{kroex}
In a length exact category, its Grothendieck monoid is atomic, thus it is not finitely generated as a monoid if and only if there exists infinitely many non-isomorphic simple objects (Proposition \ref{atomicfg}). We will give such an example.

Let $k$ be an algebraically closed field and $Q$ a Kronecker quiver, namely, $Q \colon 1 \leftleftarrows 2$. We refer the reader to \cite[Section VIII.7]{ARS} for the structure of $\mod kQ$.
We denote by $\EE$ the subcategory of $\mod kQ$ consisting of modules which does not contain $S_2$ (simple modules supported at $2$) as a direct summand. Then $\EE$ is closed under extensions, thus an exact category. Actually, $\EE$ is a torsion-free class associated with an APR cotilting module $I_1 \oplus \tau S_2$, thus it is a functorially finite torsion-free class (with infinitely many indecomposables).

Recall that for each element $x=[x_1:x_2] \in \mathbb{P}^1(k)$ in the projective line $\mathbb{P}^1(k)$, we have a regular module $R_x$ with dimension vector $(1,1)$, that is, $R_x$ is the following representation of $Q$.
\[
\begin{tikzcd}
  k & k \ar[l, shift left, "x_2"] \ar[l, shift right, "x_1"']
\end{tikzcd}
\]
This assignment is injective, that is, we have $R_x \iso R_y$ if and only if $x=y$ in $\mathbb{P}^1(k)$. Now we claim the following:
\[
\simp \EE = \{ S_1 \} \cup \{R_x \, | \, x \in \mathbb{P}^1(k)\}.
\]
Actually, for any indecomposable preprojective module $X$ except  $S_1$, it can be shown by explicit calculation that there exists an exact sequence
\[
0 \to S_1 \to X \to R \to 0
\]
such that $R$ is a non-zero regular module. Thus, such $X$ is not simple in $\EE$.
Moreover, any indecomposable regular module is known to be written as a finite extension of $R_x$ for some $x \in \mathbb{P}^1(k)$. Therefore, indecomposable modules except $S_1$ and $R_x$'s are not simple.
On the other hand, consider the monoid homomorphism $\udim \colon \MM(\EE) \to \N^2$ induced by the dimension vector with respect to $\{[S_1],[S_2]\}$ (see Definition \ref{dimvecdef}) and its image $\udim \EE := \udim(\MM(\EE))$.
Then each $R_x$ must be a simple object, because $\udim R_x = (1,1)$ is an atom of $\udim \EE$. Thus, $R_x$ is simple in $\EE$ for each $x \in \mathbb{P}^1 (k)$.

Therefore, $\MM(\EE)$ has infinitely many atoms (since $k$ is an infinite field), so it is \emph{not finitely generated}.

Furthermore, we claim that $\MM(\EE)$ is not cancellative.
In fact, for each $x = [x_1:x_2] \in \mathbb{P}^1(k)$, we have an exact sequence
\[
\begin{tikzcd}[column sep= large]
  0 \rar & S_1 \rar["\begin{sbmatrix}-x_2\\ x_1 \end{sbmatrix}"] &  P_2 \rar & R_x \rar & 0
\end{tikzcd}
\]
where $P_2$ is an indecomposable projective module corresponding to $2 \in Q$, and the left map is an embedding of $S_1$ into the socle $S_1 \oplus S_1$ of $P_2$.
This shows that $[S_1] + [R_x] = [P_2]$ holds  for every $x \in \mathbb{P}^1(k)$. Since $[R_x] \neq [R_y]$ in $\MM(\EE)$ for $x \neq y$, we must conclude that $\MM(\EE)$ is \emph{not cancellative}.

Moreover, this gives an example such that \emph{non-isomorphic simples may represent the same element in the Grothendieck group}. By Proposition \ref{setofdim}, we have an isomorphism $\udim \colon \KK_0(\EE) \xrightarrow{\sim} \Z^2$. Thus, for each $x,y \in \mathbb{P}^1(k)$, we have $[R_x] = [R_y]$ in $\KK_0(\EE)$.

\subsubsection{Only finitely many indecomposables, but not cancellative}\label{finex}
In the previous example, the category we considered has infinitely many indecomposable objects. If an exact category $\EE$ has finitely many indecomposables, then $\MM(\EE)$ is finitely generated by a trivial reason, but it \emph{may not be cancellative} as we shall see.

Let $\Lambda$ be the following algebra, which is defined by the ideal quotient of the path algebra:
\[
\Lambda := k \left(
\begin{tikzcd}
  1 \ar[loop, out = 120, in = 60, looseness=5, "\beta", pos = 0.15] & 2 \lar["\alpha"']
\end{tikzcd}
\right) / (\beta^2)
\]
Then the Auslander-Reiten quiver of $\mod \Lambda$ is as follows:
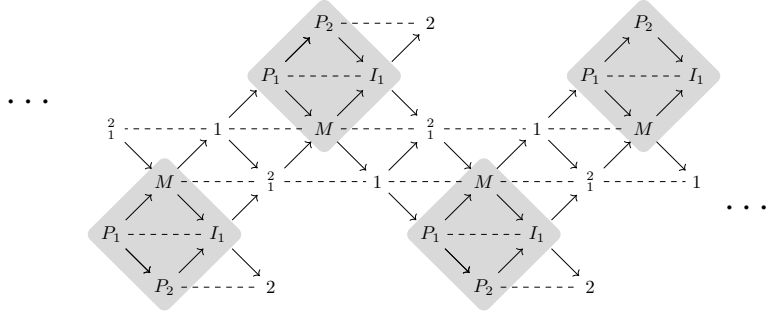
\[
\scalebox{0.7}{
\begin{tikzpicture}
  \fill[gray!30, rounded corners]
  (0,0.5) -- (1.5,-1) -- (0,-2.5) -- (-1.5,-1) -- cycle;
  \fill[gray!30, rounded corners]
  (3,3.5) -- (4.5,2) -- (3,.5) -- (1.5,2) -- cycle;
  \fill[gray!30, rounded corners]
  (6,0.5) -- (7.5,-1) -- (6,-2.5) -- (4.5,-1) -- cycle;
  \fill[gray!30, rounded corners]
  (9,3.5) -- (10.5,2) -- (9,.5) -- (7.5,2) -- cycle;

  \node at (-2.5,1.5) {\Huge $\cdots$};
  \node at (11,-.5) {\Huge $\cdots$};

  \node (21a) at (-1,1) {$\begin{smallmatrix}2 \\ 1 \end{smallmatrix}$};
  \node (11) at (-1,-1) {$P_1$};
  \node (121) at (0,0) {$M$};
  \node (211) at (0,-2) {$P_2$};
  \node (1) at (1,1) {1};
  \node (2121) at (1,-1) {$I_1$};
  \node (11b) at (2,2) {$P_1$};
  \node (21) at (2,0) {$\begin{smallmatrix}2 \\ 1 \end{smallmatrix}$};
  \node (2) at (2,-2) {2};
  \node (211b) at (3,3) {$P_2$};
  \node (121b) at (3,1) {$M$};
  \node (2121b) at (4,2) {$I_1$};
  \node (1b) at (4,0) {1};
  \node (2b) at (5,3) {2};
  \node (21b) at (5,1) {$\begin{smallmatrix}2 \\ 1 \end{smallmatrix}$};
  \node (11c) at (5,-1) {$P_1$};
  \node (121c) at (6,0) {$M$};
  \node (211c) at (6,-2) {$P_2$};
  \node (1c) at (7,1) {1};
  \node (2121c) at (7,-1) {$I_1$};
  \node (11d) at (8,2) {$P_1$};
  \node (21c) at (8,0) {$\begin{smallmatrix}2 \\ 1 \end{smallmatrix}$};
  \node (2c) at (8,-2) {2};
  \node (211d) at (9,3) {$P_2$};
  \node (121d) at (9,1) {$M$};
  \node (2121d) at (10,2) {$I_1$};
  \node (1d) at (10,0) {1};

  \draw[->] (21a) -- (121);
  \draw[->] (11) -- (121);
  \draw[->] (11) -- (211);
  \draw[->] (121) -- (1);
  \draw[->] (121) -- (2121);
  \draw[->] (211) -- (2121);
  \draw[->] (1) -- (11b);
  \draw[->] (1) -- (21);
  \draw[->] (11) -- (211);
  \draw[->] (2121) -- (21);
  \draw[->] (2121) -- (2);
  \draw[->] (11b) -- (211b);
  \draw[->] (11b) -- (121b);
  \draw[->] (21) -- (121b);
  \draw[->] (121b) -- (1b);
  \draw[->] (121b) -- (2121b);
  \draw[->] (211b) -- (2121b);
  \draw[->] (1b) -- (11c);
  \draw[->] (1b) -- (21b);
  \draw[->] (11b) -- (211b);
  \draw[->] (2121b) -- (21b);
  \draw[->] (2121b) -- (2b);

  \draw[->] (11c) -- (211c);
  \draw[->] (11c) -- (121c);
  \draw[->] (21b) -- (121c);
  \draw[->] (121c) -- (1c);
  \draw[->] (121c) -- (2121c);
  \draw[->] (211c) -- (2121c);

  \draw[->] (1c) -- (11d);
  \draw[->] (1c) -- (21c);
  \draw[->] (11c) -- (211c);
  \draw[->] (2121c) -- (21c);
  \draw[->] (2121c) -- (2c);
  \draw[->] (11d) -- (211d);
  \draw[->] (11d) -- (121d);
  \draw[->] (21c) -- (121d);
  \draw[->] (121d) -- (1d);
  \draw[->] (121d) -- (2121d);
  \draw[->] (211d) -- (2121d);

  \draw[dashed] (11) -- (2121);
  \draw[dashed] (11b) -- (2121b);
  \draw[dashed] (11c) -- (2121c);
  \draw[dashed] (11d) -- (2121d);
  \draw[dashed] (121) -- (21);
  \draw[dashed] (121b) -- (21b);
  \draw[dashed] (121c) -- (21c);
  \draw[dashed] (211) -- (2);
  \draw[dashed] (211b) -- (2b);
  \draw[dashed] (211c) -- (2c);
  \draw[dashed] (1) -- (121b);
  \draw[dashed] (1b) -- (121c);
  \draw[dashed] (1c) -- (121d);
  \draw[dashed] (21) -- (1b);
  \draw[dashed] (21b) -- (1c);
  \draw[dashed] (21c) -- (1d);
  \draw[dashed] (1) -- (21a);
\end{tikzpicture}
}
\]
where we write composition factors of modules, and $P_1 = \begin{smallmatrix} 1 \\ 1 \end{smallmatrix}$, $P_2 = \begin{smallmatrix} 2 \\ 1 \\ 1\end{smallmatrix}$, $I_1 = \begin{smallmatrix}2 \\ 1 && 2 \\ &1\end{smallmatrix}$ and $M := \begin{smallmatrix}1 && 2 \\ &1 \end{smallmatrix}$.
Let $\EE$ denote the additive subcategory of $\mod\Lambda$ corresponding to these gray regions: $\ind \EE = \{P_1, P_2, I_1, M \}$.
Then $\EE$ is shown to be closed under extensions in $\mod\Lambda$
\footnote{This does not seem to be so trivial. One categorical way to show this is to use results in \cite{en1,en2}. By \cite[Corollary 3.10]{en2}, one can endow $\EE$ with the exact structure which corresponds to the Auslander-Reiten quiver of $\EE$ drawn below. In this exact structure, $\EE$ has a progenerator $P_1 \oplus P_2 = \Lambda$, and one can consider the Morita type embedding $\EE(\Lambda, -) \colon \EE \to \mod\Lambda$, which is nothing but the inclusion functor. Then by \cite[Proposition 2.8]{en1}, its image, $\EE$, is an extension-closed subcategory of $\mod\Lambda$.}
, and the Auslander-Reiten quiver of the exact category $\EE$ is as follows:
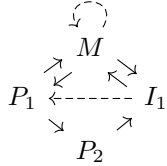
\[
\begin{tikzcd}[column sep = tiny, row sep = tiny]
  & M\ar[rd, shift left = .7ex] \ar[ld, shift left = .7ex]\ar[loop, in = 120, out = 60, dashed, looseness=4]\\
  P_1 \ar[rd] \ar[ru, shift left = .7ex]& & I_1 \ar[lu, shift left = .7ex] \ar[ll,dashed]\\
  & P_2 \ar[ru]
\end{tikzcd}
\]
By checking subobjects, it can be shown that all 4 indecomposable objects in $\EE$ are simple objects in $\EE$. However, by the Auslander-Reiten quiver of $\mod\Lambda$, it can be checked that we have the following conflations in $\EE$:
\[
\begin{tikzcd}[row sep = 0]
  0 \rar & P_1 \rar & M \oplus P_2 \rar & I_1 \rar & 0, \\
  0 \rar & M \rar & P_1 \oplus I_1 \rar & M \rar & 0, \\
  0 \rar & P_1 \rar & P_2 \oplus P_2 \rar & I_1 \rar & 0.
\end{tikzcd}
\]
This implies the equality $[M] + [M] = [M] + [P_2] = [P_2] + [P_2]$ in $\MM(\EE)$. Since $[M]\neq[P_2]$ by Proposition \ref{simpleatom}, we must have that $\MM(\EE)$ is \emph{not cancellative}.

\section{Problems}
In this section, we collect some open problems on several topics in this paper.

As we saw in Section 7, the computation of the Grothendieck monoid is rather difficult if it fails to be cancellative.
\begin{problem}
  Let $\EE$ be an exact category. Is there any criterion to check whether $\MM(\EE)$ is cancellative?
\end{problem}
This leads to the following question.
\begin{problem}
  Is there an example of an exact category $\EE$ which satisfies the following conditions:
  \begin{enumerate}
    \item $\EE$ has finitely many indecomposable objects up to isomorphisms.
    \item $\EE = {}^\perp U$ for a cotilting module $U$ over an artin algebra $\Lambda$, or more strongly, $\EE$ is a torsion-free class of $\mod \Lambda$.
    \item $\MM(\EE)$ is not cancellative.
  \end{enumerate}
\end{problem}
If we drop (1) or (2), then we have such an example (Section \ref{kroex} and \ref{finex} respectively).

In what follows, let $\FF$ be a functorially finite torsion-free class of $\mod\Lambda$ for an artin algebra $\Lambda$. The most general (and thus difficult) problem is the following:
\begin{problem}\label{comprob}
  Compute the Grothendieck monoid $\MM(\FF)$, more precisely, draw the Cayley quiver as we did in Proposition \ref{compex}.
\end{problem}

As a first approximation to $\MM(\FF)$, the following problem naturally occurs, which is of interest in its own right.
\begin{problem}
  For a given torsion-free class $\FF$, classify (or count) simple objects in $\FF$.
\end{problem}

Of course, we cannot expect the general classification of simples, but it may be done when $\Lambda$ and $\FF$ are given explicitly.

The cancellative quotient $\MM(\FF)_\can$ is easier to handle with by Corollary \ref{setofdim}: it is nothing but the monoid of dimension vectors $\udim \FF$.
Moreover, it is a finitely generated submonoid of $\N^n$ by Proposition \ref{affinemonoid}. Such a monoid is called a \emph{positive affine monoid}, and this class is investigated in the theory of combinatorial commutative algebra and toric geometry via its monoid algebra. We refer the reader to \cite{brg} for the details on affine monoids and affine monoid algebras.
\begin{problem}\label{combiprob}
  For a given torsion-free class $\FF$, investigate the combinatorial property of the affine monoid $\udim \FF$, the monoid of dimension vectors of modules in $\FF$. In particular,
  \begin{itemize}
    \item Is this monoid normal?
    \item Describe its minimal generating set (this is related to simple objects in $\FF$).
    \item When is this monoid homogeneous (this is related to the unique length property of $\FF$)?
    \item Compute invariants of this monoid, such as extreme rays, the class group, support hyperplanes, and so on.
  \end{itemize}
\end{problem}

Finally, we consider topics in Section 7. Let $Q$ be a quiver of type A (or more generally, any Dynkin type, see Remark \ref{forthcoming}).
Theorem \ref{sortabletorf} gives an explicit description of torsion-free classes of $\mod kQ$. Thus, it is reasonable to expect that Problems \ref{comprob} and \ref{combiprob} may be easier in this case.
In addition, the author does not know when the uniqueness of lengths holds, except Example \ref{ulpnonjh}.
\begin{problem}
  In the situation of Theorem \ref{sortabletorf}, is there a combinatorial criterion for $w$ to check whether $\FF(w)$ has the unique length property?
\end{problem}

Also the following (purely combinatorial) problem is of interest.
\begin{problem}
  For a Coxeter element $c$, is there any closed formula which gives the number of torsion-free classes satisfying (JHP), or equivalently, the number of $c$-sortable elements $w$ such that $\#\supp(w) = \#\Binv(w)$ holds?
\end{problem}

\begin{appendix}

\section{Preliminaries on monoids}
We collect some basic definitions and properties on monoids needed in this paper. Recall that \emph{monoids are always assumed to be commutative monoids with units}, and that {we use an additive notation with unit $0$}.

\subsection{Basic definitions}
First, we collect some basic definitions on monoids.
\begin{definition}
  Let $M$ be a monoid.
  \begin{enumerate}
    \item $M$ is \emph{reduced} if $a + b = 0$ implies $a = b = 0$ for $a,b\in M$.
    \item $M$ is \emph{cancellative} if $a + x = a + y$ implies $x = y$ for $a,x,y\in M$.
  \end{enumerate}
\end{definition}

There exists a natural pre-order $\leq$ on any monoid, defined as follows.
\begin{definition}
  Let $M$ be a monoid. We define $x \leq y$ if there exists some $a \in M$ such that $y = x + a$.
\end{definition}
It can be checked that $x \leq y$ if and only if $y \in x + M$ if and only if $y + M \subset x + M$.
This pre-order is sometimes called \emph{Green's pre-order} in semigroup theory, e.g. \cite{gri1,gri2}, or \emph{divisibility pre-order} in the multiplicative theory of integral domains, e.g. \cite{ghk}.

\begin{definition}
  A monoid $M$ is called \emph{naturally partially ordered} if the pre-order $\leq$ on $M$ is a partial order, that is, $x \leq y$ and $y \leq x$ implies $x = y$.
\end{definition}
The following properties can be easily checked.
\begin{proposition}\label{natpo}
  Let $M$ be a monoid. Then $M$ is reduced if and only if $0 \leq x \leq 0$ implies $x = 0$ for every $x \in M$. In particular, a naturally partially ordered monoid is reduced.
\end{proposition}
Although we cannot naively take quotients of monoids as in abelian groups since there may not exist additive inverses, we can obtain some kind of \emph{quotient monoids} by considering quotient sets with respect to \emph{monoid congruences}, defined as follows.
\begin{definition}
  Let $M$ be a monoid. The equivalence relation $\sim$ is called a \emph{(monoid) congruence} if $x \sim y$ implies $a+x \sim a+y$ for every $a,x,y \in M$. In this case, the quotient set $M/\sim$ of equivalence classes naturally has the structure of a monoid.
\end{definition}
We often use the smallest monoid congruence generated by some binary relation. See  e.g. \cite[Propositions I.4.1, I.4.2]{gri1} for the details.
\begin{proposition}\label{gencong}
  Let $M$ be a monoid and $\sim$ an arbitrary (not necessarily equivalence) binary relation on $M$. Then there exists the smallest monoid congruence $\approx$ which contains $\sim$, that is, $x \sim y$ implies $x \approx y$ for $x,y \in M$.
\end{proposition}

\subsection{Factorization properties on monoids}
Let us define several notions on the \emph{unique factorization property} on monoids. We refer the reader to \cite{ghk} for the details in this subsection (be aware that \emph{monoids} in \cite{ghk} are assumed to be \emph{cancellative}, but contents in this subsection hold in non-cancellative monoids).
The most typical one is the freeness of monoids, which corresponds to (JHP) for exact categories by our main result (Theorem \ref{JHchar}).
\begin{definition}
  Let $M$ be a monoid.
  \begin{enumerate}
    \item We say that a monoid is \emph{finitely generated} if there exists a finite subset of $M$ which generates $M$.
    \item For a subset $A$ of $M$, we say that $M$ is \emph{free} on $A$ if every element $x\in M$ can be written as a finite sum of elements in $A$ in a unique way up to permutations. A monoid is called \emph{free} if it is free on some subset of $M$.
  \end{enumerate}
\end{definition}
For a set $A$, we denote by $\N^{(A)}$ the submonoid $\bigoplus_{a \in A} \N a$ of the free abelian group $\Z^{(A)} := \bigoplus_{a \in A} \Z a$ with basis $A$, which consists of finite sums of non-negative linear combinations of elements in $A$. It is easy to see that $\N^{(A)}$ is free on $A$ and that a monoid is free if and only if it is isomorphic to $\N^{(A)}$ for some set $A$.
If $A$ is a finite set, then we often write $\N^{A} = \N^{(A)}$ and $\Z^{A} = \Z^{(A)}$. Moreover it is obvious that a free monoid is reduced.

Now let us consider elements of a monoid which cannot be decomposed into smaller ones. Here, for simplicity, we only consider \emph{reduced} monoids (this is a reasonable assumption since Grothendieck monoids are reduced by Proposition \ref{mreduced}).
\begin{definition}\label{monoiddef}
  Let $M$ be a reduced monoid.
  \begin{enumerate}
    \item A non-zero element $x$ of $M$ is called an \emph{atom} if $x = y + z$ implies either $y=0$ or $z=0$ for $y,z \in M$. We denote by $\Atom M$ the set of all atoms in $M$.
    \item $M$ is called \emph{atomic} if $\Atom M$ generates $M$, that is, every element of $M$ is a finite sum of atoms.
    \item $M$ is called \emph{factorial} if every element can be expressed as a finite sum of atoms, and this expression is unique up to permutations.
    \item $M$ is called \emph{half-factorial} if it is atomic, and for every element $x$ and expressions
    \[
    x = a_1 + \cdots + a_n
    \]
    with $a_i \in \Atom M$, the number $n$ depends only on $x$.
  \end{enumerate}
\end{definition}

The following observations can be proved directly from the definitions.
\begin{proposition}\label{factprop}
  Let $M$ be a reduced monoid. Then the following hold.
  \begin{enumerate}
    \item If $M$ is generated by a subset $A$, then $\Atom M \subset A$ holds. In particular, $\Atom M$ is a finite set for a finitely generated monoid $M$.
    \item $M$ is finitely generated and atomic if and only if $M$ is atomic and $\Atom M$ is a finite set.
    \item If $M$ is free on a subset $A$ of $M$, then $A = \Atom M$ holds.
    \item If $M$ is free, then it is atomic, cancellative and factorial.
    \item $M$ is free if and only if $M$ is factorial.
  \end{enumerate}
\end{proposition}

\subsection{Group completion}
For a given monoid $M$, there is the universal construction which transforms $M$ into a group. We call it a \emph{group completion} in this paper.
\begin{definition}
  Let $M$ be a monoid. Then the \emph{group completion} $\gp M$ is an abelian group $\gp M$ together with a map $\iota \colon M \to \gp M$ which satisfies the following universal property:
  \begin{enumerate}
    \item $\iota$ is a monoid homomorphism.
    \item Every monoid homomorphism $f \colon M \to G$ into a group $G$ factors uniquely $\iota$, that is, there exists a unique group homomorphism $\ov{f} \colon \gp M \to G$ which satisfies $f = \iota \ov{f}$.
  \end{enumerate}
\end{definition}
\begin{example}\label{freecompl}
  Let $M$ be a monoid which is free on $A$. Then $\gp M$ is a free abelian group with basis $A$, that is, $\gp (\N^{(A)}) = \Z^{(A)}$. The original basis $A$ can be reconstructed from $M$ since $A = \Atom M$, but this information is lost when taking the group completion, since bases of the free abelian group is far from being unique. This is one reason why we consider monoids, not groups.
\end{example}
The explicit construction of the group completion is given as follows.
\begin{proposition}
  Let $M$ be a monoid. Define an equivalence relation $\sim$ on the set $M \times M$ by
  \[
  (x_1,y_1) \sim (x_2,y_2) \, :\Leftrightarrow \, \text{there exists an element $a \in M$ such that } x_1 + y_2 + a = x_2 + y_1 + a.
  \]
  Then the quotient set $(M \times M)/ \sim$, together with a map $M \to (M \times M)/\sim$ given by $x \mapsto (x,0)$, is a group completion of $M$ .
\end{proposition}

The  cancellation property is related to the group completion as follows.
\begin{proposition}
  Let $M$ be a monoid. Then it is cancellative if and only if the natural map $\iota \colon M \to \gp M$ is an injection if and only if there is an injective monoid homomorphism into some group. In this case, every injective monoid homomorphism $\varphi \colon M \to G$ such that $\varphi(M)$ generates $G$ is automatically a group completion of $M$.
\end{proposition}
It follows that the image of the group completion is always cancellative. One of the difficulties of dealing with a monoid $M$ is that $M$ may not be cancellative. Thus, this image is much easier to deal with than $M$, and it has more information on its group completion $\gp M$. This corresponds to the \emph{positive part} of the Grothendieck group, see Section \ref{pospartsubsec}.
\begin{definition}
  Let $M$ be a monoid. We denote by $M_\can$ the image of the group completion $\iota \colon M \to \gp M$, and call it the \emph{cancellative quotient} of $M$.
\end{definition}
We leave it to the reader to check $M_\can$ is actually the largest cancellative quotient of $M$:
\begin{proposition}\label{canquot}
  Let $M$ be a monoid and define an equivalence relation $\sim_\can$ on $M$ by
  \[
  x \sim_\can y \,:\Leftrightarrow\, \text{there exists an element $a \in M$ such that } x + a = y + a.
  \]
  Then $\sim_\can$ is a monoid congruence on $M$, and we have an isomorphism of monoids $M_\can \iso M/ \sim_\can$.
  Consequently, for every monoid homomorphism $\varphi \colon M \to N$ to a cancellative monoid $N$, there exists a unique monoid homomorphism $\ov{\varphi} \colon M_\can \to N$ such that $\varphi = \ov{\varphi}\iota$.
\end{proposition}

\subsection{Length-like functions on monoids}
We will introduce a kind of \emph{length} on monoids, which corresponds to length-like functions on exact categories introduced in Definition \ref{lengthlikedef}.
\begin{definition}\label{llmonoid}
  Let $M$ be a monoid. A \emph{length-like function} on $M$ is a monoid homomorphism $\nu \colon M \to \N$ such that $\nu(x) = 0$ implies that $x = 0$.
\end{definition}
The existence of a length-like function implies some nice properties.
\begin{proposition}\label{llprop}
  Let $M$ be a monoid and suppose that there exists a length-like function $\nu \colon M \to \N$ on $M$. Then the following hold.
  \begin{enumerate}
    \item $M$ is naturally partially ordered, hence reduced.
    \item $M$ is atomic.
  \end{enumerate}
\end{proposition}
\begin{proof}
  (1)
  Suppose that $x \leq y \leq x$ holds for $x,y \in M$. Since $\nu$ is a monoid homomorphism, it follows immediately that $\nu(x) \leq \nu(y) \leq \nu(x)$ holds, thus $\nu(x) = \nu(y)$. On the other hand, we have $y = x + a$ for some $a \in M$ by $x \leq y$. Thus, $\nu(y) = \nu(x) + \nu(a)$, which implies that $\nu(a)=0$. Since $\nu$ is a length-like function, $a = 0$ holds, hence $x = y$.
  Thus, $M$ is naturally partially ordered, so it is reduced by Proposition \ref{natpo}.

  (2)
  Suppose that $M$ is not atomic. Then take an element $x$ such that:
  \begin{enumerate}[label={\upshape(\alph*)}]
    \item $x \neq 0$.
    \item $x$ cannot be expressed as a finite sum of atoms.
    \item $\nu(x)$ is minimal among those $x$ which satisfy (a) and (b).
  \end{enumerate}
  Obviously $x$ is not an atom by (b). Thus, there is a decomposition $x = y + z$ with $y, z \neq 0$. Clearly either $y$ or $z$ satisfies (b), so let us assume that $y$ satisfies (b). However, we have $\nu(x) = \nu(y) + \nu(z)$ with $\nu(y), \nu(z) \neq 0$ since $\nu$ is a length-like function. Therefore, $\nu(y) < \nu(x)$ and $y$ satisfies (a) and (b). This contradicts the minimality of $x$.
\end{proof}

For later use, we show the following characterization of half-factorial monoids.
\begin{lemma}\label{freell}
  Let $M$ be a monoid. Then the following are equivalent:
  \begin{enumerate}
    \item $M$ is a half-factorial monoid.
    \item $M$ has a length-like function $\nu$ satisfying $\nu(a) = 1$ for every $a \in \Atom M$.
  \end{enumerate}
\end{lemma}
\begin{proof}
  (1) $\Rightarrow$ (2)
  Suppose that $M$ is half-factorial. We define a monoid homomorphism $l \colon M \to \N$ as follows:
  For every $x$ in $\EE$, we can write $x = \sum_{i=1}^l a_i$ in $\MM(\EE)$ with $a_i \in \Atom M$ for each $i$ since $M$ is atomic. We set $l(x) := l$. Since $M$ is half-factorial, $l$ does not depend on the choice of expressions, thus this map is well-defined. Furthermore, it is easy to see that $l$ is a length-like function and that $l(a) = 1$ for every $a \in \Atom M$.

  (2) $\Rightarrow$ (1):
  First, observe that $M$ is reduced and atomic by Proposition \ref{llprop}.
  Let $x$ be an element of $M$. Consider any expression
  \[
  x = a_1 + \cdots + a_n
  \]
  with $a_i \in \Atom M$. Then by (2), we have that $\nu(x) = \nu(a_1) + \cdots + \nu(a_n) = n$.
  Therefore, the number $n$ depends only on $x$, so $M$ is half-factorial.
\end{proof}

\subsection{Characterizations of free monoids}
In what follows, we give a criterion for a given monoid to be free. We use this results to check (JHP) in Section 4.

If a monoid $M$ is free, then it has a length-like function (Lemma \ref{freell}), its group completion is also free, and its rank coincides with the number of atoms. In general, we have the following inequality.
\begin{proposition}\label{ineq}
  Let $M$ be a reduced atomic monoid and suppose that $\gp M$ is a free abelian group. Then the following inequality holds.
  \[
  \rank (\gp M) \leq \# \Atom M
  \]
\end{proposition}
\begin{proof}
  Let $\iota \colon M \to \gp M$ denote the group completion. As an abelian group, $\gp M$ is generated by $\iota M$, so it is generated by $\iota(\Atom M)$. Thus, $\rank (\gp M) \leq \# \Atom M$ holds.
\end{proof}

The following gives a kind of converse to this. This is an important characterization of free monoids, which is very useful to our setting.

\begin{theorem}\label{freechar0}
  Let $M$ be a monoid, and denote by $\iota \colon M \to \gp M$ the group completion of $M$. Then the following are equivalent:
  \begin{enumerate}
    \item $M$ is a free monoid.
    \item $M$ has a length-like function, $\iota$ is injective on $\Atom M$, and $\gp M$ is a free abelian group with basis $\iota (\Atom M)$.
    \item $M$ is reduced atomic, $\iota$ is injective on $\Atom M$, and $\gp M$ is a free abelian group with basis $\iota (\Atom M)$.
    \item $M$ is reduced atomic, $\iota$ is injective on $\Atom M$, and all elements in $\iota (\Atom M)$ are linearly independent over $\Z$.
  \end{enumerate}
\end{theorem}
\begin{proof}
  (1) $\Rightarrow$ (2):
  $M$ has a length-like function by Proposition \ref{freell}.
  Since $M$ is free, it is free on $\Atom M$ by Proposition \ref{factprop} (3), and we have an isomorphism $M \iso \N^{(\Atom M)}$. Thus, we have $\gp M \iso \Z^{(\Atom M)}$ by Example \ref{freecompl}. Thus, (2) follows.

  (2) $\Rightarrow$ (3):
  This follows from Proposition \ref{llprop}.

  (3) $\Rightarrow$ (4):
  This is trivial.

  (4) $\Rightarrow$ (1):
  Define a monoid homomorphism $\varphi \colon \N^{(\Atom M)} \to M$ by $\varphi(a) = a$ for each $a \in \Atom M$. We claim that this map is an isomorphism of monoids. It suffices to show that $\varphi$ is a bijection.

  Since $M$ is reduced and atomic, every element of $M$ is a finite sum of atoms. Thus, $\varphi$ is a surjection. On the other hand, consider the following commutative diagram of monoids
  \[
  \begin{tikzcd}
    \N^{(\Atom M)} \rar["\varphi", twoheadrightarrow] \dar[hookrightarrow, "i"'] & M \dar["\iota"] \\
    \Z^{(\Atom M)} \rar["\ov{\varphi}"] & \gp M,
  \end{tikzcd}
  \]
  where $\iota$ and $i$ are group completions of $M$ and $\N^{(\Atom M)}$ respectively, and $\ov{\varphi}$ is a group homomorphism induced by $\varphi$. Here $i$ is obviously an injection.
  Moreover, since $\iota$ is injective on $\Atom M$ and all elements in $\iota(\Atom M)$ are linearly independent, $\ov{\varphi}$ is an injection. Then the above commutative diagram shows that so is $\varphi$. Therefore, $\varphi$ is a bijection.
\end{proof}

Under the assumption of finite generation, we have a more convenient characterization, in which we only have to count the number of atoms.
\begin{corollary}\label{freechar}
  Let $M$ be a monoid, and denote by $\iota \colon M \to \gp M$ the group completion of $M$. Then the following are equivalent:
  \begin{enumerate}
    \item $M$ is finitely generated and free as a monoid.
    \item $M$ is a free monoid and $\# \Atom M$ is finite.
    \item $M$ is reduced and atomic, $\iota$ is injective on $\Atom M$ and $\gp M$ is a free abelian group of finite rank with basis $\iota (\Atom M)$.
    \item The following hold:
    \begin{enumerate}
      \item $M$ is reduced and atomic.
      \item $\gp M$ is a free abelian group of finite rank.
      \item $\# \Atom M = \rank (\gp M)$ holds, where $\rank$ denotes a rank as an abelian group.
    \end{enumerate}
  \end{enumerate}
\end{corollary}
\begin{proof}
  (1) $\Rightarrow$ (2):
  This follows from Proposition \ref{factprop} (1).

  (2) $\Rightarrow$ (3):
  This follows from (1) $\Rightarrow$ (3) in Theorem \ref{freechar0}.

  (3) $\Rightarrow$ (4): Trivial.

  (4) $\Rightarrow$ (1):
  We use the same notations and strategy as in the proof of (4) $\Rightarrow$ (1) in Theorem \ref{freechar0}. It suffices to show that the natural map $\ov{\varphi} \colon \Z^{\Atom M} \to \gp M$ is an isomorphism.
  Since $\varphi$ is a surjection, so is $\ov{\varphi}$. On the other hand, we have that $\gp M$ is free of rank $\# \Atom M$. By counting ranks, it follows that $\ov{\varphi}$ is an isomorphism.
\end{proof}

\medskip\noindent
{\bf Acknowledgement.}
The author would like to thank his supervisor Osamu Iyama for many helpful comments and discussions.
He is grateful to Jacob Greenstein for pointing out the use of Grothendieck monoids and explaining related results in \cite{bg}.
He would also like to thank the anonymous referee for many valuable comments.
This work is supported by JSPS KAKENHI Grant Number JP21J00299.

\end{appendix}


\begin{thebibliography}{199}
  \bibitem[AIR]{air}
  T. Adachi, O. Iyama, I. Reiten,
  \emph{$\tau$-tilting theory},
  Compos. Math. 150 (2014), no. 3, 415--452.

  \bibitem[AIRT]{airt}
  C. Amiot, O. Iyama, I. Reiten, G. Todorov,
  \emph{Preprojective algebras and c-sortable words},
  Proc. Lond. Math. Soc. (3) 104 (2012), no. 3, 513--539.

  \bibitem[Asa]{asai}
  S. Asai,
  \emph{Semibricks},
  Int. Math. Res. Not. IMRN 2020, no. 16, 4993--5054.

  \bibitem[ASS]{ASS}
  I. Assem, D. Simson, A. Skowro\'nski,
  \emph{Elements of the representation theory of associative algebras Vol. 1}, London Mathematical Society Student Texts, 65,
Cambridge University Press, Cambridge, 2006.

  \bibitem[AR3]{applications}
  M. Auslander, I. Reiten, \emph{Applications of contravariantly finite subcategories}, Adv. Math. 86 (1991), no. 1, 111--152.

  \bibitem[ARS]{ARS}
  M. Auslander, I. Reiten, S.O. Smal\o,
  \emph{Representation theory of Artin algebras},
  Cambridge Studies in Advanced Mathematics, 36. Cambridge University Press, Cambridge, 1995.

  \bibitem[BaGe]{baeg}
  N.R. Baeth, A. Geroldinger,
  \emph{Monoids of modules and arithmetic of direct-sum decompositions},
  Pacific J. Math. 271 (2014), no. 2, 257--319.

  \bibitem[BeGr]{bg}
  A. Berenstein, J. Greenstein,
  \emph{Primitively generated Hall algebras},
  Pacific J. Math. 281 (2016), no. 2, 287--331.

  \bibitem[BB]{bb}
  A. Bj\"orner, F. Brenti,
  \emph{Combinatorics of Coxeter groups},
  Graduate Texts in Mathematics, 231. Springer, New York, 2005.

  \bibitem[Bro1]{bro1}
  G. Brookfield,
  \emph{Direct sum cancellation of Noetherian modules},
  J. Algebra 200 (1998), no. 1, 207--224.

  \bibitem[Bro2]{bro2}
  G. Brookfield,
  \emph{The Grothendieck group and the extensional structure of Noetherian module categories},
  Algebra and its applications (Athens, OH, 1999), 111--131,
  Contemp. Math., 259, Amer. Math. Soc., Providence, RI, 2000.

  \bibitem[Bro3]{bro3}
  G. Brookfield,
  \emph{The extensional structure of commutative Noetherian rings},
  Comm. Algebra 31 (2003), no. 6, 2543--2571.

  \bibitem[BrGu]{brg}
  W. Bruns, J. Gubeladze,
  \emph{Polytopes, rings, and K-theory},
  Springer Monographs in Mathematics. Springer, Dordrecht, 2009.

  \bibitem[BHLR]{bhlr}
  T. Br\"ustle, S. Hassoun, D. Langford, S. Roy,
  \emph{Reduction of exact structures},
  J. Pure Appl. Algebra 224 (2020), no. 4, 106212, 29 pp.

  \bibitem[BHT]{bht}
  T. Br\"ustle, S. Hassoun, A. Tattar,
  \emph{Intersections, sums, and the Jordan-Hölder property for exact categories},
  J. Pure Appl. Algebra 225 (2021), no. 11, 106724, 35 pp.

  \bibitem[B\"uh]{buhler}
  T. B\"uhler, \emph{Exact categories}, Expo. Math. 28 (2010), no. 1, 1--69.

  \bibitem[Eno1]{en1}
  H. Enomoto, \emph{Classifying exact categories via Wakamatsu tilting}, J. Algebra 485 (2017), 1--44.

  \bibitem[Eno2]{en2}
  H. Enomoto, \emph{Classifications of exact structures and Cohen-Macaulay-finite algebras}, Adv. Math. 335 (2018), 838--877.

  \bibitem[Eno3]{enforth}
  H. Enomoto, \emph{Bruhat inversions in Weyl groups and torsion-free classes over preprojective algebras}, Comm. Algebra 49 (2021), no. 5, 2156--2189.

  \bibitem[Fac1]{fac1}
  A. Facchini,
  \emph{Direct sum decompositions of modules, semilocal endomorphism rings, and Krull monoids},
  J. Algebra 256 (2002), no. 1, 280--307.

  \bibitem[Fac2]{fac2}
  A. Facchini,
  \emph{Krull monoids and their application in module theory},
  Algebras, rings and their representations, 53--71, World Sci. Publ., Hackensack, NJ, 2006.

  \bibitem[Gab]{gabriel}
  P. Gabriel,
  \emph{Indecomposable representations. II},
  Symposia Mathematica, Vol. XI (Convegno di Algebra Commutativa, INDAM, Rome, 1971), pp. 81--104. Academic Press, London, 1973.

  \bibitem[GHK]{ghk}
  A. Geroldinger, F. Halter-Koch,
  \emph{Non-unique factorizations: Algebraic, combinatorial and analytic theory},
  Pure and Applied Mathematics (Boca Raton), 278. Chapman \& Hall/CRC, Boca Raton, FL, 2006.

  \bibitem[Gri1]{gri1}
  P. A. Grillet, \emph{Semigroups: An introduction to the structure theory}, Monographs and Textbooks in Pure and Applied Mathematics, 193. Marcel Dekker, Inc., New York, 1995.

  \bibitem[Gri2]{gri2}
  P. A. Grillet, \emph{Commutative semigroups},
  Advances in Mathematics (Dordrecht), 2. Kluwer Academic Publishers, Dordrecht, 2001.


  \bibitem[HR]{hr}
  S. Hassoun, S. Roy, \emph{Admissible intersection and sum property}, arXiv:1906.03246.

  \bibitem[IT]{it}
  C. Ingalls, H. Thomas,
  \emph{Noncrossing partitions and representations of quivers},
  Compos. Math. 145 (2009), no. 6, 1533--1562.

  \bibitem[Iya]{iy}
  O. Iyama,
  \emph{The relationship between homological properties and representation theoretic realization of Artin algebras},
  Trans. Amer. Math. Soc. 357 (2005), no. 2, 709--734.

  \bibitem[MS]{ms}
  F. Marks, J. \v{S}t\!'ov\'{i}\v{c}ek,
  \emph{Torsion classes, wide subcategories and localisations},
  Bull. Lond. Math. Soc. 49 (2017), no. 3, 405--416.

  \bibitem[Mit]{mitchel}
  B. Mitchell,
  \emph{Theory of categories},
  Pure and Applied Mathematics, Vol. XVII Academic Press, New York-London 1965 xi+273 pp.

  \bibitem[Miy]{miyashita}
  Y. Miyashita,
  \emph{Tilting modules of finite projective dimension},
  Math. Z. 193 (1986), no. 1, 113--146.

  \bibitem[PR]{pr}
  M.I. Platzeck, I. Reiten,
  \emph{Modules of finite projective dimension for standardly stratified algebras},
  Comm. Algebra 29 (2001), no. 3, 973--986.

  \bibitem[Tho]{thomas}
  H. Thomas, \emph{Coxeter groups and quiver representations}, In Surveys in representation theory of algebras, volume 716 of Contemp. Math., pages 173--186. Amer. Math. Soc., Providence, RI, 2018.

  \bibitem[Qui]{qu}
  D. Quillen, \emph{Higher algebraic K-theory: I},
  Lecture Notes in Math., Vol. 341, Springer, Berlin 1973.

  \bibitem[Ra\u{\i}]{raikov}
  D.A. Ra\u{\i}kov,
  \emph{Semiabelian categories},
  Dokl. Akad. Nauk SSSR 188 1969 1006--1009.

  \bibitem[Rea]{reading}
  N. Reading,
  \emph{Clusters, Coxeter-sortable elements and noncrossing partitions},
  Trans. Amer. Math. Soc. 359 (2007), no. 12, 5931--5958.

  \bibitem[Rum1]{rumpstar}
  W. Rump,
  \emph{$\ast$-modules, tilting, and almost abelian categories},
  Comm. Algebra 29 (2001), no. 8, 3293--3325.

  \bibitem[Rum2]{rump}
  W. Rump,
  \emph{Almost abelian categories},
  Cahiers Topologie G\'eom. Diff\'erentielle Cat\'eg. 42 (2001), no. 3, 163--225.

  \bibitem[Rum3]{rumpraikov}
  W. Rump,
  \emph{A counterexample to Raikov's conjecture},
  Bull. Lond. Math. Soc. 40 (2008), no. 6, 985--994.

  \bibitem[St]{st}
  B. Stenstr\"om, \emph{Rings of quotients: An introduction to methods of ring theory}, Springer-Verlag, NewYork-Heidelberg, 1975.

  \bibitem[Yos]{yo}
  Y. Yoshino, \emph{Cohen-Macaulay modules over Cohen-Macaulay rings},
  London Mathematical Society Lecture Note Series 146, Cambridge University Press, Cambridge, 1990.

\end{thebibliography}
\end{document}